\documentclass[11pt, reqno]{amsart}
\usepackage[utf8x]{inputenc}
\usepackage{amsmath}
\usepackage{mathrsfs}
\usepackage{amstext}
\usepackage{amsxtra}
\usepackage{amssymb}
\usepackage{amsgen}
\usepackage{amsfonts}
\usepackage{amsthm}
\usepackage{amscd}
\usepackage{latexsym}
\usepackage[margin=0.6in]{geometry}
\usepackage{cite}
\usepackage{setspace}
\usepackage{bbm}
\usepackage{cancel}

\newtheorem{theorem}{Theorem}[section]
\newtheorem{lemma}{Lemma}[section]
\newtheorem{proposition}{Proposition}[section]
\newtheorem{corollary}{Corollary}[section]
\newtheorem{assumption}{A Priori Assumption}

\theoremstyle{definition}
\newtheorem*{definition}{Definition}

\theoremstyle{remark}
\newtheorem{remark}{Remark}[section]

\newcommand{\oht}{\mathfrak{H}}

\newcommand{\nht}{\mathcal{H}}

\newcommand{\pv}{\,\text{p.v.}}
\newcommand{\sgn}{\,\text{sgn}}
\newcommand{\ivec}{\mathbbm{i}}
\newcommand{\jvec}{\mathbbm{j}}
\newcommand{\kvec}{\mathbbm{k}}
\newcommand{\xnew}{\mathfrak{x}}
\newcommand{\ynew}{\mathfrak{y}}
\newcommand{\znew}{\mathfrak{z}}
\newcommand{\vecpart}{\mathfrak{V}}

\numberwithin{equation}{section}

\title{A Justification of the Modulation Approximation to the 3D Full Water Wave Problem}
\author{Nathan Totz}
\address{Department of Mathematics \\
Duke University \\
Durham, NC, 27708}
\thanks{The author would like to thank Sijue Wu for her discussions on her formulation of the 3D water wave problem, as well as her helpful comments and suggestions on the draft of this paper.  The author was supported in part by NSF grant DMS-0800194.}

\begin{document}

\maketitle

\begin{abstract}
We consider modulational solutions to the 3D inviscid incompressible irrotational infinite depth water wave problem neglecting surface tension.  For such solutions, it is well known that one formally expects the modulation to be a profile traveling at group velocity and governed by a 2D hyperbolic cubic nonlinear Schr\"odinger equation.  In this paper we justify this fact by providing rigorous error estimates in Sobolev spaces.  We reproduce the multiscale calculation to derive an approximate wave packet-like solution to the evolution equations with mild quadratic nonlinearities constructed by Sijue Wu.  Then we use the energy method along with the method of normal forms to provide suitable a priori bounds on the difference between the true and approximate solutions.
\end{abstract}

\section{Introduction}

The three dimensional water wave problem concerns the motion of an interface separating a region of zero density (e.g., air) from an inviscid, incompressible, irrotational fluid of uniform density that is under the influence of gravity.  We assume that the fluid region is below the air region, that the fluid region is of infinite depth, that the interface approaches a horizontal plane at infinity and that the velocity and acceleration tend to zero at spatial infinity.  Denote by $\kvec = \langle 0, 0, 1\rangle$ the upward vertical unit vector, $\Omega(t)$ the fluid region at time $t \geq 0$ and $\Sigma(t)$ the interface at time $t \geq 0$.  Then if surface tension is neglected, the motion of the fluid is described by
\begin{align}\label{EulerVector}
\mathbf{v}_t + (\mathbf{v} \cdot \nabla)\mathbf{v} = -\kvec - \nabla \mathbf{p} & \qquad \text{on }\Omega(t), \, t \geq 0 \notag \\
\nabla \cdot \mathbf{v} = 0, \quad \nabla \times \mathbf{v} = 0 & \qquad \text{on }\Omega(t), \, t \geq 0 \\
\mathbf{p} = 0 & \qquad \text{on }\Sigma(t), \, t \geq 0 \notag \\
(1, \mathbf{v})\text{ is tangent to }(t, \Sigma(t)). & \qquad \notag
\end{align}
where $\mathbf{v}$ is the fluid velocity and $\mathbf{p}$ is the fluid pressure.  In \cite{WuLocal3D}, the evolution of this system was shown to be equivalent to a system for the motion of the evolution of the free surface $\Sigma(t)$.  Specifically, if one parametrizes $\Sigma(t)$ by $\Xi(\alpha, \beta, t) \in \mathbb{R}^3$ with Lagrangian coordinates $\alpha, \beta$, the main evolution equation for the interface take the compact form
\begin{equation}\label{WaterWaveLagrangeFirstPass}
\Xi_{tt} + \kvec = \mathfrak{a}(\Xi_\alpha \times\Xi_\beta)
\end{equation}
\begin{equation}\label{XiTAnalyticFirstPass}
(I - \oht)\Xi_t = 0
\end{equation}
where $\mathfrak{a}(\alpha, \beta, t) = -\frac{1}{|\Xi_\alpha \times \Xi_\beta|}\frac{\partial p}{\partial \mathbf{n}}$ and $\mathbf{n}$ is the outward-pointing unit normal of $\Sigma(t)$.  Just as in the 2D problem, \eqref{XiTAnalyticFirstPass} is a condition defined entirely on $\Sigma(t)$ that is equivalent to the fact that the fluid is incompressible and irrotational in $\Omega(t)$.  The operator $\oht$ is called the Hilbert transform associated to $\Xi$; it serves the same purpose and has many of the same properties as the ordinary Hilbert transform associated to a curve in the complex plane, except that it requires Clifford analysis to define (see Proposition \ref{PotentialTheoryFacts}  for a precise definition).

Our goal is to study a special class of solutions to \eqref{WaterWaveLagrangeFirstPass}-\eqref{XiTAnalyticFirstPass} that are close to a wave packet propagating in the $\ivec = \langle 1, 0, 0\rangle$ direction with a special scaling, i.e., a solution that when written in coordinates is of the form
\begin{equation}\label{WavePacketVector}
\Xi(\alpha, \beta, t) \sim \left\langle \alpha, \beta, 0 \right\rangle + \epsilon \left\langle \Re(Ae^{i(k\alpha + \omega t)}), 0, \Im(Ae^{i(k\alpha + \omega t)})\right\rangle + o(\epsilon),
\end{equation} where $k > 0$ is the fixed wave number of the wave packet, $\omega$ is the wave frequency related to $k$ through the dispersion relation $\omega^2 = k$, and $A = A(\epsilon\alpha, \epsilon\beta, \epsilon t, \epsilon^2 t)$ is complex-valued.  The special scaling relationships in the approximate solution are natural to study when applying the methods of multiscale analysis.

Many authors (e.g., \cite{ZakharovInfiniteDepth}, \cite{AblowitzSegur}, \cite{CraigSchantzSulem}) have sought an approximate solution to the water wave problem of this form along with higher order correctors chosen so that when the approximate solution is substituted into the water wave equations the residual terms are physically of size $O(\epsilon^4)$.  If, in the case of infinite depth, one performs this process on the Euler equations, one finds that the amplitude $A = A(\epsilon(\alpha + \omega^\prime t), \epsilon\beta, \epsilon^2t) := A(X, Y, T)$ of the wave packet is a traveling-wave profile that travels at group velocity $\omega^\prime(k)$ for times on the order $O(\epsilon^{-1})$, and $A = A(X, Y, T)$ satisfies the so-called ``hyperbolic'' cubic nonlinear Schr\"odinger equation (HNLS) for times on the order $O(\epsilon^{-2})$:
\begin{equation}\label{HNLSVector}
i A_T + a A_{XX} - b A_{YY} + c A|A|^2 = 0
\end{equation}
where $a, b, c$ are positive constants depending on $k$ and $\omega$.  However, this formal calculation assumes that a solution can be developed in an asymptotic series in $\epsilon$, a fact that needs justification.  This sort of justification should not be taken for granted: there are examples of modulation approximations derived by seemingly reasonable formal arguments which do not give the correct dynamics (c.f. \cite{SchneiderValidityNW}, \cite{GallaySchneiderKP}).  Notice that, because of the slow time dependence of $A$, the HNLS dynamics are not apparent in the full solution unless we consider solutions on time scales of the order $O(\epsilon^{-2})$.  Typically, a rigorous justification of this approximation to the water wave problem would entail showing the following steps:
\begin{itemize}
\item[(i)]{The HNLS equation \eqref{HNLSVector} is locally well-posed in a suitable function space.}
\item[(ii)]{An approximate solution $\tilde{\Xi}$ of the form \eqref{WavePacketVector} can be found which formally satisfies the equation for $\Xi$ up to residual terms of physical size at most $o(\epsilon^3)$.}
\item[(iii)]{The system \eqref{EulerVector} is well posed on a space containing the approximate solutions $\tilde{\Xi}$, and solutions $\Xi$ initially close to wave packet-like solutions exist for times on the order $O(\epsilon^{-2})$.}
\item[(iv)]{The remainder $\Xi - \tilde{\Xi}$ is of size at most $o(\epsilon)$ in a suitable function space.}
\end{itemize}

Step (i) can be shown to hold in a variety of function spaces; standard results are collected in \cite{CazenaveSemilinearSchrodingerEquations}.  In this vein, we mention the global well-posedness results of Ghidaglia and Saut for small data in \cite{GhidagliaSautDaveyStewartson}.  Still, for large data, little is known about the well-posedness of HNLS beyond local well-posedness in $H^s$ for $s \geq 2$.  As mentioned earlier, many authors have performed Step (ii), and we mention in particular the more rigorous work \cite{CraigSchantzSulem} that gives suitable estimates of the residual in $L^q(\mathbb{R}^2)$ Sobolev spaces for $2 < q < \infty$ in the more general finite depth case.

Step (iii) and Step (iv) have not been performed for the 3D problem to date; the purpose of this paper to perform them along with appropriate versions of Steps (i) and (ii).  As in the 2D problem, the main difficulty in completing this part of the program is showing the existence of wave packet-like solutions to the water wave problem on $O(\epsilon^{-2})$ time scales.  Since the $L^2$ norm of a wave packet is even larger in 3D than in 2D, this difficulty is correspondingly magnified.  Indeed, the $L^2$ norm of such wave packets are $O(1)$ in $L^2$ and so do not even vanish as $\epsilon \to 0$.  In \cite{TotzWu2DNLS}, the difficulties were resolved in the 2D setting by finding a formulation of the water wave equations having no quadratic nonlinearities.  As recognized in \cite{KSMCubicNonlinearityLongtimeRemainder}, justifying modulation approximations is made much simpler for such equations.  This was accomplished by means of a fully nonlinear change of variables.

In \cite{WuGlobal3D}, Wu developed an analogue of this change of variables in 3D (denoted in this paper by $\kappa$) depending on the unknown $\Xi$ and used it along with the method of invariant vector fields to prove the global well-posedness of the system \eqref{WaterWaveLagrangeFirstPass}-\eqref{XiTAnalyticFirstPass} by constructing a system of equations in $\Xi \circ \kappa^{-1} =: \zeta = \xnew\ivec + \ynew\jvec + \znew\kvec$ and $\chi := (I - \nht)\znew\kvec$ which is equivalent to the system \eqref{WaterWaveLagrangeFirstPass}-\eqref{XiTAnalyticFirstPass} and is of the form
\begin{equation}\label{NewWaterWaveFirstPass}
(\partial_t + (\kappa_t \circ \kappa^{-1} \cdot \nabla))^2\chi - \zeta_\beta \times \chi_\alpha + \zeta_\alpha \times \chi_\beta = G
\end{equation}
\begin{equation}\label{ZetaTAnalyticFirstPass}
(I - \nht)(\partial_t + (\kappa_t \circ \kappa^{-1} \cdot \nabla))\zeta = 0
\end{equation} where $G$ consists of terms of third and higher order terms, and $\nht$ now denotes the Hilbert transform associated to $\zeta$.  Note that we have abused notation slightly here by reusing $\alpha, \beta$ to denote the independent variables of the transformed problem.

The outline of our strategy is the same as that used in \cite{TotzWu2DNLS}: rather than find a formal approximation of $\Xi$ and justify it directly, we instead find a formal approximation $\tilde{\zeta}$ of the form \eqref{WavePacketVector} for $\zeta$, and then use energy estimates to construct a priori bounds for the remainder $\zeta - \tilde{\zeta}$.  We find in the course of the calculation that this formal approximation $\tilde{\zeta}$ is only in the usual $L^2$-Sobolev space provided we take $A$ in an $L^2$-Sobolev space with some mild decay.  We therefore show that \eqref{HNLSVector} is well-posed in this space, completing Step (i).  Then, after showing that we have suitable control over the change of variables $\kappa$, one can use the a priori estimates of the transformed remainder along with the $O(\epsilon^{-2})$ existence of the approximate solution $\tilde{\zeta}$ to give a priori bounds of the original solution $\Xi$ for $O(\epsilon^{-2})$ times.  Step (iii) then follows by a bootstrapping argument, and so Step (iv) follows immediately as well.

There are two main difficulties that arise using this approach in the 3D problem that were not present in the 2D problem.  First, there were no quadratic terms in the nonlinearity of the governing equations in the 2D problem, whereas in 3D ``null form'' quadratic terms of the form $f_\beta g_\alpha - f_\alpha g_\beta$ appear in the equations for the derivatives of $\zeta$.  In the paper \cite{WuGlobal3D}, these terms are controlled using Klainerman-Sobolev norms constructed from the invariant vector fields associated to the water wave equations.  However, the wave packet-like solutions are large with respect to these norms since such solutions do not possess the symmetries associated to the invariant vector fields.  Therefore we cannot use Klainerman-Sobolev norms to gain effective control of the null form terms here.

The second difficulty arises from the slow spatial scaling of the modulation, namely that the residual of the approximate solution in 3D is too large.  If a function $S(X, Y)$ is of size $O(1)$ in $L^2_{X, Y}$, then the function $S(\epsilon\alpha, \epsilon \beta)$ is of size $O(\epsilon^{-1})$ in $L^2_{\alpha, \beta}$.  Thus if one constructed an approximate solution with residual of physical size $O(\epsilon^4)$, then the residual would merely have size $O(\epsilon^3)$ in $L^2(\mathbb{R}^2)$.  One can then read off from the energy inequality that we could justify at most that $\zeta - \tilde{\zeta} = O(\epsilon)$ in Sobolev space.  Besides failing to satisfy (iv), this weak control of the remainder would present the more serious obstacle that we could not even guarantee that $\kappa$ is invertible for $O(\epsilon^{-2})$ times!  (See Remark \ref{BigHorizontalError}.)  The aforementioned work \cite{CraigSchantzSulem} does not circumvent this difficulty in our setting, since it gives $o(\epsilon^3)$ estimates of the residual in $L^q$ for $2 < q < \infty$, and our energy method requires us to have $o(\epsilon^3)$ bounds on the residual in $L^2$.

These difficulties are resolved by refining the methods used in \cite{TotzWu2DNLS} in two respects.  To resolve the latter problem above, we attempt to develop the approximate solution to a higher order.  In doing so, we find terms in the higher order correctors which are not in $L^2(\mathbb{R}^2)$ unless mild decay restrictions are placed on low derivatives of $A$.  Moreover, we find that there are terms of physical size $O(\epsilon^4)$ that cannot be accounted for by choosing appropriate correctors for the approximate solution; that is, approximating the system \eqref{NewWaterWaveFirstPass}-\eqref{ZetaTAnalyticFirstPass} with solutions of the leading term \eqref{WavePacketVector} is only formally consistent to terms of the order $O(\epsilon^4)$ (see Remark \ref{ResidualFormallyEpsilon4}).  Despite this, all of these terms appear in the energy estimates in such a way so that they can be regarded as an order smaller than they first appear.  This allows us to regard the residual of the approximate solution as being of size $O(\epsilon^4)$ in $L^2$.

Second, in order to overcome the former difficulty of the quadratic null-form nonlinearities, we introduce third order corrections to the energy and use the method of normal forms.  This involves perturbing the remainder by a quadratic expression which is explicitly constructed in frequency space to cancel the quadratic contributions.  A fair question to ask is why one would not use this method directly from the outset to eliminate quadratic nonlinearities from the system \eqref{WaterWaveLagrangeFirstPass}-\eqref{XiTAnalyticFirstPass}.  Indeed, this approach has been used in the context of global well-posedness of this system by \cite{GermainMasmoudiShatah} as well as in modulation justification by \cite{SchneiderWayneNLSJustify}.  However, doing so forces one to work in more restrictive classes of solutions than is done here.\footnote{see the discussion in Appendix C of \cite{WuGlobal3D}.}  Wu's transform instead eliminates all quadratic nonlinearities except for those of null-form type, and the special structure of these nonlinearities avoids the singular behavior that necessitates restricting one's class of solutions (to see the cancellation explicitly see Lemma \ref{DenominatorEstimate}).

We also note that arbitrary wave packet-like initial data need not satisfy the compatibility conditions required by solutions to the system \eqref{NewWaterWaveFirstPass}-\eqref{ZetaTAnalyticFirstPass}, and so we show that one can always construct initial data satisfying these constraints that differs by at most $O(\epsilon^\frac52)$ from any given wave-packet like candidate for the initial data.

We now state the main result of this paper.  Let $\dot{H}^s(d)$ be the homogeneous Sobolev space of functions $f(x, y)$ for which the weak derivatives of $(1 + x^2 + y^2)^\frac{d}{2}f$ of order exactly $s$ are in $L^2$, let $\dot{H}^s = \dot{H}^s(0)$, let $H^s(d) = \dot{H}^0(d) \cap \dot{H}^s(d)$, and let $H^s = H^s(0)$.  The contributions of this paper in completing the steps (i)-(iv) above in the transformed unknown $\zeta$ are summarized in
\begin{theorem}\label{MainTheorem}
Let $k > 0$, $s \geq 9$, $\delta > 0$, and $A_0 \in H^{s + 13} \cap H^3(\delta)$ be given.  Then there exists a $\mathscr{T} > 0$ depending on $s$ and $\|A_0\|_{H^{s + 13} \cap H^3(\delta)}$ and a solution $A \in C([0, \mathscr{T}], H^{s + 13} \cap H^3(\delta))$ to the initial value problem \eqref{HNLSVector} for constants $a, b, c$ depending on $k$ with $A(0) = A_0$.  Moreover, for any $\mathscr{T} > 0$ such that the above holds, there exists an $\epsilon_0 > 0$ depending on $k$, $s$, $\mathscr{T}$, $\|A_0\|_{H^{s + 13} \cap H^3(\delta)}$, and $\delta$ so that the following hold:
\begin{itemize}
\item[(a)]{There is an approximate solution $\tilde{\zeta}$ of \eqref{NewWaterWaveFirstPass}-\eqref{ZetaTAnalyticFirstPass} having the form \eqref{WavePacketVector} in $C([0, \mathscr{T}\epsilon^{-2}], H^{s + 9})$ satisfying \eqref{NewWaterWaveFirstPass} up to a residual of size $O(\epsilon^4)$ when measured in $H^s$.}
\item[(b)]{There exists initial data $\Xi_0$ constructed using $A_0$ which satisfies the compatibility conditions of the system \eqref{WaterWaveLagrangeFirstPass}-\eqref{XiTAnalyticFirstPass}.  Moreover, if $\kappa_0$ is the change of variables constructed through $\Xi_0$ and $\zeta_0 = \Xi_0 \circ \kappa_0^{-1}$, then for $\eta = \frac12$ this data satisfies 
\begin{equation}\label{InitialDataBounds}
\| \, |\mathcal{D}|^\frac12(\zeta_0 - \tilde{\zeta}(0))\|_{H^{s + \frac12}} + \|\partial_t(\zeta_0 - \tilde{\zeta}(0))\|_{H^{s + \frac12}} + \|\partial_t^2(\zeta_0 - \tilde{\zeta}(0))\|_{H^s} \leq C\epsilon^{2 + \eta}
\end{equation} where $C$ depends on $k$, $s$, $\mathscr{T}$, $\|A_0\|_{H^{s + 13} \cap H^3(\delta)}$ and $\delta$.}
\item[(c)]{For any initial data of \eqref{WaterWaveLagrangeFirstPass}-\eqref{XiTAnalyticFirstPass} satisfying \eqref{InitialDataBounds} with $\eta = 0$, there is a solution $\Xi$ of \eqref{WaterWaveLagrangeFirstPass}-\eqref{XiTAnalyticFirstPass} satisfying $$\left(|\mathcal{D}|^\frac12\Bigl(\Xi - (\alpha, \beta, 0)\Bigr), \Xi_t, \Xi_{tt}\right) \in C\left([0, \mathscr{T}\epsilon^{-2}], H^{s + \frac12} \times H^{s + \frac12} \times H^s\right).$$
Moreover for all $0 \leq t \leq \mathscr{T} \epsilon^{-2}$ and $0 < \iota < 1$, $$\| \, |\mathcal{D}|^\frac12(\zeta(t) - \tilde{\zeta}(t))\|_{H^{s + \frac12}} + \|\Xi_t(t) - (\tilde{\zeta} \circ \kappa)_t(t))\|_{H^{s + \frac12}} + \|\Xi_{tt}(t) - \partial_t^2(\tilde{\zeta} \circ \kappa)(t))\|_{H^s} \leq C\epsilon^{2 - \iota}$$ where $C$ depends only on $k$, $s$, $\mathscr{T}$, $\|A_0\|_{H^{s + 13} \cap H^3(\delta)}$, $\delta$ and $\iota$.}
\end{itemize}
\end{theorem}

\begin{remark}
In fact the error in the stability estimate (c) can be improved to $\iota = 0$.  However, then the estimate is only valid on an interval of time $[0, \mathscr{T}^\prime\epsilon^{-2}]$ where in general $\mathscr{T}^\prime \leq \mathscr{T}$.
\end{remark}

In particular, this theorem implies that solutions to the water wave problem in Lagrangian coordinates with wave packet-like initial data exists for $O(\epsilon^{-2})$ times.  Unlike in \cite{TotzWu2DNLS}, half-derivative control of the remainder along with the embedding $L^\infty(\mathbb{R}^2) \supset \dot{H}^\frac12(\mathbb{R}^2) \cap \dot{H}^2(\mathbb{R}^2)$ is sufficient to give us immediate control on the $L^\infty$ norm of the solution and its derivatives in the transformed coordinates.  

\begin{corollary}\label{MainTheoremLInftyBounds}
Under the same hypotheses as Theorem \ref{MainTheorem}, we have furthermore that $$\|\zeta(t) - \tilde{\zeta}(t)\|_{W^{s - 1^+, \infty}} \leq C\epsilon^{2 - \iota}$$ for all $0 \leq t \leq \mathscr{T}\epsilon^{-2}$, where $C$ depends on the same quantities as in Theorem \ref{MainTheorem}.
\end{corollary}

There is still the question of how to derive justification of the asymptotics in more physically meaningful coordinates.  Theorem \ref{MainTheorem} justifies the modulation approximation for the velocity and acceleration fields of $\Xi$ in Lagrangian coordinates by straightforwardly changing variables by $\kappa$.  However, changing variables in the spatial derivatives contributes an error of the order $O(\epsilon)$ to the horizontal component of $\Xi$ and its derivatives (see again Remark \ref{BigHorizontalError}.)  Therefore, while we can justify the asymptotics for the vertical component $z$ of $\Xi$, we cannot do so for the horizontal component.

To rectify this, we give an Eulerian version of the justification.  Since $\zeta(t)$ parametrizes $\Sigma(t)$, Corollary \ref{MainTheoremLInftyBounds} guarantees that $\Sigma(t)$ is a graph for $O(\epsilon^{-2})$ times provided $\epsilon_0 > 0$ is chosen sufficiently small.  Parametrize this graph by $\Sigma(t) = \{ \langle \alpha, \beta, h(\alpha, \beta, t)\rangle : (\alpha, \beta) \in \mathbb{R}^2\}$.  If we decompose $\zeta = \tau + \znew\kvec$ into horizontal and vertical components, then $h = \znew \circ \tau^{-1}$, where $\tau^{-1}$ denotes the inverse of $\tau$ as a map $\mathbb{R}^2 \to \mathbb{R}^2$.  In the same way, the Eulerian velocity field is defined by $\mathfrak{v} = (\partial_t + (\kappa_t \circ \kappa^{-1}))\zeta \circ \tau^{-1}$.  We analogously set $\tilde{\zeta} = \tilde{\tau} + \tilde{\znew}\kvec$ and define the approximate Eulerian position $\tilde{h} = \tilde{\znew} \circ \tilde{\tau}^{-1}$, with the approximate Eulerian velocity $\tilde{\mathfrak{v}}$ defined similarly.  Then we can use Theorem \ref{MainTheorem} to obtain

\begin{theorem}\label{EulerianTheorem}
Let $k > 0$, $s \geq 9$, $\delta > 0$, $\iota > 0$ and $A_0, A, \mathscr{T}$ be as in the hypothesis of Theorem \ref{MainTheorem}, and let $h$, $\tilde{h}$ be constructed as above.  Then there exists an $\epsilon_0 > 0$ depending on $k$, $s$, $\mathscr{T}$, $\|A_0\|_{H^{s + 13} \cap H^3(\delta)}$, and $\delta$ so that for all $0 < \epsilon < \epsilon_0$, there exists initial data $h_0, \mathfrak{v}_0$ which for $\eta = \frac12$ satisfies:
$$\|\,|\mathcal{D}|^\frac12(h_0 - \tilde{h}(0))\|_{H^{s + \frac12}} + \|\mathfrak{v}_0 - \tilde{\mathfrak{v}}(0)\|_{H^{s + \frac12}} \leq C\epsilon^{2 + \eta}$$ and moreover for all such initial data satisfying this bound with $\eta = 0$, the quantity $h$ exists for times $[0, \mathscr{T}\epsilon^{-2}]$ and moreover for all $0 \leq t \leq \mathscr{T}\epsilon^{-2}$ satisfies
$$\|\,|\mathcal{D}|^\frac12(h(t) - \tilde{h}(t))\|_{H^{s + \frac12}} + \|\mathfrak{v}(t) - \tilde{\mathfrak{v}}(t)\|_{H^{s + \frac12}} \leq C\epsilon^{2 - \iota}$$ where $C$ depends on the same quantities as in Theorem \ref{MainTheorem}.
\end{theorem}

This paper is organized as follows: in Section 2 we review Clifford Analysis and introduce the main evolution equations, as well as relations between the associated quantities involved.  In Section 3 we formally expand the Hilbert Transform $\nht$, compute the correctors to the approximate solution \eqref{WavePacketVector}, as well as address some analytic issues brought about by the formula for the correctors.  In Section 4 we derive evolution equations for the remainders between the true and approximate solution and associated quantities, and use them to derive a priori bounds on the remainders via energy estimates.  In particular, we construct normal form transformations and third-order corrections to the energy to eliminate the quadratic terms that arise in the energy inequality.  Finally, since arbitrarily chosen wave packet-like initial data need not satisfy the compatibility conditions for the water wave system, we show in Section 5 how to construct such admissible data suitably close to given wave packet-like initial data, as well as use a bootstrapping argument to show existence of the water wave problem on $O(\epsilon^{-2})$ times.

\section{The Governing Equations}

	\subsection{Notation and Clifford-Algebraic Preliminaries}

For a 2-vector $(x, y)$, denote $|(x, y)| = \sqrt{x^2 + y^2}$ and $\langle(x, y)\rangle = \sqrt{1 + x^2 + y^2}$.  We also write $$|(x, y)|_\leq = \begin{cases} |(x, y)| & |(x, y)| \leq 1 \\ 0 & |(x, y)| \geq 1\end{cases}$$ along with $|(x, y)|_\geq = |(x, y)| - |(x, y)|_\leq$.  We write the Jacobian of a map $\phi : \mathbb{R}^2 \to \mathbb{R}^2$ by $J(\phi)$, and sometimes denote $f \circ \phi =: U_\phi f$.  The commutator is written as $[X, Y] = XY - YX$.

The algebra of quaternions $\mathbb{H}$ consist of a vector space spanned by the elements $1, \ivec, \jvec, \kvec$ along with a product called quaternion multiplication characterized by W. R. Hamilton's celebrated formulas $$\ivec^2 = \jvec^2 = \kvec^2 = \ivec\jvec\kvec = -1$$  Observe that the above laws imply that multiplication restricted to distinct units of $\ivec, \jvec, \kvec$ agrees with the usual cross product multiplication on $\mathbb{R}^3$.

For a quaternion $q = q_0 + q_1 \ivec + q_2 \jvec + q_3 \kvec$, we sometimes denote components of a large quaternion-valued expression by $q_i := \left\{q_0 + q_1 \ivec + q_2 \jvec + q_3 \kvec\right\}_i$.  Define $\overline{q} = q_0 - q_1 \ivec - q_2 \jvec - q_3 \kvec$, as well as the \textbf{scalar part} $\Re(q) = \frac{1}{2}(q + \overline{q}) = q_0$ and the \textbf{vector part} $\vecpart(q) = \frac12(q - \overline{q}) = \sum_{j = 1}^3 q_i e_i$.  

We call a quaternion \textbf{real-valued} if $\vecpart(f) = 0$, \textbf{vector-valued} if $\Re(f) = 0$, \textbf{$1, \jvec$-valued} if $f = f_0 + \jvec f_2$ for two real-valued functions $f_0, f_2$, etc.  Many of the quantities in this paper are $1, \jvec$-valued quaternions, and \textit{for these quaternions only} we define the scalar $\Im(a + b\jvec) = b$ in analogy with the complex numbers.  We define the inner product $p \cdot q = \sum_{i = 0}^3 p_i q_i$.  If $p$ and $q$ have no scalar parts, we identify $p$, $q$ with vectors in $\mathbb{R}^3$ and define their cross product $p \times q$ in the usual sense of $\mathbb{R}^3$.  In cases of ambiguous multiplication, the order of operations in this paper will be to first perform cross products followed by quaternion multiplications.  Finally, 
we denote by $q^\dagger = \kvec q \kvec$, which for vector-valued quaternions corresponds to reflection across the $\ivec, \jvec$-plane.\footnote{The reader is cautioned that in \cite{WuGlobal3D}, this operation $\cdot^\dagger$ is instead denoted using the overbar $\overline{\,\cdot\,}$, which we reserve instead for ordinary quaternion conjugation.}

For vector quantities $p, q$ we have $pq = p \times q - p \cdot q$.  A quaternion generalization of the ordinary scalar triple product of three vectors is given in the following

\begin{proposition}\label{SkewedTripleProduct} 
For quaternions $f, g, v$ with $v$ vector-valued we have $f \cdot (vg) = -g \cdot (vf)$.  In particular, $f \cdot (vf) = 0$.
\end{proposition}

\begin{proof}
Expanding $f$ and $g$ into scalar and vector parts, we have
\begin{align*}
f \cdot (vg) & = \vecpart(f) \cdot (v \times \vecpart(g)) - \Re(f)(v \cdot \vecpart(g)) + \Re(g)(v \cdot \vecpart(f))
\end{align*}
and now observe that this expression is antisymmetric under interchanging $f$ and $g$.
\end{proof}

For $\theta \in \mathbb{R}$, define 
\begin{equation}\label{EulerIdentityQuaternion}
e^{\jvec\theta} := \cos(\theta) + \jvec\sin(\theta)
\end{equation}
We define the left $\jvec$-Fourier transform of a function $f : \mathbb{R}^2 \to \mathbb{H}\,$ by
\begin{equation}\label{JFourierTransformLeft}
(\mathcal{F}_\jvec^L f)(\xi) = (\mathcal{F}_\jvec^L f)(\xi_1, \xi_2) := \frac{1}{(2\pi)^2}\iint_{\mathbb{R}^2} e^{-\jvec(\alpha \xi_1 + \beta \xi_2)} f(\alpha, \beta) \, d\alpha d\beta 
\end{equation}
Since we have the natural identification $\mathbb{C} \cong \mathbb{R} + \mathbb{R}\jvec$, all of the usual formulas for the ordinary Fourier transform still hold with $\jvec$ in place of $i$ provided the $\jvec$-Fourier transform acts on a $1, \jvec$ valued function.  To extend calculation to $\mathbb{H}\,$-valued functions, we use the identities $\mathcal{F}_\jvec \ivec \cdot = \ivec \overline{\mathcal{F}}_\jvec \cdot$ and $\mathcal{F}_\jvec \kvec \cdot = \kvec \overline{\mathcal{F}}_\jvec \cdot$ where we denote $$(\overline{\mathcal{F}}_\jvec f)(\xi) :=  \frac{1}{(2\pi)^2}\iint_{\mathbb{R}^2} e^{\jvec(\alpha \xi_1 + \beta \xi_2)} f \, d\alpha d\beta$$  It is clear from the definition that $(\overline{\mathcal{F}}f)(\xi, \eta) = (\mathcal{F}f)(-\xi, -\eta)$.  The Plancherel Identity for the $\jvec$-Fourier transform continues to hold even in the presence of non-commutativity, since we can write for two $\mathbb{H}$-valued functions $f, g$:
\begin{align*}
\iint f \cdot g \, d\alpha \, d\beta & = \iint (f_0 + f_2\jvec) \cdot (g_0 + \jvec g_2) + (f_1 - f_3\jvec)\ivec \cdot (g_1 - \jvec g_3)\ivec \, d\alpha \, d\beta \\
& = \iint (f_0 + f_2\jvec) \cdot (g_0 + \jvec g_2) + (f_1 - f_3\jvec) \cdot (g_1 - \jvec g_3) \, d\alpha \, d\beta \\
& = \iint \mathcal{F}^L_\jvec[(f_0 + f_2\jvec)] \cdot \mathcal{F}^L_\jvec[(g_0 + \jvec g_2)] + \mathcal{F}^L_\jvec[(f_1 - f_3\jvec)\ivec] \cdot \mathcal{F}^L_\jvec[(g_1 - \jvec g_3)\ivec] \, d\alpha \, d\beta \\
& = \iint \mathcal{F}^L_\jvec[f] \cdot \mathcal{F}^L_\jvec[g] \, d\alpha \, d\beta
\end{align*}
The $\jvec$-Fourier transform of a quaternion product is less well-behaved.  We denote the convolution product of two functions by $$(f \star g)(x) = \frac{1}{(2\pi)^2}\iint_{\mathbb{R}^2} f(x - y)g(y)\, dy$$ Observe also that the convolution property $\mathcal{F}_\jvec(fg) = \mathcal{F}_\jvec(f) \star \mathcal{F}_\jvec(g)$ still holds provided $f$ is $1, \jvec$-valued.  Similarly, the right $\mathcal{F}_\jvec$ Fourier Transform is defined by
\begin{equation}\label{JFourierTransformRight}
(\mathcal{F}_\jvec^R f)(\xi) := \frac{1}{(2\pi)^2}\iint_{\mathbb{R}^2}  f(\alpha, \beta) e^{-\jvec(\alpha \xi_1 + \beta \xi_2)}\, d\alpha d\beta 
\end{equation}
As above, $\mathcal{F}_\jvec^R$ enjoys the Plancherel Identity, and the convolution property provided the \textit{right} factor is $1, \jvec$-valued.

Fractional derivative control will be crucial in the sequel.  It is convenient to work with fractional solid derivatives.  If $f \in C_0^\infty$, these can be defined using either of the above versions of the Fourier transform as multipliers $$\mathcal{F}(|\mathcal{D}|^q f) = |\xi|^q \mathcal{F}(f)$$  If in addition $-2 < q < 0$, we can formally write $|\mathcal{D}|^q$ as a principal value convolution operator in physical space:
\begin{equation}\label{FracSolidDeriv}
|\mathcal{D}|^q f = \frac{1}{C_q}(f \star \pv |(\alpha, \beta)|^{-2 - q})
\end{equation}
where the constant $C_q = 4\pi\Gamma(q/2)\Gamma((2 - q)/2)^{-1}$; here $\Gamma$ is the Euler Gamma function (c.f. \cite{SteinFourierAnalysis}).  We will continue using \eqref{FracSolidDeriv} formally when $q$ is not a nonpositive even integer.\footnote{Indeed this can be made rigorous by analytically continuing \eqref{FracSolidDeriv} as an analytic function of $q$.}

We will use multi-index notation for denoting derivatives.  Let $j = (j_1, j_2) \in \mathbb{N}^2$ be given, and suppose $f = f(\alpha, \beta)$.  Then denote $\partial^j := \frac{\partial^{j_1}}{\partial \alpha^{j_1}} \frac{\partial^{j_2}}{\partial \beta^{j_2}}$.  Addition, subtraction, and ordering of multi-indices is componentwise.  Define the length $|j| = j_1 + j_2$.

For $s \in \mathbb{N}$, define the Sobolev spaces $W^{s, \infty}$ to be the space of functions $f \in L^1_{\text{loc}}(\mathbb{R}^n)$ for which the norm $\|f\|_{W^{s, \infty}} = \sum_{|j| \leq s} \|\partial^j f\|_{L^\infty}$ is finite.  For $s \in \mathbb{R}$, define $\dot{H}^s$ as the completion of $C_0^\infty(\mathbb{R}^2)$ with respect to the norm $\|f\|_{\dot{H}^s} = \|\,|\mathcal{D}|^s f\|_{L^2}$, and let $H^s$ be the completion of $C_0^\infty(\mathbb{R}^2)$ with respect to the norm $\|f\|_{H^s}^2 = \|f\|_{L^2}^2 + \|f\|_{\dot{H}^s}^2$.  For an interval $I \subset \mathbb{R}$ and a Banach space $X$, let $C^j(I, X)$ denote the functions $f(\alpha, \beta, t)$ for which $\sup_I \|\partial^i f/\partial t^j\|_X < \infty$ for all $0 \leq i \leq j$.

To further control half derivatives, we will use the following version of complex interpolation between Sobolev spaces.

\begin{proposition}\label{ComplexInterpolation}
(c.f. Theorems 4.4.1 and 6.4.5 of \cite{InterpolationSpaces}) 
\begin{itemize}
\item[(a)]{Let $X_1^0, \ldots, X_n^0, X_1^1, \ldots, X_n^1$ be Banach spaces, and suppose that $X_j^\theta$ is the $\theta$-complex interpolation space between $X_j^0$ and $X_j^1$.  Suppose further that $T$ is an $n$-multilinear operator so that $T : X_1^0 \times \cdots \times X_n^0 \to Y^0$ is continuous with operator norm $\|T\|_0$ and $T : X_1^1 \times \cdots \times X_n^1 \to Y^1$ is continuous with operator norm $\|T\|_1$.  Then for all $0 < \theta < 1$, $T : X_1^\theta \times \cdots \times X_n^\theta \to Y^\theta$ is continuous with operator norm at most $\|T\|_0^{(1 - \theta)}\|T\|_1^\theta$.}
\item[(b)]{The $\theta$-complex interpolation space between $\dot{H}^{s_0}$ and $\dot{H}^{s_1}$ is $\dot{H}^{(1 - \theta) s_0 + \theta s_1}$.}
\end{itemize}
\end{proposition}

In order to generalize complex analysis to 3D we give a brief overview of Clifford analysis in the quaternion context.  For more information see \cite{WuGlobal3D}, and for a full development see Chapter 2 of \cite{CliffordAlgebras}.

For a $C^2$ open set $\Omega \subset \mathbb{R}^3$, let $F : \Omega \subset \mathbb{R}^3 \to \mathbb{H}$ be given.  We define $F$ to be \textbf{analytic on $\Omega$} if $\mathbf{D}F = 0$, where we have introduced the Dirac operator $\mathbf{D} = \ivec \frac{\partial}{\partial x} + \jvec \frac{\partial}{\partial y} + \kvec \frac{\partial}{\partial z}$.  This implies that each component of $F$ is harmonic; conversely if $\varphi$ is a real-valued harmonic function on $\Omega$, $\mathbf{D}\varphi$ is analytic on $\Omega$.  For vector-valued $F$, observe that the scalar and vector parts of $\mathbf{D}F = 0$ reduce to the div-curl system $\nabla \cdot F = 0, \nabla \times F = 0$.

Denote the fundamental solution of the Laplacian in 3D by $$\Gamma(\vec{x}) = \Gamma(|\vec{x}|) = -\frac{1}{4\pi|\vec{x}|}$$ and denote the Clifford analogue of the Cauchy kernel by $$K(\vec{x}) = -2\mathbf{D}\Gamma(\vec{x}) = -\frac{1}{2\pi}\frac{\vec{x}}{|\vec{x}|^3} \quad \vec{x} \neq \vec{0}$$  One can construct a Hilbert transform associated with a $C^2$ boundary $\Sigma = \partial \Omega$ in parallel with the classical construction of the Hilbert transform associated to a curve in the complex plane.  We collect the properties of the Hilbert transform that we will use in this paper in the

\begin{proposition}\label{PotentialTheoryFacts}
(c.f. Chapter 2 of \cite{CliffordAlgebras})  Let $\Omega$ be a $C^2$ domain in $\mathbb{R}^3$ with boundary $\partial \Omega = \Sigma$.  Let $\mathbf{n}(\vec{x})$ be the outward unit normal to $\Sigma$ at $\vec{x}$, and let $dS(\vec{x})$ be the surface measure of $\Sigma$.  Then
\begin{itemize}
\item[(a)]{If $F$ is analytic on $\Omega$ and decays at infinity, then we have the Cauchy integral formula $$F(\vec{x}) = \frac{1}{2}\iint_\Sigma K(\vec{y} - \vec{x})\mathbf{n}(\vec{y})F(\vec{y}) \, dS(\vec{y})$$ for all $\vec{x} \in \Omega$.}
\item[(b)]{For an $\mathbb{H}$-valued function $f$ defined on $\Sigma$ that decays at infinity, define the Hilbert transform associated to $\Sigma$ by $$H_\Sigma f(\vec{x}) = \pv \iint K(\vec{y} - \vec{x})\mathbf{n}(\vec{y})f(\vec{y})\, dS(\vec{y})$$  Then the Cauchy integral $$Cf(\vec{x}) = \frac{1}{2}\iint_\Sigma K(\vec{y} - \vec{x})\mathbf{n}(\vec{y})f(\vec{y}) \, dS(\vec{y})$$ is analytic on $\Omega$ and extends continuously to the closure $\overline{\Omega}$.  Moreover on $\Sigma$ we have the Plemelj relation $$Cf = \frac12 (I + H_\Sigma)f$$}
\item[(c)]{Suppose that $F$ is a continuous $\mathbb{H}$-valued function defined on $\overline{\Omega}$ that decays at infinity.  Then $F$ is analytic on $\Omega$ if and only if $F(\vec{x}) = H_\Sigma F(\vec{x})$ for all $\vec{x} \in \Sigma$.}
\item[(d)]{The Hilbert transform $H_\Sigma$ satisfies $H_\Sigma^2 = I$, $H_\Sigma[H_\Sigma, T] = -[H_\Sigma, T]H_\Sigma$ on $L^2$.} 
\end{itemize}
\end{proposition}

In the sequel we will denote the Hilbert transform of a surface $\Sigma$ parametrized by a function $\gamma$ by $\nht_\gamma$, reserving the symbols $\oht, \nht, \nht_0$ for the special cases $\gamma = \Xi, \zeta, \alpha \ivec + \beta \jvec$.  In fact the results of Proposition \ref{PotentialTheoryFacts} continue to hold provided the parametrization $\gamma$ satisfies the \textbf{chord-arc condition}: There exist constants $\nu, N > 0$ so that $$\nu \leq \sup_{x \neq y} \frac{|\gamma(x) - \gamma(y)|}{|x - y|} \leq N$$ We denote the double layer potential operator $\mathcal{K}_\gamma = \Re(\nht_\gamma)$ associated to $\nht_\gamma$ by
$$\mathcal{K}_\gamma f(\vec{x}) = \iint_\Sigma \left(K(\vec{y} - \vec{x}) \cdot \mathbf{n}(\vec{y})\right) f(\vec{y}) \, dS(\vec{y})$$ and we denote $\mathcal{K}_\gamma = \mathfrak{K}, \mathcal{K}$ when $\gamma = \Xi, \zeta$ respectively.

Define the (real) adjoint of the Hilbert transform $\nht_\gamma$ through the usual $L^2(\mathbb{R}^2)$ inner product, and denote it by $\nht_\gamma^*$.  Then we have the formula $$\nht_\gamma^* f = -\iint \mathbf{n}(\vec{x}) K(\vec{y} - \vec{x}) f(\vec{y}) dS(\vec{y})$$

	\subsection{Reformulation of Euler's Equation without Quadratic Nonlinearities}

Here we record the evolution equations for $\Xi$ and $\zeta$.  We will often write $\Xi_\beta \partial_\alpha - \Xi_\alpha \partial_\beta = (N \times \nabla)$, where $N = \Xi_\alpha \times \Xi_\beta$.  For brevity we refer to the literature for proofs whenever possible, especially \cite{WuLocal3D} and \cite{WuGlobal3D}.

Key to the cubic nature of the water wave problem in new coordinates are the following explicit commutator identities:

\begin{proposition}\label{CommutatorIdentities}
(c.f. Lemma 3.1 in \cite{WuLocal3D} and Lemma 1.2 of \cite{WuGlobal3D}.)  Let $\gamma : \mathbb{R}^2 \to \mathbb{H}$ be vector-valued and satisfy the chord-arc condition.  Then the following identities hold:
\begin{itemize}
 \item[(a)]{Let $f = f(\alpha, \beta, t, s)$.  For $\partial = \partial_\alpha, \partial_\beta, \partial_t, \partial_s$, denote $\partial^\prime = \partial_{\alpha^\prime}, \partial_{\beta^\prime}, \partial_t, \partial_s$ respectively.  Then we have $$[\partial, \nht_\gamma]f = \iint K(\gamma^\prime - \gamma)(\partial\gamma - \partial^\prime \gamma^\prime) \times (\gamma^\prime_{\beta^\prime}\partial_{\alpha^\prime} - \gamma^\prime_{\alpha^\prime}\partial_{\beta^\prime})f^\prime \, d\alpha^\prime d\beta^\prime$$}
 \item[(b)]{Let $\Gamma = \gamma_\alpha \times \gamma_\beta$.  Then for any scalar-valued function $g$, $$[g(\gamma_\beta \partial_\alpha - \gamma_\alpha \partial_\beta), \nht_\gamma] = \iint K(\gamma^\prime - \gamma)\left(g\Gamma - g^\prime \Gamma^\prime\right) \times (\gamma^\prime_{\beta^\prime}\partial_{\alpha^\prime} - \gamma^\prime_{\alpha^\prime}\partial_{\beta^\prime})f^\prime \, d\alpha^\prime d\beta^\prime$$}
 \item[(c)]{
\begin{align*}
[\partial_t^2, \nht_\gamma] & = \iint K(\gamma^\prime - \gamma)(\gamma_{tt} - \gamma_{tt}^\prime) \times (\gamma^\prime_{\beta^\prime}\partial_{\alpha^\prime} - \gamma^\prime_{\alpha^\prime}\partial_{\beta^\prime})f^\prime \, d\alpha^\prime d\beta^\prime \\
& \quad + \iint \partial_t K(\gamma^\prime - \gamma)(\gamma_t - \gamma_t^\prime) \times (\gamma^\prime_{\beta^\prime}\partial_{\alpha^\prime} - \gamma^\prime_{\alpha^\prime}\partial_{\beta^\prime})f^\prime \, d\alpha^\prime d\beta^\prime \\
& \quad + \iint K(\gamma^\prime - \gamma)(\gamma_t - \gamma_t^\prime) \times (\gamma^\prime_{t\beta^\prime}\partial_{\alpha^\prime} - \gamma^\prime_{t\alpha^\prime}\partial_{\beta^\prime})f^\prime \, d\alpha^\prime d\beta^\prime \\
& \quad + 2 \iint K(\gamma^\prime - \gamma)(\gamma_t - \gamma_t^\prime) \times (\gamma^\prime_{\beta^\prime}\partial_{\alpha^\prime} - \gamma^\prime_{\alpha^\prime}\partial_{\beta^\prime})f_t^\prime \, d\alpha^\prime d\beta^\prime
\end{align*}
}
\end{itemize}
\end{proposition}

Let $\Xi - (\alpha \ivec + \beta \jvec) = x\ivec + y\jvec + z\kvec$; using \eqref{WaterWaveLagrangeFirstPass} along with Proposition \ref{CommutatorIdentities} we can derive the evolution equation for $(I - \oht)z\kvec$ as in Proposition 1.3 of \cite{WuGlobal3D}:

\begin{align}\label{OldEvolutionEquations}
& \quad (\partial_t^2 - \mathfrak{a}(\Xi_\beta \partial_\alpha - \Xi_\alpha \partial_\beta))(I - \oht)z\kvec \notag \\
& \qquad\qquad = \iint K(\Xi^\prime - \Xi)(\Xi_t - \Xi_t^\prime) \times (\Xi^\prime_{\beta^\prime}\partial_{\alpha^\prime} - \Xi^\prime_{\alpha^\prime}\partial_{\beta^\prime})(\Xi^\dagger_t)^\prime\kvec \, d\alpha^\prime, d\beta^\prime \notag \\
& \qquad\qquad - \iint K(\Xi^\prime - \Xi)(\Xi_t - \Xi_t^\prime) \times (\Xi^\prime_{t\beta^\prime}\partial_{\alpha^\prime} - \Xi^\prime_{t\alpha^\prime}\partial_{\beta^\prime})z^\prime\kvec \, d\alpha^\prime d\beta^\prime \\
& \qquad\qquad - \iint \partial_t K(\Xi^\prime - \Xi)(\Xi_t - \Xi_t^\prime) \times (\Xi^\prime_{\beta^\prime}\partial_{\alpha^\prime} - \Xi^\prime_{\alpha^\prime}\partial_{\beta^\prime})z^\prime\kvec \, d\alpha^\prime d\beta^\prime \notag
\end{align}

The first term in the nonlinearity above can be rewritten as cubic since we will see that $\Xi_t$ and $\Xi_t^\dagger$ are orthogonal in $L^2$ up to higher order terms.  However, we must consider derivatives of this equation in order to close the energy, and then the quantity $\mathfrak{a} - 1$ appears and is only linearly small, spoiling the cubic structure.  This obstacle is overcome by introducing Wu's transform (c.f. (1.28) of \cite{WuGlobal3D}):

\begin{equation}\label{KappaDefinition}
\kappa = \Xi - (I + \oht - \mathfrak{K})z\kvec
\end{equation}

While it is immediate from the definition that $\kappa$ has no $\kvec$ component, $\kappa$ is in fact $\ivec, \jvec$-valued as a consequence of the general result:

\begin{proposition}\label{ChangeOfVariablesIJValued}
Let $\gamma = \gamma_1\ivec + \gamma_2 \jvec + \gamma_3 \kvec$ satisfy the chord-arc condition.  Then the quantity $(\nht_\gamma - \mathcal{K}_\gamma)\gamma_3\kvec$ is $\ivec, \jvec$-valued.
\end{proposition}

\begin{proof}
The proof in the case $\gamma = \Xi$ is given on p. 9 of \cite{WuGlobal3D} culminating in identity (1.30); we need only note that the derivation of (1.30) does not depend on any properties of $\Xi$ except that it parametrizes a surface.
\end{proof}

Using the identification $\mathbb{R}^2 \cong \mathbb{R}\ivec + \mathbb{R}\jvec \subset \mathbb{H}$, it makes sense to regard $\kappa$ as a change of variables on $\mathbb{R}^2$ and so consider compositions $f \circ \kappa$ for functions $f : \mathbb{R}^2 \to \mathbb{H}$.  In fact we will show that $\kappa$ is a diffeomorphism in Proposition \ref{KappaControlledByZeta}.  Since we have not yet specified the initial parametrization of the original Lagrangian coordinates, choose $\Xi(\alpha, \beta, 0)$ so that $\kappa(\alpha, \beta, 0) = \alpha \ivec + \beta \jvec$.

Denote $\mathcal{D} = \ivec\partial_\alpha + \jvec\partial_\beta$.  We change variables in \eqref{OldEvolutionEquations} by writing $$\lambda := \zeta - (\alpha \ivec + \beta \jvec) := \xnew\ivec + \ynew\jvec + \znew\kvec := \Xi \circ \kappa^{-1}, \qquad \qquad \nht = \nht_\zeta$$ along with the notation
$$D_t = \partial_t + (\kappa_t \circ \kappa^{-1}) \cdot \mathcal{D}, \qquad b = \kappa_t \circ \kappa^{-1}, \qquad \mathcal{A} = (\mathfrak{a}J(\kappa)) \circ \kappa^{-1}$$ and so have (c.f. (1.35) of \cite{WuGlobal3D}):

\begin{align}\label{NewEvolutionEquations}
& \quad (D_t^2 - \mathcal{A}(\zeta_\beta \partial_\alpha - \zeta_\alpha \partial_\beta))(I - \nht)\znew\kvec \notag \\
& \qquad\qquad = \iint K(\zeta^\prime - \zeta)(D_t \zeta - D_t^\prime \zeta_t^\prime) \times (\zeta^\prime_{\beta^\prime}\partial_{\alpha^\prime} - \zeta^\prime_{\alpha^\prime}\partial_{\beta^\prime})D_t ^\prime (\zeta^\dagger)^\prime\kvec \, d\alpha^\prime d\beta^\prime \notag \\
& \qquad\qquad - \iint K(\zeta^\prime - \zeta)(D_t \zeta - D_t^\prime \zeta_t^\prime) \times (\partial_\beta^\prime D_t^\prime \zeta^\prime \partial_{\alpha^\prime} - \partial_\alpha^\prime D_t^\prime \zeta^\prime\partial_{\beta^\prime})\znew^\prime\kvec \, d\alpha^\prime d\beta^\prime \\
& \qquad\qquad - \iint D_t K(\zeta^\prime - \zeta)(D_t\zeta - D_t^\prime\zeta^\prime) \times (\zeta^\prime_{\beta^\prime}\partial_{\alpha^\prime} - \zeta^\prime_{\alpha^\prime}\partial_{\beta^\prime})\znew^\prime\kvec \, d\alpha^\prime d\beta^\prime \notag
\end{align}

Set $\mathcal{P} := D_t^2 - \mathcal{A}(\zeta_\beta \partial_\alpha - \zeta_\alpha \partial_\beta)$, and denote \eqref{NewEvolutionEquations} by $\mathcal{P}(I - \nht)\znew\kvec = G$.  By taking a derivative $D_t$ to \eqref{NewEvolutionEquations}, we arrive at the following evolution equation for $D_t(I - \nht)\znew\kvec$:

\begin{equation}\label{NewEvolutionEquationsDt}
\mathcal{P}D_t(I - \nht)\znew\kvec = [\mathcal{P}, D_t](I - \nht)\znew\kvec + D_t G
\end{equation}

Moreover, we can write the commutator $[D_t, \mathcal{P}]$ in two different ways: first by straightforwardly distributing the $D_t$, and second by changing variables with respect to $\kappa$:

\begin{align}\label{CommutatorDtWithP}
[D_t, \mathcal{P}] & = (D_t\mathcal{A})(\mathcal{N} \times \nabla) + \mathcal{A}((D_t\zeta_\beta)\partial_\alpha - (D_t\zeta_\alpha)\partial_\beta) \notag \\
& = \mathcal{A}U_\kappa^{-1}\left(\frac{\mathfrak{a}_t}{\mathfrak{a}}\right)(\mathcal{N} \times \nabla) + (\partial_\beta D_t\zeta \partial_\alpha - \partial_\alpha D_t\zeta \partial_\beta)
\end{align}

In order to close these equations we need formulas expressing $b$, $\mathcal{A} - 1$ and $D_t\mathcal{A}$ as quadratic functions of $\zeta$ and its derivatives.  Denote $P := \alpha \ivec + \beta \jvec$.  An immediate yet key consequence of \eqref{KappaDefinition} is the relation

\begin{equation}\label{LambdaIdentity}
\lambda = (I + \nht - \mathcal{K})\znew\kvec
\end{equation}

\begin{proposition}\label{BADtAFormulas}
The following identities hold:
\begin{itemize}
\item[(a)]{$$(I - \nht)b = -[D_t, \nht](I + \nht)\znew\kvec + (I - \nht)D_t\mathcal{K}\znew\kvec$$}
\item[(b)]{\begin{align*}
(I - \mathcal{K})\mathcal{A} & = \Bigl\{\kvec + [D_t, \nht]D_t\zeta + [\mathcal{A}(\mathcal{N} \times \nabla), \nht] (I + \nht)\znew\kvec  \\
& \qquad + (I - \nht)\left(-\mathcal{A}\zeta_\beta \times (\partial_\alpha\mathcal{K}\znew\kvec) + \mathcal{A}\zeta_\alpha \times (\partial_\beta\mathcal{K}\znew\kvec) + \mathcal{A}(\lambda_\alpha \times \lambda_\beta)\right)\Bigr\}_3 \notag
\end{align*} }
\item[(c)]{\begin{align*}
(I - \nht)U_\kappa^{-1}(\mathfrak{a}_t(\Xi_\alpha \times \Xi_\beta)) & = [D_t^2 + \mathcal{A}(\mathcal{N} \times \nabla), \nht]D_t\zeta \\
\end{align*}
\begin{align*}
& = 2 \iint K(\zeta^\prime - \zeta)(D_t^2\zeta - (D_t^2\zeta)^\prime) \times (\zeta_{\beta^\prime}\partial_{\alpha^\prime} - \zeta_{\alpha^\prime}\partial_{\beta^\prime}) D_t^\prime\zeta^\prime \, d\alpha \, d\beta \\
& \quad + 2 \iint K(\zeta^\prime - \zeta)(D_t\zeta - D_t^\prime\zeta^\prime) \times (\zeta_{\beta^\prime}\partial_{\alpha^\prime} - \zeta_{\alpha^\prime}\partial_{\beta^\prime}) (D_t^2\zeta)^\prime \, d\alpha \, d\beta \\
& \quad + \iint ((D_t^\prime\zeta^\prime - D_t\zeta) \cdot \nabla)K(\zeta^\prime - \zeta)((D_t\zeta - D_t^\prime\zeta^\prime) \times (\zeta_{\beta^\prime}\partial_{\alpha^\prime} - \zeta_{\alpha^\prime}\partial_{\beta^\prime}) D_t^\prime\zeta^\prime \, d\alpha \, d\beta \\
& \quad + \iint K(\zeta^\prime - \zeta)\Bigl((D_t\zeta - D_t^\prime\zeta^\prime) \times \partial_{\beta^\prime}D_t^\prime \zeta^\prime)\partial_{\alpha^\prime}D_t^\prime \zeta^\prime - (D_t\zeta - D_t^\prime\zeta^\prime) \times \partial_{\alpha^\prime}D_t^\prime \zeta^\prime)\partial_{\beta^\prime}D_t^\prime \zeta^\prime \Bigr) d\alpha \, d\beta
\end{align*}  }
\end{itemize}
\end{proposition}

\begin{proof}
The first two formulas are shown as in footnote 6 of Proposition 1.4 of \cite{WuGlobal3D}, where for the second we have taken the $\kvec$-component.  The third formula is (2.39) of \cite{WuGlobal3D}.  
\end{proof}

\begin{remark}
The idea behind the formulas of Proposition \ref{BADtAFormulas} is to use the fact that $D_t\zeta$ is the trace of an analytic function in $\Omega(t)$ and that $\lambda$ is, up to terms of higher order, the trace of an analytic function on $\Omega(t)$ as well.  One generates formulas from this fact by recognizing that, if an almost-analytic function $\theta$ satisfies a formula of the form $T\theta = Q$ for some operator $T$, then $$(I - \nht)Q = (I - \nht)T\theta = [T, \nht]\theta + T(I - \nht)\theta,$$ and these two terms are typically of second order by hypothesis and thanks to the commutator formulas of Proposition \ref{CommutatorIdentities}.  Parts (a), (b), and (c) follow in essence for $Q = b, \mathcal{A} - 1, U_\kappa^{-1}(\mathfrak{a}_t(\Xi_\alpha \times \Xi_\beta))$ by respectively applying the above method to $\theta = \lambda, \lambda, D_t \zeta$ and $T = D_t, (D_t^2 - \mathcal{A}(\mathcal{N} \times \nabla)), (D_t^2 + \mathcal{A}(\mathcal{N} \times \nabla))$.
\end{remark}

	\subsection{Analytic Estimates}

We first record some preliminary estimates, which we will need in order to close our energy estimates in $H^s$.

\begin{proposition}\label{SobolevEmbeddings}
\begin{itemize}
\item[(a)]{If $f \in H^2$, then $f \in L^\infty$ and $\|f\|_{L^\infty} \leq C\|f\|_{H^2}$, where $C$ is a universal constant.}
\item[(b)]{If $|\mathcal{D}|^\frac12 f \in H^\frac32$, then $f \in L^\infty$ and $\|f\|_{L^\infty} \leq C\||\mathcal{D}|^\frac12 f\|_{H^\frac32}$, where $C$ is a universal constant.}
\item[(c)]{Let $f, g : \mathbb{R}^2 \to \mathbb{H}\,$.  Then for any $0 < q < 1$ we have 
\begin{align*}
\||\mathcal{D}|^{-1}(fg)\|_{L^2} & \leq C_q(\|g\|_{L^\infty} + \|g\|_{L^{2 - q}})\|f\|_{L^2} \\
& \leq C_q(\|g\|_{L^\infty} + \|g\|_{L^2(\frac{2q}{2 - q})})\|f\|_{L^2}
\end{align*}}
\end{itemize}
\end{proposition}

\begin{proof}
(a) is the usual Sobolev embedding.  To prove (b), we use \eqref{FracSolidDeriv} and Young's Inequality to derive, for $f \in \mathscr{S}(\mathbb{R}^2)$, that
\begin{align*}
\||\mathcal{D}|^{-1/2}f\|_{L^\infty} & = \|f \star |(\alpha, \beta)|^{-3/2}\|_{L^\infty} \\
& \leq \|f \star |(\alpha, \beta)|^{-3/2}_{\leq}\|_{L^\infty} + \|f \star |(\alpha, \beta)|^{-3/2}_{\geq}\|_{L^\infty} \\
& \precsim \|f\|_{L^\infty} + \|f\|_{L^2} \\
& \precsim \|f\|_{H^\frac32}
\end{align*}
Finally, to prove (c) we have by Young's Inequality and H\"older's Inequality that
\begin{align*}
\||\mathcal{D}|^{-1}(fg)\|_{L^2} & \leq \|(fg) \star |(\alpha, \beta)|^{-1}_\leq\|_{L^2} + \|(fg) \star |(\alpha, \beta)|^{-1}_\geq\|_{L^2} \\
& \leq \||(\alpha, \beta)|^{-1}_\leq\|_{L^1} \|g\|_{L^\infty}\|f\|_{L^2} + \||(\alpha, \beta)|^{-1}_\geq\|_{L^\frac{2 - q}{1 - q}} \|g\|_{L^{2 - q}}\|f\|_{L^2} \\
& \leq C_q(\|g\|_{L^\infty} + \|g\|_{L^{2 - q}})\|f\|_{L^2}
\end{align*}
The second inequality of (c) now follows by H\"older's inequality applied to $$\|g\langle (\alpha, \beta) \rangle^p \langle (\alpha, \beta) \rangle^{-p}\|_{L^{2 - q}}$$ The estimate requires that we choose $p\frac{2(2 - q)}{q} > 2$.
\end{proof}

\begin{remark}\label{dPlusNotation}
From this point on in the paper, we use the notation $0^+$ and allow constants to depend on the fixed implicit parameter $\delta > 0$.  For example, the estimate (d) becomes $\||\mathcal{D}|^{-1}(fg)\|_{L^2} \leq C(\|g\|_{L^\infty} + \|g\|_{L^2(0^+)})\|f\|_{L^2}$.
\end{remark}

Finally, in order to estimate the singular integral terms appearing in our formulation, we use the celebrated

\begin{theorem}\label{SingularIntegralL2Estimate}
(Coifman-Meyer-McIntosh-David)  Let $J \in C^1(\mathbb{R}^d, \mathbb{R}^l)$, $A_i \in C^1(\mathbb{R}^d)$ for $i = 1, \ldots, m$, and $F \in C^\infty(\mathbb{R}^l)$.  For $x, y \in \mathbb{R}^d$, define 
\begin{align*}
S_1(A_1, A_2, \ldots, A_m, f) & = \iint F\left(\frac{J(x) - J(y)}{|x - y|}\right)\frac{A_1(x) - A_1(y)}{|x - y|} \cdots \frac{A_m(x) - A_m(y)}{|x - y|} \frac{f(y)}{|x - y|^d} \, dy \\
& := \iint k_1(x, y) f(y) \, dy
\end{align*}
and for $\partial = \partial_{y_k}$ for some $k = 1, \ldots, d$,
\begin{align*}
S_2(A_1, A_2, \ldots, A_m, f) & = \iint F\left(\frac{J(x) - J(y)}{|x - y|}\right)\frac{A_1(x) - A_1(y)}{|x - y|} \cdots \frac{A_m(x) - A_m(y)}{|x - y|} \frac{\partial f(y)}{|x - y|^{d - 1}} \, dy \\
& := \iint k_2(x, y) f(y) \, dy
\end{align*}
Suppose that $k_1(x, y) = -k_1(y, x)$ and $k_2(x, y) = k_2(y, x)$.  Then there exists a constant $C = C(F, \|\nabla J\|_{L^\infty})$ so that
\begin{itemize}
\item[(a)]{$\|S_1(A_1, A_2, \ldots, A_m, f)\|_{L^2} \leq C(1 + m^4)\|\nabla A_1\|_{X_1}\ldots \|\nabla A_m\|_{X_m}\|f\|_{X_0}$,}
\item[(b)]{$\|S_2(A_1, A_2, \ldots, A_m, f)\|_{L^2} \leq C(1 + m^4)\|\nabla A_1\|_{X_1}\ldots \|\nabla A_m\|_{X_m}\|f\|_{X_0}$,}
\end{itemize}
where in both cases one of the spaces $X_i$ for $i = 0, 1, \ldots, m$ is $L^2$ and the others are $L^\infty$.
\end{theorem}

\begin{proof}
See \cite{WuGlobal3D} as well as \cite{CoifmanMeyerMcintoshL2Bounds} and \cite{CoifmanMeyerMcintoshL2BoundsLip}.
\end{proof}

To use this result in $H^s$, we have the

\begin{proposition}\label{SingularIntegralHsEstimate}
Let $s \geq 4$ be given.  Let $\gamma$ be a parametrization of a surface satisfying the chord-arc condition and such that $\|\nabla(\gamma - P)\|_{H^{s - 1}}$ is finite.  Then 
\begin{itemize}
\item[(a)]{For $\partial = \partial_\alpha,\, \partial_\beta$, if $$Tf(\alpha, \beta) = \iint K(\alpha, \beta, \alpha^\prime, \beta^\prime) f(\alpha^\prime, \beta^\prime) \, d\alpha^\prime \, d\beta^\prime,$$ then $$[\partial, T]f = \iint ((\partial + \partial^\prime) K(\alpha, \beta, \alpha^\prime, \beta^\prime))f(\alpha^\prime, \beta^\prime) \, d\alpha^\prime \, d\beta^\prime$$}
\item[(b)]{For a multiindex $j$ of length $n$, $$\|[\partial^j, \nht_\gamma]f\|_{L^2} \leq C\|\nabla(\gamma - P)\|_{H^{n - 1}}\|\nabla f\|_{H^{n - 1}}$$}
\item[(c)]{$$\|\nht_\gamma f\|_{H^s} \leq C(1 + \|\nabla(\gamma - P)\|_{H^{s - 1}})\|f\|_{H^s}$$}
\end{itemize}
\end{proposition}

\begin{proof}
(a) is immediate after an integration by parts.  Likewise, (c) follows immediately from (b) and Theorem \ref{SingularIntegralL2Estimate}.  To prove (b), suppose without loss of generality that $f$ is scalar-valued.  Write the multi index $j$ of length $n$ as a sum $j = j_1 + j_2 + \cdots + j_n$ with each $|j_m| = 1$.  We first write
$$[\partial^j, \nht_\gamma]f = \sum_{m = 1}^n \partial^{j_1 + \cdots + j_{m - 1}}[\partial^{j_m}, \nht_\gamma]\partial^{j_{m + 1} + \cdots + j_n}f$$  Motivated by the expression for $[\partial^{j_m}, \nht_\gamma]$, denote $$K_l^\alpha := \{K(\gamma - \gamma^\prime)(\partial^{j_m}(\gamma - P) - (\partial^{j_m})^\prime(\gamma^\prime - P)) \times \gamma_\beta\}_l$$ and $$K_l^\beta := \{K(\gamma - \gamma^\prime)(\partial^{j_m}(\gamma - P) - (\partial^{j_m})^\prime(\gamma^\prime - P)) \times \gamma_\alpha\}_l$$ Then we can write the $l$th component of each term in the above sum as
\begin{align*}
\{\partial^{j_1 + \cdots + j_{m - 1}}[\partial^{j_m}, \nht_\gamma]\partial^{j_{m + 1} + \cdots + j_n}f\}_l & = \iint ((\partial + \partial^\prime) - \partial^\prime)^{j_1 + \cdots + j_{m - 1}}\left(K_l^\beta\right)(\partial^\prime)^{j_{m + 1} + \cdots + j_n}\partial_{\alpha^\prime}f^\prime \\
& \qquad\qquad - ((\partial + \partial^\prime) - \partial^\prime)^{j_1 + \cdots + j_{m - 1}}\left(K_l^\alpha\right)(\partial^\prime)^{j_{m + 1} + \cdots + j_n}\partial_{\beta^\prime}f^\prime \, d\alpha \, d\beta \\
& = \sum_{p \leq j_1 + \cdots + j_{m - 1}} \biggl(\iint (\partial + \partial^\prime)^p(K_l^\beta) (\partial^\prime)^{j - p}\partial_{\alpha^\prime}f^\prime \\
& \qquad\qquad - (\partial + \partial^\prime)^p(K_l^\alpha) (\partial^\prime)^{j - p}\partial_{\beta^\prime}f^\prime\biggr) \, d\alpha \, d\beta
\end{align*}
Note that $K_l^\alpha$ and $K_l^\beta$ are kernels of the type of those in Theorem \ref{SingularIntegralL2Estimate}(b) .  Since the operator $(\partial + \partial^\prime)$ acts on functions of the form $g(\alpha, \beta) - g(\alpha^\prime, \beta^\prime)$ by $$(\partial + \partial^\prime)\left(g(\alpha, \beta) - g(\alpha^\prime, \beta^\prime)\right) = (\partial g)(\alpha, \beta) - (\partial g)(\alpha^\prime, \beta^\prime)$$ it follows that $(\partial + \partial^\prime)^j$ acting on these kernels is also a kernel of the same type.  

In order to achieve the optimal bounds we must estimate in cases.  If $j < n - 2$, then every term in the kernel has at most $n - 3$ derivatives, and so we may apply Theorem \ref{SingularIntegralL2Estimate} in any way we please along with Sobolev embedding.  In the other cases the dangerous terms are those with a large number of derivatives falling on one of the differences given by components of  $$\partial^k\gamma(\alpha, \beta) - \partial^k\gamma(\alpha^\prime, \beta^\prime) = \partial^k(\gamma - P)(\alpha, \beta) - \partial^k(\gamma - P)(\alpha^\prime, \beta^\prime)$$  If $j = n - 2, n - 1$ then since we assumed that $n \geq 4$ there will be such a difference having either $n - 2$ or $n - 1$ derivatives, and the other terms of the expression will have at most $2$ derivatives.  Therefore we can estimate the term with the highest number of derivatives in $L^2$ and the others in $L^\infty$.  Finally, if $j = n$ and all of the derivatives fall on a difference, then estimate by splitting that difference into two separate singular integrals and estimating each using Theorem \ref{SingularIntegralL2Estimate}(a) and Sobolev embedding.
\end{proof}

\begin{proposition}\label{DifferenceHilbertTransformHs}
Suppose $s \geq 4$ is given.  Let $\gamma_0, \gamma_1$ parametrize two surfaces both satisfying the chord-arc condition and such that $\|\nabla(\gamma_0 - P)\|_{H^{s - 1}}$ and $\|\nabla(\gamma_1 - P)\|_{H^{s - 1}}$ are finite.  Then for $f$ in $H^s$ and $W^{s, \infty}$ respectively we have the estimates
\begin{equation*}
\|(\nht_{\gamma_1} - \nht_{\gamma_0})f\|_{H^s} \leq C\|\nabla(\gamma_0 - \gamma_1)\|_{H^{s - 1}}\|f\|_{H^s} \text{  or  } C\|\nabla(\gamma_0 - \gamma_1)\|_{H^{s - 1}}\|f\|_{W^{s, \infty}}
\end{equation*}
where the constant $C$ depends on $\|\nabla (\gamma_0 - P)\|_{H^{s - 1}}$ and $\|\nabla (\gamma_1 - P)\|_{H^{s - 1}}$.
\end{proposition}

\begin{proof}\label{DiffHilbertTransforms}
Set $\gamma_s = \gamma_0 + s(\gamma_1 - \gamma_0)$.  We express $\nht_{\gamma_1} - \nht_{\gamma_0}$ using Proposition \ref{CommutatorIdentities} as
\begin{align}
& \quad (\nht_{\gamma_1} - \nht_{\gamma_0})f = \int_0^1 \partial_s (\nht_{\gamma_s} f) \, ds = \int_0^1 [\partial_s, \nht_{\gamma_s}]f \, ds \\
& = \int_0^1 \iint K(\gamma_s^\prime - \gamma_s)\Bigl((\gamma_1 - \gamma_0) - (\gamma_1^\prime - \gamma_0^\prime))\Bigr) \times (\partial_\beta\gamma_s\,\partial_\alpha - \partial_\alpha\gamma_s\, \partial_\beta) f \, d\alpha^\prime\,d\beta^\prime \, ds \notag
\end{align}
After using Minkowski's Inequality, the bounds now follow as in Proposition \ref{SingularIntegralHsEstimate}.
\end{proof}

\begin{remark}\label{MoreHsEstimates}
One can always weaken estimates of Propositions \ref{SingularIntegralHsEstimate} and \ref{DifferenceHilbertTransformHs} if it is more convenient to estimate in $W^{s, \infty}$.  For instance, the bounds $$\|[\partial^j, \nht_\gamma]f\|_{H^n} \leq C\|\nabla(\gamma - P)\|_{W^{n - 1, \infty}}\|\nabla f\|_{H^{n - 1}} \text{  and  }  C\|\nabla(\gamma - P)\|_{H^{n - 1}}\|\nabla f\|_{W^{n - 1, \infty}}$$ both follow by less careful estimates.
\end{remark}

The next proposition will imply that the energy we construct in Section 4 controls half derivatives.

\begin{proposition}\label{HalfDerivativeEstimates}
\begin{itemize}
\item[(a)]{For $g$ vector-valued and $f$, $h$ quaternion valued, $$\left|\iint f \cdot (g_\beta h_\alpha - g_\alpha h_\beta) \, d\alpha \, d\beta\right| \leq \|\nabla g\|_{L^\infty} \|f\|_{\dot{H}^\frac12}\|h\|_{\dot{H}^\frac12}$$}
\item[(b)]{For $f$ a vector valued function satisfying $f = -\nht f$ and for sufficiently small $\|\nabla\lambda\|_{L^\infty}$, there is a constant $C(\|\nabla\lambda\|_{L^\infty}) > 0$ so that $$\frac{1}{C}\|f\|_{\dot{H}^\frac12}^2 \leq  -\iint f \cdot (\mathcal{N} \times \nabla)f \leq C\|f\|_{\dot{H}^\frac12}^2$$}
\end{itemize}
\end{proposition}

\begin{proof}
To show (a), consider the bilinear mapping $$T(f, h) = \iint f \cdot (g_\beta h_\alpha - g_\alpha h_\beta) \, d\alpha \, d\beta$$ We have by an integration by parts and Proposition \ref{SkewedTripleProduct} that
\begin{align*}
\iint f \cdot (g_\beta h_\alpha - g_\alpha h_\beta) \, d\alpha \, d\beta & = \iint -f_\alpha \cdot (g_\beta h) + f_\beta \cdot (g_\alpha h) \, d\alpha \, d\beta \\
& = \iint h \cdot (g_\beta f_\alpha - g_\alpha f_\beta) \, d\alpha \, d\beta
\end{align*}  Now H\"older's inequality implies that the $L^2$ norm is bounded by both $$\|\nabla g\|_{L^\infty}\|f\|_{L^2}\|g\|_{\dot{H}^1} \qquad \text{and} \qquad \|\nabla g\|_{L^\infty}\|h\|_{L^2}\|f\|_{\dot{H}^1}.$$  Then (a) follows by these two estimates and applying Proposition \ref{ComplexInterpolation}.

The right-hand inequality of (b) follows immediately from (a).  To prove the left-hand inequality of (b), we manipulate the following integral using the identity $\kvec\mathcal{D} = \nht_0|\mathcal{D}|$:
\begin{align*}
& \quad \iint -f \cdot (\zeta_\beta f_\alpha - \zeta_\alpha f_\beta) \, d\alpha \, d\beta \\
& = \iint -f \cdot (\kvec\mathcal{D})f \, d\alpha \, d\beta + \iint -f \cdot (\lambda_\beta \partial_\alpha - \lambda_\alpha \partial_\beta)f \, d\alpha \, d\beta \\
& = \iint f \cdot (\nht_0|\mathcal{D}|\nht_0 f) \, d\alpha \, d\beta \\
& \qquad \qquad  + \iint f \cdot (\jvec \partial_\alpha - \ivec \partial_\beta)(\nht - \nht_0) f \, d\alpha \, d\beta + \iint -f \cdot (\lambda_\beta \partial_\alpha - \lambda_\alpha \partial_\beta)f \, d\alpha \, d\beta \\
& = \|f\|_{\dot{H}^\frac12}^2 + \iint f \cdot (\jvec \partial_\alpha - \ivec \partial_\beta)(\nht - \nht_0) f \, d\alpha \, d\beta + \iint -f \cdot (\lambda_\beta \partial_\alpha - \lambda_\alpha \partial_\beta)f \, d\alpha \, d\beta
\end{align*}
Note that by Proposition \ref{DifferenceHilbertTransformHs} we have both $\|(\nht - \nht_0)f\|_{L^2} \leq C\|\nabla\lambda\|_{L^\infty}\|f\|_{L^2}$ and for $\partial = \partial_\alpha, \partial_\beta$, $$\|\partial(\nht - \nht_0)f\|_{L^2} = \|(\nht - \nht_0)\partial f\|_{L^2} + \|[\partial, \nht]f\|_{L^2} \leq C\|\nabla\lambda\|_{L^\infty}\|f\|_{\dot{H}^1}$$  Thus, by Proposition \ref{ComplexInterpolation} applied to the operator $T = \nht - \nht_0$, we have $\|(\nht - \nht_0)f\|_{\dot{H}^\frac12} \leq C\|\nabla\lambda\|_{L^\infty}\|f\|_{\dot{H}^\frac12}$, and so by (a) the above becomes
\begin{align*}
\|f\|_{\dot{H}^\frac12}^2 & \leq \iint -f \cdot (\mathcal{N} \times \nabla)f \, d\alpha \, d\beta + \|f\|_{\dot{H}^\frac12}\|(\nht - \nht_0)f\|_{\dot{H}^\frac12} + \|\nabla \lambda\|_{L^\infty}\|f\|_{\dot{H}^\frac12}^2 \\
& \leq \iint -f \cdot (\mathcal{N} \times \nabla)f \, d\alpha \, d\beta + C\|\nabla \lambda\|_{L^\infty}\|f\|_{\dot{H}^\frac12}^2
\end{align*}
Now (b) follows provided $\|\nabla\lambda\|_{L^\infty}$ is sufficiently small.
\end{proof}

We close this section with half-derivative estimates of operators formed by commuting singular integrals and derivative operators.

\begin{proposition}\label{CommutatorEstimates}
Let $s \geq 4$, let $\mathfrak{d}, \mathfrak{d}_1, \mathfrak{d}_2 = \partial_\alpha, \partial_\beta, \partial_t, \partial_s, D_t$, and let $\gamma$ parametrize a Lipschitz surface and satisfy the chord-arc condition.  Then the following estimates hold for all $0 \leq \nu \leq 1$:
\begin{itemize}
\item[(a)]{There is a constant $C$ depending on $\|\nabla(\gamma - P)\|_{H^s}$ so that $$\|[\mathfrak{d}, \nht_\gamma]f\|_{H^s} \leq C\|\,|\mathcal{D}|^{1 - \nu}\mathfrak{d}(\gamma - P)\|_{H^s} \|\,|\mathcal{D}|^\nu f\|_{H^s}$$}
\item[(b)]{There is a constant $C$ depending on $\|\nabla(\gamma - P)\|_{H^s}$ so that
$$\|[\mathfrak{d}_1, [\mathfrak{d}_2, \nht_\gamma]]f\|_{H^s} \leq C(\|\nabla \mathfrak{d}_1\mathfrak{d}_2 (\gamma - P)\|_{H^{s - 1}} + \|\nabla \mathfrak{d}_1(\gamma - P)\|_{H^s}\|\nabla \mathfrak{d}_2 (\gamma - P)\|_{H^{s - 1}})\|f\|_{H^s}$$}
\end{itemize}
\end{proposition}

\begin{proof}
Write
$$[\mathfrak{d}, \nht_\gamma]f = \iint K(\gamma^\prime - \gamma)(\mathfrak{d}\gamma - \mathfrak{d}^\prime\gamma^\prime) \times (\gamma_\beta^\prime f_\alpha^\prime - \gamma_\alpha^\prime f_\beta^\prime) \, d\alpha \, d\beta$$
and distribute $\partial^j$ derivatives as in Proposition \ref{SingularIntegralHsEstimate}.  By Proposition \ref{ComplexInterpolation} and Theorem \ref{SingularIntegralL2Estimate}, we can estimate the operator
$$T(g, h) = \iint K(\gamma^\prime - \gamma)(g - g^\prime) \times (\gamma_\beta^\prime h_\alpha^\prime - \gamma_\alpha^\prime h_\beta^\prime) \, d\alpha \, d\beta$$ by $$\|T(g, h)\|_{L^2} \leq C\|g\|_{H^{3 - \nu}}\|h\|_{\dot{H}^\nu} \quad \text{and} \quad C\|g\|_{\dot{H}^{1 - \nu}}\|\, |\mathcal{D}|^\nu h\|_{H^\frac32}$$  Part (a) follows by applying these estimates to the term $T(\partial^j\mathfrak{d}(\gamma - P), h)$ and any of the terms where no derivatives fall on $f$.

Similarly, (b) follows by applying Proposition \ref{SingularIntegralHsEstimate} to the following explicit commutator formula:
\begin{align*}
[\mathfrak{d}_1, [\mathfrak{d}_2, \nht_\gamma]]f & = \iint K(\gamma - \gamma^\prime) (\mathfrak{d}_1 \mathfrak{d}_2 \gamma - \mathfrak{d}_1^\prime \mathfrak{d}_2^\prime \gamma^\prime) \times (\gamma_\beta^\prime f_\alpha^\prime - \gamma_\alpha^\prime f_\beta^\prime) \, d\alpha \, d\beta \\
& \quad + \iint K(\gamma - \gamma^\prime) (\mathfrak{d}_2 \gamma - \mathfrak{d}_2^\prime \gamma^\prime) \times ((\mathfrak{d}_1^\prime \gamma_\beta^\prime) f_\alpha^\prime - (\mathfrak{d}_1^\prime \gamma_\alpha^\prime) f_\beta^\prime) \, d\alpha \, d\beta \\
& \quad + \iint (\mathfrak{d}_1 + \mathfrak{d}_1^\prime)K(\gamma - \gamma^\prime) (\mathfrak{d}_2 \gamma - \mathfrak{d}_2^\prime \gamma^\prime) \times (\gamma_\beta^\prime f_\alpha^\prime - \gamma_\alpha^\prime f_\beta^\prime) \, d\alpha \, d\beta
\end{align*}
\end{proof}

\begin{remark}
As in Proposition \ref{SingularIntegralHsEstimate}, we are always free to relax to $L^\infty$ estimates if it is convenient, provided one of the factors is still estimated in $L^2$.  If $f = \tilde{f} + (f - \tilde{f})$, we will often need to use the above estimates in such a way that we replace $\||\mathcal{D}|^\nu f\|_{H^s}$ in the above proposition by $\||\mathcal{D}|^\nu(f - \tilde{f})\|_{H^s} + \|\tilde{f}\|_{W^{s + 1, \infty}}$, with similar modifications for the other quantities in the proposition.
\end{remark}

\begin{remark}
Since we can write the difference of two Hilbert transforms as in Proposition \ref{DiffHilbertTransforms}, part (a) of the above Lemma along with Proposition \ref{ComplexInterpolation} implies the estimates
\begin{align*}
\|(\nht_{\gamma_1} - \nht_{\gamma_0})f\|_{H^s} & \leq C\|\,|\mathcal{D}|^{1 - \nu}(\gamma_1 - \gamma_0)\|_{H^s} \|\,|\mathcal{D}|^\nu f\|_{H^s} \quad \text{for}\quad  0 \leq \nu \leq 1
\end{align*}
where $C$ now depends on $\|\nabla(\gamma_i - P)\|_{H^s}$, $i = 0, 1$.
\end{remark}

\section{The Formal Calculation of the Approximate Solution}

In this section we determine the correctors to the wave packet-like approximate solution so that the residual to \eqref{NewEvolutionEquations} is physically of order $O(\epsilon^5)$.  Our first task is to write the wave packet in \eqref{WavePacketVector} in terms of quaternions.  In analogy to our choice of wave packet in \cite{TotzWu2DNLS}, we can use \eqref{EulerIdentityQuaternion} in order to take the following as the leading term of our approximate solution:
\begin{align}\label{WavePacketQuaternion}
\tilde{\lambda} & = \epsilon \ivec A(\epsilon\alpha, \epsilon\beta, \epsilon t, \epsilon^2 t) e^{\jvec(k\alpha + \omega t)} + O(\epsilon^2) \notag \\
& = \epsilon \left( \Re(Ae^{\jvec(k\alpha + \omega t)})\ivec + \Im(Ae^{\jvec(k\alpha + \omega t)})\kvec\right) + O(\epsilon^2) \\
& := \epsilon(\xnew^{(1)}\ivec + \ynew^{(1)}\jvec + \znew^{(1)}\kvec) + O(\epsilon^2), \notag
\end{align}
where the function $A$ is $1, \jvec$-valued.  In order to systematically develop the correctors of $\tilde{\lambda}$, we adopt a multiscale ansatz for $\tilde{\lambda}$.  If we let $\alpha_0 = \alpha$, $\alpha_1 = \epsilon\alpha$, $\beta_1 = \epsilon\beta$, $t_0 = t$, $t_1 = \epsilon t$, $t_2 = \epsilon^2 t$, we write 
\begin{align}\label{WavePacketMultiscale}
\tilde{\lambda} & = \epsilon \ivec A(\alpha_1, \beta_1, t_1, t_2) e^{\jvec\phi} + O(\epsilon^2) \notag \\
& = \epsilon\left(\Re(Ae^{\jvec\phi})\ivec + \Im(Ae^{\jvec\phi})\kvec\right) + O(\epsilon^2)
\end{align}
where we have introduced the phase 
\begin{equation}\label{Phase}
\phi := k\alpha_0 + \omega t_0
\end{equation}
If we are to seek such an ansatz, we must interpret the action of the operators in \eqref{NewEvolutionEquations} and Proposition \ref{BADtAFormulas} on multiscale functions and interpret the result as multiscale functions.  Interpreting derivatives in this way is straightforward by the Chain rule: $$\partial_\alpha = \partial_{\alpha_0} + \epsilon \partial_{\alpha_1}, \qquad \partial_\beta = \epsilon \partial_{\beta_1}, \qquad \partial_t = \partial_{t_0} + \epsilon\partial_{t_1} + \epsilon^2 \partial_{t_2}$$  However it is not immediately clear how to interpret $\nht$ acting on a multiscale function.  One can formally expand the kernel of $\nht$ into a power series of homogeneous terms, and each of these terms yields an operator that can be written in terms of iterates of commutators with known quantities and the flat Hilbert transform 
\begin{equation}\label{FlatHilbertTransform}
\nht_0 f(\alpha, \beta) = \frac{1}{2\pi^2} \iint_{\mathbb{R}^2} \frac{(\alpha - \alpha^\prime)\ivec + (\beta - \beta^\prime)\jvec}{|(\alpha, \beta) - (\alpha^\prime, \beta^\prime)|^3} \kvec f(\alpha^\prime, \beta^\prime) \, d\alpha^\prime d\beta^\prime
\end{equation}
This reduces the problem to understanding how $\nht_0$ acts on multiscale functions.  Since the only multiscale functions that arise in our formal calculation are in essence of the form $F(\alpha_1, \beta_1, t_1, t_2)e^{n\jvec\phi}$ for $k \in \mathbb{R}$, $n \in \mathbb{Z}$, and $F$ a $1, \jvec$-valued function, we begin by understanding how $\nht_0$ acts on these types of functions.

	\subsection{The Action of the Flat Hilbert Transform on Wave Packets}

We first observe that because our wave packets are concentrated in frequency space about the fixed frequency $(k, 0)$, we can always localize a smooth wave packet about its wave number in Fourier space at the expense of a small error.  In this section we use only the left $\jvec$-Fourier transform with frequency variable $\xi = (\xi_1, \xi_2)$, and denote $\mathcal{F}_\jvec^L[f] = \mathcal{F}[f] = \hat{f}$ for brevity.

\begin{lemma}\label{WavePacketCutoff}
Let $s, m \geq 0$, $k \neq 0$, and $\epsilon > 0$ be given.  Let $\mathcal{B}_k$ be the Fourier multiplier with symbol $$\hat{\mathcal{B}}_k(\xi) := \begin{cases} 1 & |(\xi_1 - k, \xi_2)| \leq \frac{1}{2}k \cr 0 & \text{otherwise}\end{cases}$$ Then for any function $A \in H^{s + m}$, there is a constant $C$ depending only on $k, s, m$ so that $$\|A(\epsilon\alpha, \epsilon\beta)e^{\jvec k\alpha} - \mathcal{B}_k A(\epsilon\alpha, \epsilon\beta)e^{\jvec k\alpha}\|_{H^s} \leq C \epsilon^{m - 1}\|A\|_{H^{s + m}}$$
\end{lemma}

\begin{proof}
We calculate by Plancherel's Identity that for any $m \geq 0$ that
\begin{align*}
& \quad\; \|A(\epsilon\alpha, \epsilon \beta)e^{\jvec k\alpha} - \mathcal{B}_k A(\epsilon\alpha, \epsilon \beta)e^{\jvec k\alpha}\|_{L^2_{\alpha, \beta}} \\
& = \left(\iint_{|(\xi_1 - k, \xi_2)| > \frac{1}{2}k} \left|\frac{1}{\epsilon^2}\hat{A}\left(\frac{\xi_1 - k}{\epsilon}, \frac{\xi_2}{\epsilon}\right)\right|^2 d\xi_1\,d\xi_2\right)^{1/2} \\
& \leq \left(\iint_{|(\xi_1 - k, \xi_2)| > \frac{1}{2}k} \left|\frac{\epsilon^m}{\langle(\xi_1 - k, \xi_2)\rangle^m}\frac{1}{\epsilon^2}\widehat{\langle\mathcal{D}\rangle^m A}\left(\frac{\xi_1 - k}{\epsilon}, \frac{\xi_2}{\epsilon}\right)\right|^2 d\xi_1\,d\xi_2\right)^{1/2} \\
& \leq C \epsilon^{m - 1}\|A\|_{H^m}
\end{align*} where the constant $C$ depends only on $k$ and $s$.  Note that we have lost a power of $\epsilon$ by measuring $A$ in the slow variable $\epsilon\alpha$.  Since $\mathcal{B}_k$ commutes with differentiation, the result now follows upon applying the above to $\partial^j A$ for $|j| \leq s$.
\end{proof}

This result allows us to interpret the action of $\nht_0$ on a wave-packet by expanding the symbol of $\nht_0$ in a Taylor series about the frequency $(k, 0)$.  The effect of this is to write the action of $\nht_0$ on a wave packet as a series of differential operators, which are then easily interpreted as operators on multiscale functions.

\begin{proposition}\label{FlatHilbertTransformExpansion}
Let $F \in H^{s + 4}$ be a $1, \jvec$-valued function, and denote $f(\alpha, \beta) = F(\epsilon\alpha, \epsilon\beta)e^{\jvec k \alpha}$ for $k \in \mathbb{R}$.  When $k = 0$ we interpret $\nht^{(0)}(F(\alpha_1, \beta_1)) = (\nht_0F)(\alpha_1, \beta_1)$ with no correctors.  When $k \neq 0$ we have the following estimate $$\|
(\nht_0 - \nht_0^{(0)} + \epsilon \nht_0^{(1)} + \epsilon^2\nht_0^{(2)} + \epsilon^3\nht_0^{(3)})f\|_{H^s} \leq C\epsilon^3\|F\|_{H^{s + 4}}$$ where the operators $\mathcal{H}_0^{(j)}$ for $j = 0, 1, 2, 3$ are
\begin{equation*}
\nht_0^{(0)}f = -\sgn(k)f
\end{equation*}
\begin{equation*}
\nht_0^{(1)}f = -\frac{1}{|k|} \ivec\partial_{\beta_1}f
\end{equation*}
\begin{equation*}
\nht_0^{(2)}f = -\frac{1}{2k|k|} \partial_{\beta_1}^2f + \frac{1}{k|k|} \kvec \partial_{\alpha_1 \beta_1}f
\end{equation*}
\begin{equation*}
\nht_0^{(3)}f = -\frac{1}{|k|^3}\jvec\partial_{\alpha_1 \beta_1 \beta_1}f + \frac{1}{|k|^3}\ivec\partial_{\alpha_1\alpha_1\beta_1}f - \frac{1}{2|k|^3}\ivec \partial_{\beta_1\beta_1\beta_1}f
\end{equation*}
\end{proposition}

\begin{proof}
The case $k = 0$ is immediate since $\nht_0$ is invariant under dilations.  Hence it suffices to consider the case where $k \neq 0$.  Applying \eqref{FracSolidDeriv} to the components of \eqref{FlatHilbertTransform} yields formulas for the symbols of the Riesz transforms:
\begin{equation}\label{RieszFormulas}
\mathcal{F}\left(\frac{1}{2\pi}\pv\frac{\alpha}{|(\alpha, \beta)|^3}\right) = -\jvec \frac{\xi}{|(\xi, \eta)|} \qquad \mathcal{F}\left(\frac{1}{2\pi}\pv\frac{\beta}{|(\alpha, \beta)|^3}\right) = -\jvec \frac{\eta}{|(\xi, \eta)|}
\end{equation} we can write the symbol of $\nht_0 = -\jvec\mathcal{R}_1 + \ivec\mathcal{R}_2$ in coordinates  as follows:
\begin{align*}
\mathcal{F}\nht_0 & = \mathcal{F}(\ivec \mathcal{R}_1 \kvec + \jvec \mathcal{R}_2 \kvec) \\
& = \mathcal{F}(-\jvec\mathcal{R}_1 + \jvec\mathcal{R}_2\kvec) \\
& = -\jvec\mathcal{F}\mathcal{R}_1 + \jvec \mathcal{F}\mathcal{R}_2\kvec \\
& = -\jvec\left(-\jvec\frac{\xi_1}{|\xi|}\right)\mathcal{F} + \jvec\left(-\jvec \frac{\xi_2}{|\xi|}\right)\mathcal{F}\kvec \\
& = -\frac{\xi_1}{|\xi|}\mathcal{F} + \frac{\xi_2}{|\xi|}\mathcal{F}\kvec
\end{align*}  We will expand these symbols in a formal power series $(1 + (\xi_2/\xi_1)^2)^{-1/2}$ about $(\xi_1 - k, \xi_2)$ to third order using the power series expansions $$(1 + (\xi_2/\xi_1)^2)^{-1/2} = 1 - \frac{\xi_2^2}{2\xi_1^2} + O((\xi_1 - k)^4 + \xi_2^4), \qquad \frac{1}{|\xi_1|} = \frac{1}{|k|}\sum_{i = 0}^\infty(-1)^i\frac{(\xi_1 - k)^i}{k^i},$$ which are absolutely convergent in the support of $\hat{\mathcal{B}}_k$.  We calculate that
\begin{align*}
\frac{\xi_1}{|\xi|} & = \sgn(\xi_1)(1 + (\xi_2/\xi_1)^2)^{-1/2} \\
& = \sgn(k) - \frac{1}{2}\frac{\xi_2^2}{\xi_1|\xi_1|} + O\left(\frac{\xi_2^4}{\xi_1^4}\right) \\
& = \sgn(k) - \frac{1}{2}\frac{1}{k|k|}\xi_2^2 + \frac{1}{|k|^3}\xi_2^2(\xi_1 - k) + O\left(|\xi_1 - k|^4 + |\xi_2|^4\right)
\end{align*}  Note also that since $\mathcal{B}_k$ has a scalar-valued symbol, it commutes with the Riesz transforms.  If we use the notation $f = g + \mathcal{O}(\epsilon^n)$ to abbreviate $\|f - g\|_{H^s} \leq C\epsilon^n$, we have for $A \in H^{s + 4}$ the following asymptotics in $H^s$:
\begin{align*}
\mathcal{F}\ivec\mathcal{R}_1\kvec Ae^{\jvec k \alpha} & = -\frac{\xi_1}{|\xi|}\mathcal{F}\mathcal{B}_kAe^{\jvec k\alpha} + \mathcal{O}(\epsilon^3) \\
& = \left(-\sgn(k) + \frac{1}{2}\frac{1}{k|k|}\xi_2^2 - \frac{1}{|k|^3}\xi_2^2(\xi_1 - k)\right)\hat{\mathcal{B}}_k\frac{1}{\epsilon^2}\hat{A}\left(\frac{\xi_1 - k}{\epsilon},\frac{\xi_2}{\epsilon}\right) + \mathcal{O}(\epsilon^3) \\
& = \mathcal{F}\mathcal{B}_k\left(\left(-\sgn(k) - \frac{\epsilon^2}{2k|k|}\partial_{\beta_1}^2 - \jvec\frac{\epsilon^3}{|k|^3}\partial_{\beta_1}^2\partial_{\alpha_1}\right)Ae^{\jvec k\alpha}\right) + \mathcal{O}(\epsilon^3) \\
& = \mathcal{F}\left(\left(-\sgn(k) - \frac{\epsilon^2}{2k|k|}\partial_{\beta_1}^2 - \jvec\frac{\epsilon^3}{|k|^3}\partial_{\beta_1}^2\partial_{\alpha_1}\right)Ae^{\jvec k\alpha}\right) + \mathcal{O}(\epsilon^3)
\end{align*}
where in the last step we use Lemma \ref{WavePacketCutoff} so that the number of derivatives lost exactly balances the powers of $\epsilon$ needed to achieve an error of $\mathcal{O}(\epsilon^3)$.  Similarly, using the following expansion about $(k, 0)$:
\begin{equation*}
\frac{\xi_2}{|\xi|} = \frac{\xi_2}{|k|} - \frac{\xi_2(\xi_1 - k)}{k|k|} + \frac{1}{k^2|k|}\xi_2(\xi_1 - k)^2 - \frac{1}{2}\frac{\xi_2^3}{k^2|k|} + O(|\xi_1 - k|^4 + |\xi_2|^4)
\end{equation*} we calculate in the same way that
\begin{align*}
& \mathcal{F}\jvec\mathcal{R}_2\kvec Ae^{\jvec k \alpha} \\
& = \frac{\xi_2}{|\xi|}\mathcal{F}[\mathcal{B}_k\kvec Ae^{\jvec k\alpha}]_{(\xi_1 - k, \xi_2)} + \mathcal{O}(\epsilon^3) \\
& = \left(\frac{\xi_2}{|k|} - \frac{\xi_2(\xi_1 - k)}{k|k|} + \frac{1}{k^2|k|}\xi_2(\xi_1 - k)^2 - \frac{1}{2}\frac{\xi_2^3}{k^2|k|}\right)\mathcal{F}[\mathcal{B}_k\kvec Ae^{\jvec k\alpha}]_{(\xi_1 - k, \xi_2)} + \mathcal{O}(\epsilon^3)
\end{align*}
which is equal to
\begin{equation*}
\mathcal{F}\left(\left(-\frac{\epsilon}{|k|}\ivec\partial_{\beta_1} + \frac{\epsilon^2}{k|k|}\kvec\partial_{\alpha_1}\partial_{\beta_1} + \frac{\epsilon^3}{|k|^3}\ivec\partial_{\alpha_1\alpha_1\beta_1} - \frac{\epsilon^3}{2|k|^3}\ivec \partial_{\beta_1\beta_1\beta_1}\right) Ae^{\jvec k\alpha}\right)
\end{equation*} up to an error $\mathcal{O}(\epsilon^3)$.  Summing these gives the full expansion for $\nht^{(0)}$.
\end{proof}

	\subsection{Expansion of the Full Hilbert Transform}

We now consider the expansion of the full Hilbert transform $$\nht f = \frac{1}{2\pi^2}\iint \frac{\zeta(\alpha, \beta) - \zeta(\alpha^\prime, \beta^\prime)}{|\zeta(\alpha, \beta) - \zeta(\alpha^\prime, \beta^\prime)|^3} (\zeta_\alpha(\alpha^\prime, \beta^\prime) \times \zeta_\beta(\alpha^\prime, \beta^\prime)) f(\alpha^\prime, \beta^\prime) \, d\alpha^\prime \, d\beta^\prime$$  We seek to expand the various parts of the above kernel in a perturbation from the flat Hilbert transform.  For the functions $f$ to follow, we abbreviate $f = f(\alpha, \beta)$ and $f^\prime = f(\alpha^\prime, \beta^\prime)$.  In order to anticipate formulas for our operators, we proceed formally to find the third order expansion of the difference quotient in the kernel in functions homogeneous in $\lambda - \lambda^\prime$: 
\begin{align}\label{CauchyKernelExpansion}
\frac{\zeta - \zeta^\prime}{|\zeta - \zeta^\prime|^3} & = \frac{(P - P^\prime) + (\lambda - \lambda^\prime)}{(|(P - P^\prime) + (\lambda - \lambda^\prime)|^2)^{3/2}} \notag \\
& = \frac{(P - P^\prime) + (\lambda - \lambda^\prime)}{|P - P^\prime|^3}\left(1 + 2\frac{(P - P^\prime)\cdot(\lambda - \lambda^\prime)}{|P - P^\prime|^2} + \frac{|\lambda - \lambda^\prime|^2}{|P - P^\prime|^2}\right)^{-3/2} \notag \\
& = \frac{(P - P^\prime) + (\lambda - \lambda^\prime)}{|P - P^\prime|^3}\Biggl(1 - 3\frac{(P - P^\prime)\cdot(\lambda - \lambda^\prime)}{|P - P^\prime|^2} \\
& \qquad - \frac{3}{2}\frac{|\lambda - \lambda^\prime|^2}{|P - P^\prime|^2} + \frac{15}{2}\left(\frac{(P - P^\prime)\cdot(\lambda - \lambda^\prime)}{|P - P^\prime|^2}\right)^2 \notag \\
& \qquad + \frac{15}{2}\frac{((P - P^\prime) \cdot (\lambda - \lambda^\prime))|\lambda - \lambda^\prime|^2}{|P - P^\prime|^4} - \frac{35}{2}\frac{((P - P^\prime) \cdot ((\lambda - \lambda^\prime))^3}{|P - P^\prime|^6} \notag \\
& \qquad + O\biggl(\left|\frac{(P - P^\prime) \cdot (\lambda - \lambda^\prime)}{|P - P^\prime|^2}\right|^4 + \left|\frac{|\lambda - \lambda^\prime|^2}{|P - P^\prime|^2}\right|^2\biggr)\Biggr) \notag
\end{align}  The normal vector term can be expanded exactly as follows:
\begin{equation}\label{NormalVectorExpansion}
\zeta_\alpha \times \zeta_\beta = \kvec + (\lambda_\alpha \times \jvec) + (\ivec \times \lambda_\beta) + (\lambda_\alpha \times \lambda_\beta)
\end{equation}  This expansion of the full Hilbert transform above motivates us to find the Fourier transform of the kernels of the form $-\frac{1}{2\pi^2}|P|^n$ for $n = -3, -5, -7$.  To do so, we use \eqref{FracSolidDeriv} to obtain
\begin{equation*}
\mathcal{F}\left(\frac{1}{2\pi^2}|P|^{-1}\right) = \frac{1}{|(\xi, \eta)|}
\end{equation*}  and then appeal to the identity $\Delta|P|^{-n} = n^2|P|^{-(n + 2)}$ to find that
\begin{equation}\label{RieszPotentialFormulas}
\mathcal{F}\left(\frac{1}{2\pi^2}|P|^{-3}\right) = -|\xi|, \quad \mathcal{F}\left(\frac{1}{2\pi^2}|P|^{-5}\right) = \frac{1}{9}|\xi|^{3}, \quad \mathcal{F}\left(\frac{1}{2\pi^2}|P|^{-7}\right) = -\frac{1}{225}|\xi|^{5}
\end{equation}

In order to express $|\mathcal{D}|$ in terms of ordinary differentiation and the flat Hilbert Transform, we recall the identity 
\begin{equation}\label{|D|Identity}
|\mathcal{D}| = \nht_0 \kvec \mathcal{D}
\end{equation}
which is easily verified on the Fourier side.

This allows us to further develop the expansions of $\nht_1$, $\nht_2$ and $\nht_3$ unambiguously into powers of $\epsilon$:  $\nht_1 = \epsilon \nht_1^{(1)} + \epsilon^2 \nht_1^{(2)} + \epsilon^3 \nht_1^{(3)} + O(\epsilon^4)$,  $\nht_2 = \epsilon^2 \nht_2^{(2)} + \epsilon^3 \nht_2^{(3)} + O(\epsilon^4)$, and $\nht_3 = \epsilon^3 \nht_3^{(3)} + O(\epsilon^4)$.  Then we define
\begin{align*}
\nht & = \nht_0 + \nht_1 + \nht_2 + \nht_3 + \cdots \\
& := (\nht_0^{(0)}) + \epsilon(\nht_0^{(1)} + \nht_1^{(1)}) + \epsilon^2(\nht_0^{(2)} + \nht_1^{(2)} + \nht_2^{(2)}) \\
& \quad + \epsilon^3(\nht_0^{(3)} + \nht_1^{(3)} + \nht_2^{(3)} + \nht_3^{(3)}) + O(\epsilon^4)\\
& := \nht^{(0)} + \epsilon \nht^{(1)} + \epsilon^2\nht^{(2)} + \epsilon^3\nht^{(3)} + O(\epsilon^4)
\end{align*}

We will see in the formal calculation that we need to develop the approximate solution to the fourth order, and so we set
\begin{equation*}
\tilde{\lambda} = \sum_{j = 1}^4 \epsilon^j\lambda^{(j)} = \sum_{j = 1}^4 \epsilon^j\left(\xnew^{(j)}\ivec + \ynew^{(j)}\jvec + \znew^{(j)}\kvec\right)
\end{equation*}
We defer the calculation of the multiscale operators $\nht^{(j)}$ to Appendix A and record the results that we will explicitly use in the calculation in the following

\begin{proposition}\label{HilbertTransformExpansionFormulas}
Assume that $f = f(\alpha_0, \alpha_1, \beta_1)$, and denote $p_1 = \xnew + \znew\jvec$, and $p_1^{(j)} = \xnew^{(j)} + \znew^{(j)}\jvec$.  Then we have the formulas
\begin{equation*}
\nht_1^{(1)}f = [(\xnew^{(1)} + \jvec\znew^{(1)}), \nht^{(0)}]\partial_{\alpha_0}f
\end{equation*}
\begin{align*}
\nht_1^{(2)}f & = [(\xnew^{(2)} + \jvec\znew^{(2)}), \nht^{(0)}]\partial_{\alpha_0}f + [(\xnew^{(1)} + \jvec\znew^{(1)}), \nht_0^{(1)}]\partial_{\alpha_0}f \\
& + [(\xnew^{(1)} + \jvec\znew^{(1)}), \nht^{(0)}]\partial_{\alpha_1}f + [\ynew^{(1)} - \ivec\znew^{(1)}, \nht^{(0)}]\partial_{\beta_1}f
\end{align*}
\begin{equation*}
\nht_2^{(2)} f = -[p_1^{(1)}, \nht^{(0)}](\partial_{\alpha_0} p_1^{(1)})(\partial_{\alpha_0} f) + \frac{1}{2}[p_1^{(1)}, [p_1^{(1)}, \nht^{(0)}]]\partial_{\alpha_0 \alpha_0} f
\end{equation*}
\end{proposition}

Using the formulas of Propositions \ref{FlatHilbertTransformExpansion} and \ref{HilbertTransformExpansionFormulas}, we now define
\begin{equation}\label{TildeHDefinition}
\tilde{\nht} = \nht^{(0)} + \epsilon \nht^{(1)} + \epsilon^2 \nht^{(2)} + \epsilon^3 \nht^{(3)}
\end{equation}
In Section 3.4 we will give estimates that justifies the use of power expansions to develop this approximation of the Hilbert Transform.

	\subsection{The Multiscale Calculation}

We take as our system the equations 
\begin{align}\label{WaterWaveEvolutionMultiscale}
& \quad (D_t^2 - \mathcal{A}(\zeta_\beta \partial_\alpha - \zeta_\alpha \partial_\beta))(I - \nht)\znew\kvec = [D_t, \nht]D_t\zeta^\dagger \notag \\
& \qquad\qquad - \iint K(\zeta^\prime - \zeta)(D_t \zeta - D_t^\prime \zeta^\prime) \times (\partial_\beta^\prime D_t^\prime \zeta^\prime \partial_{\alpha^\prime} - \partial_\alpha^\prime D_t^\prime \zeta^\prime\partial_{\beta^\prime})\znew^\prime\kvec \, d\alpha^\prime d\beta^\prime \\
& \qquad\qquad - \iint D_t K(\zeta^\prime - \zeta)(D_t\zeta - D_t^\prime\zeta^\prime) \times (\zeta^\prime_{\beta^\prime}\partial_{\alpha^\prime} - \zeta^\prime_{\alpha^\prime}\partial_{\beta^\prime})\znew^\prime\kvec \, d\alpha^\prime d\beta^\prime \notag
\end{align}
\begin{equation}\label{LambdaFromZMultiscale}
\lambda = (I + \nht)\znew\kvec - \mathcal{K}\znew\kvec
\end{equation}
\begin{equation}\label{BFormulaMultiscale}
(I - \nht)b = -[D_t, \nht](I + \nht)\znew\kvec + (I - \nht)D_t\mathcal{K}\znew\kvec
\end{equation}
\begin{align}\label{AFormulaMultiscale}
(I - \mathcal{K})\mathcal{A} & = \Bigl\{\kvec + [D_t, \nht]D_t\zeta + [\mathcal{A}(\mathcal{N} \times \nabla), \nht] (I + \nht)\znew\kvec  \\
& \qquad + (I - \nht)\left(-\mathcal{A}\zeta_\beta \times (\partial_\alpha\mathcal{K}\znew\kvec) + \mathcal{A}\zeta_\alpha \times (\partial_\beta\mathcal{K}\znew\kvec) + \mathcal{A}(\lambda_\alpha \times \lambda_\beta)\right)\Bigr\}_3 \notag
\end{align} 
which will allow us to successively solve for the asymptotic expansions of the quantities $\znew$, $\lambda$, $b$ and $\mathcal{A}$, which we denote by
\begin{equation*}
\lambda \sim \sum_{j = 1}^\infty \epsilon^j \lambda^{(j)} \qquad b \sim \sum_{j = 1}^\infty \epsilon^j b^{(j)} \qquad \mathcal{A} \sim \sum_{j = 0}^\infty \epsilon^j \mathcal{A}^{(j)}
\end{equation*}  We also express the multiscale expansion of the operator $\mathcal{P} = D_t^2 - \mathcal{A}(\zeta_\beta \partial_\alpha - \zeta_\alpha \partial_\beta)$ by
\begin{equation*}
\mathcal{P} \sim \sum_{j = 0}^\infty \epsilon^j \mathcal{P}^{(j)}
\end{equation*}

First, observe that it follows immediately from \eqref{AFormulaMultiscale} that $\mathcal{A}^{(0)} = 1$.  We also have $\mathcal{P}^{(0)} = \partial_{t_0}^2 - \jvec\partial_{\alpha_0}$, since there is no dependence on $\beta_0$.  Hence, the $O(\epsilon)$ terms of \eqref{WaterWaveEvolutionMultiscale} give the equation $$\mathcal{P}^{(0)}(I - \nht^{(0)})\znew^{(1)}\kvec = 0$$ This equation admits the solution $\znew^{(1)} = \Im(Ae^{\jvec\phi})$ for $\phi = k\alpha_0 + \omega t_0$ satisfying the dispersion relation $\omega^2 = k$.  Since the kernel of $\nht_0$ has no scalar part, $\mathcal{K}^{(0)} = 0$, and so we find from the $O(\epsilon)$ terms of \eqref{LambdaFromZMultiscale} that 
\begin{equation}
\lambda^{(1)} = (I + \nht^{(0)})\znew^{(1)}\kvec = \ivec Ae^{\jvec\phi}
\end{equation}
as we expected.  Notice that the right hand sides of \eqref{BFormulaMultiscale} and \eqref{AFormulaMultiscale} have no first order terms, and so $b^{(1)} = 0$ and $\mathcal{A}^{(1)} = 0$.  Using the identity $\mathcal{P}^{(1)} = 2\partial_{t_0}\partial_{t_1} - \jvec\partial_{\alpha_1} + \ivec\partial_{\beta_1}$, the $O(\epsilon^2)$ terms of \eqref{WaterWaveEvolutionMultiscale} give the following equation:
\begin{align*}
\mathcal{P}^{(0)}(I - \nht^{(0)})\znew^{(2)}\kvec & = \mathcal{P}^{(0)}\nht^{(1)}\znew^{(1)}\kvec + [\partial_{t_0}, \nht^{(1)}_1]\partial_{t_0}(\zeta^{(1)})^\dagger - \mathcal{P}^{(1)}(I - \nht^{(0)})\znew^{(1)}\kvec \\
& = A_{\beta_1}e^{\jvec\phi} + \left(2\jvec\omega(A_{t_1} - \omega^\prime A_{\alpha_1})e^{\jvec\phi}\ivec - A_{\beta_1}e^{\jvec\phi} \right) \\
& = 2\jvec\omega(A_{t_1} - \omega^\prime A_{\alpha_1})e^{\jvec\phi}\ivec
\end{align*}  To suppress secular terms, we choose $A = A(\alpha_1 + \omega^\prime t_1, \beta_1, t_2) = A(X, Y, T)$.  Then we solve \eqref{WaterWaveEvolutionMultiscale} to the order $O(\epsilon^2)$ by taking $\znew^{(2)} = \Im(Be^{\jvec\phi}) + M^{(2)}$, where $B$ and $M^{(2)}$ are $1, \jvec$-valued and scalar-valued functions of $X, Y, T$, respectively, to be determined later.  From the formula $$\nht^{(1)}\znew^{(1)}\kvec = \frac{1}{2}k(I - \nht_0)|A|^2\kvec + \frac{\jvec}{k}\Im(A_Ye^{\jvec\phi})$$ and the fact that $(\nht^{(1)} - \mathcal{K}^{(1)})\znew^{(1)}\kvec$ is just $\nht^{(1)}\znew^{(1)}\kvec$ without its $\kvec$-component, we calculate that
\begin{align*}
\lambda^{(2)} & = (I + \nht^{(0)})\znew^{(2)}\kvec + \nht^{(1)}\znew^{(1)}\kvec - \mathcal{K}^{(1)}\znew^{(1)}\kvec \\
& = \ivec Be^{\jvec\phi} + (I + \nht_0)M^{(2)}\kvec - \frac{1}{2}k\nht_0(|A|^2\kvec) + \frac{\jvec}{k}\Im(A_Ye^{\jvec\phi})
\end{align*}  Next, we find that
the $O(\epsilon^2)$ terms of \eqref{BFormulaMultiscale} yield the condition
\begin{align*}
(I - \nht^{(0)})b^{(2)} & = -[\partial_{t_0}, \nht_1^{(1)}](I + \nht^{(0)})\znew^{(1)}\kvec \\
& \quad + (I - \nht^{(0)})\partial_{t_0}\mathcal{K}^{(1)}\znew^{(1)}\kvec \\
& = (I - \nht_0)(-k\omega|A|^2\ivec)
\end{align*}  The fact that $b^{(2)}$ is $\ivec, \jvec$-valued now forces the choice $b^{(2)} = -k\omega|A|^2\ivec$.

Finally we calculate $\mathcal{A}^{(2)}$ from \eqref{AFormulaMultiscale}:
\begin{align*}
\mathcal{A}^{(2)} & = \Bigl\{[\partial_{t_0}, \nht^{(1)}_1]\partial_{t_0}\lambda^{(1)} + [\jvec\partial_{\alpha_0}, \nht^{(1)}](I + \nht^{(0)})\znew^{(1)}\kvec \\
& \quad + [\jvec\partial_{\alpha_1} - \ivec\partial_{\beta_1}, \nht^{(0)}](I + \nht^{(0)})\znew^{(1)}\kvec + (I - \nht^{(0)})(-\jvec \times \partial_{\alpha_0}\mathcal{K}^{(1)}\znew^{(1)}\kvec)\Bigr\}_3 \\
& = \bigl\{ k^2(I - \nht_0)|A|^2\kvec  - k^2(I - \nht_0)|A|^2\kvec - 2A_Y e^{\jvec\phi} \\
& \quad + 2A_Y e^{\jvec\phi} + 0 \Bigr\}_3 \\
& = 0.
\end{align*}

We now collect the $O(\epsilon^3)$ terms, beginning with those contributed from \eqref{WaterWaveEvolutionMultiscale}.  We record the following useful formulas derived through Proposition \ref{HilbertTransformExpansionFormulas} for calculating terms below involving $\nht^{(2)}$:

\begin{align}\label{H2FormulaWaveNumber-1}
(\nht^{(2)}_1 + \nht^{(2)}_2) \overline{F}e^{-\jvec\phi} & = -k^2 A^2 \overline{F} e^{\jvec\phi} + (I - \nht_0)\left(A\overline{F}_X + \frac{1}{2}\overline{A}F_Y\kvec - kB\overline{F}\jvec\right)
\end{align}
\begin{align}\label{H2FormulaWaveNumber1}
(\nht^{(2)}_1 + \nht^{(2)}_2) Fe^{\jvec\phi} & = -\frac{1}{2}(I - \nht_0)A\overline{F}_Y\kvec + \frac{1}{2}\jvec(\overline{A}\,\overline{F}_Y - \overline{A}_Y\overline{F})e^{-2\jvec\phi}\ivec\end{align}

Note that $$\mathcal{P}^{(2)} = 2\partial_{t_0}\partial_{t_2} + \partial_{t_1}^2 + 2b_1^{(2)}\partial_{\alpha_0}\partial_{t_0} - \lambda^{(1)}_{\beta_1}\partial_{\alpha_0} + \lambda^{(1)}_{\alpha_0}\partial_{\beta_1}$$ The $O(\epsilon^3)$ terms contributed from \eqref{WaterWaveEvolutionMultiscale} now give
\begin{align*}
\mathcal{P}^{(0)}(I - \nht^{(0)})\znew^{(3)}\kvec & = \mathcal{P}^{(0)}\nht^{(1)}\znew^{(2)}\kvec \\
& \quad + \mathcal{P}^{(0)}\nht^{(2)}\znew^{(1)} \kvec \\
& \quad - \mathcal{P}^{(1)}(I - \nht^{(0)})\znew^{(2)} \kvec \\
& \quad + \mathcal{P}^{(1)}\nht^{(1)}\znew^{(1)} \kvec \\
& \quad - \mathcal{P}^{(2)}(I - \nht^{(0)})\znew^{(1)}\kvec \\
& \quad + [\partial_{t_0}, \nht^{(1)}_1]\partial_{t_0}(\lambda^{(2)})^\dagger \\
& \quad + [\partial_{t_0}, \nht^{(1)}_1]\partial_{t_1}(\lambda^{(1)})^\dagger \\
& \quad + [\partial_{t_0}, \nht^{(1)}_1]b^{(2)} \\
& \quad + [\partial_{t_0}, \nht^{(2)}_1 + \nht^{(2)}_2]\partial_{t_0}(\lambda^{(1)})^\dagger \\
& \quad + [\partial_{t_1}, \nht^{(1)}_1]\partial_{t_0}(\lambda^{(1)})^\dagger \\
& \quad + [b_1^{(2)} \partial_{\alpha_0}, \nht_0]\partial_{t_0}(\lambda^{(1)})^\dagger
\end{align*}
\begin{align*}
& - \frac{1}{2\pi^2} \iint  \left(\frac{\lambda^{(1)}_{t_0} - (\lambda^{(1)}_{t_0})^\prime}{|P - P^\prime|^3}
  - 3 \frac{P - P^\prime}{|P - P^\prime|^3}  \frac{(P - P^\prime) \cdot (\lambda^{(1)}_{t_0} - (\lambda^{(1)}_{t_0})^\prime)}{|P - P^\prime|^2}\right) \\
  & \hspace{9cm} (\lambda^{(1)}_{t_0} - (\lambda^{(1)}_{t_0})^\prime) \times (\jvec (\znew^{(1)}_{\alpha_0})^\prime)\kvec \, dP^\prime
\end{align*}
\begin{equation*}
= I_1 + I_2 + I_3 + \cdots + I_{12}
\end{equation*}

First, it is quick to see that $I_7 = I_8 = I_{10} = I_{11} = 0$.  For the rest of the terms we calculate:
\begin{align*}
I_1 & = \overline{B}_Y e^{-\jvec\phi}
\end{align*}
\begin{align*}
I_2 & = -\frac{1}{2k}\overline{A}_{YY}e^{-\jvec\phi}\ivec - \frac{1}{k}\overline{A}_{XY}e^{-\jvec\phi}\jvec
\end{align*}
\begin{align*}
I_3 & = -\overline{B}_Ye^{-\jvec\phi} - (I - \nht_0)|\mathcal{D}|M^{(2)}\kvec
\end{align*}
\begin{align*}
I_4 & = -\frac{1}{2k}A_{YY}e^{\jvec\phi}\ivec + \frac{1}{2k}\overline{A}_{YY}e^{-\jvec\phi}\ivec + \frac{1}{k}\jvec\overline{A}_{XY}e^{-\jvec\phi} + \frac{k}{2}(I - \nht_0)|\mathcal{D}|(|A|^2\kvec)
\end{align*}
\begin{align*}
I_5 & = \omega\left(2\jvec A_T - \omega^{\prime\prime} A_{XX} + 2k^2\omega A|A|^2 \right)e^{\jvec\phi}\ivec
\end{align*}
\begin{align*}
I_6 & = \frac12 k (I - \nht_0)A\overline{A}_Y\jvec
\end{align*}
\begin{align*}
I_9 & = -\frac12 k (I - \nht_0)A\overline{A}_Y\jvec
\end{align*}
Simplifying $I_{12}$ by writing $(\lambda^{(1)}_{t_0} - (\lambda^{(1)}_{t_0})^\prime) \times (\jvec (\znew^{(1)}_{\alpha_0})^\prime)\kvec = (\lambda^{(1)}_{t_0} - (\lambda^{(1)}_{t_0})^\prime) (\znew^{(1)}_{\alpha_0})^\prime\ivec$ and calculating as in Appendix A gives
\begin{align*}
I_{12} & = [\lambda_{t_0}^{(1)}, [\lambda_{t_0}^{(1)}, |\mathcal{D}^{(0)}|]]\znew^{(1)}_{\alpha_0}\ivec + [\xnew^{(1)}_{t_0}, [\xnew^{(1)}_{t_0}, 2\ivec|\mathcal{D}^{(0)}|]]\znew_{\alpha_0}^{(1)} - [\znew^{(1)}_{t_0}, [\xnew^{(1)}_{t_0}, 2\ivec|\mathcal{D}^{(0)}|]]\znew_{\alpha_0}^{(1)}\jvec \\
& =[\lambda_{t_0}^{(1)}, [\lambda_{t_0}^{(1)}, |\mathcal{D}^{(0)}|]]\znew^{(1)}_{\alpha_0}\ivec + [\xnew^{(1)}_{t_0} + \jvec \znew^{(1)}_{t_0} , [\xnew^{(1)}_{t_0}, 2|\mathcal{D}^{(0)}|]]\znew_{\alpha_0}^{(1)}\ivec \\
& = -\frac12 k^2 [(\xnew^{(1)} + \jvec\xnew^{(1)}), [(\xnew^{(1)} + \jvec\xnew^{(1)}), |D^{(0)}|]](\xnew^{(1)} - \jvec\xnew^{(1)})\ivec \\
& = -k^3A|A|^2e^{\jvec\phi}\ivec
\end{align*}
and so
\begin{align*}
\mathcal{P}^{(0)}(I - \nht^{(0})\znew^{(3)}\kvec & = (I - \nht_0)\left(-|\mathcal{D}|M^{(2)}\kvec + \frac{k}{2} |\mathcal{D}|(|A|^2\kvec)\right) \\
& \quad + \omega\left(2\jvec A_T - \omega^{\prime\prime} A_{XX} + 2\omega^{\prime\prime}A_{YY} + k^2\omega A|A|^2 \right)e^{\jvec\phi}\ivec
\end{align*}

In order to suppress secular terms in $\znew^{(3)}$, we choose $M^{(2)} = \frac{k}{2}|A|^2$ and insist that $A$ satisfy the cubic hyperbolic nonlinear Schr\"odinger equation:
\begin{equation}\label{HNLSQuaternion}
2\jvec A_T - \omega^{\prime\prime}A_{XX} + 2\omega^{\prime\prime}A_{YY} + k^2\omega A|A|^2 = 0
\end{equation}  With these choices, we update the formulas for $\znew^{(2)}$ and $\lambda^{(2)}$:
\begin{equation}\label{Z2Formula}
\znew^{(2)} = \Im(Be^{\jvec\phi}) + \frac{1}{2}k|A|^2
\end{equation}
\begin{align}\label{Lambda2Formula}
\lambda^{(2)} & = \ivec Be^{\jvec\phi} + \frac{1}{2}k(I + \nht_0)(|A|^2\kvec) - \frac{1}{2}k\nht_0(|A|^2\kvec) + \frac{\jvec}{k}\Im(A_Ye^{\jvec\phi}) \notag \\
& = \ivec Be^{\jvec\phi} + \frac{1}{2}k|A|^2\kvec + \frac{\jvec}{k}\Im(A_Ye^{\jvec\phi})
\end{align} where at this stage in the calculation $B$ remains to be determined.  With these choices made, we can solve the equation for $\znew^{(3)}$ by choosing $\znew^{(3)} = M^{(3)}$ for some scalar-valued function $M^{(3)}$ of slow variables, again to be determined later.  With this choice, we can calculate $\lambda^{(3)}$ as follows:
\begin{align}\label{Lambda3Formula}
\lambda^{(3)} & = (I + \nht^{(0)})\znew^{(3)}\kvec + (\nht^{(1)} - \mathcal{K}^{(1)})\znew^{(2)}\kvec + (\nht^{(2)} - \mathcal{K}^{(2)})\znew^{(1)}\kvec \notag \\
& = (I + \nht_0)M^{(3)}\kvec  \\
& \quad + \frac{\jvec}{k}\Im(B_Ye^{\jvec\phi}) + \frac{1}{2}(I - \nht_0)k\nht_0(A\overline{B})\kvec - \left\{\frac{1}{2}(I - \nht_0)k\nht_0(A\overline{B})\kvec\right\}_3\kvec \notag \\
& \quad + \frac{1}{2k^2}\Re(A_{YY}e^{\jvec\phi})\ivec - \frac{1}{k^2}\Im(\jvec A_{XY}e^{\jvec\phi})\jvec \notag \\
& \quad + \left(-\frac{k^2}{2}A|A|^2e^{\jvec\phi}\ivec + (I - \nht_0)\left(\frac{1}{2}A\overline{A}_X\ivec + \frac{1}{4}(A\overline{A}_Y + \overline{A}A_Y)\jvec + \frac{1}{2}kB\overline{A}\kvec \right)\right) \notag \\
& \quad - \left\{-\frac{k^2}{2}A|A|^2e^{\jvec\phi}\ivec + (I - \nht_0)\left(\frac{1}{2}A\overline{A}_X\ivec + \frac{1}{4}(A\overline{A}_Y + \overline{A}A_Y)\jvec + \frac{1}{2}kB\overline{A}\kvec \right)\right\}_3\kvec \notag
\end{align}  

Note that the scalar part of the right hand side of \eqref{Lambda3Formula} is $\frac12\mathcal{R}_2\partial_X|A|^2 - \frac12\mathcal{R}_1\partial_Y|A|^2 = 0$.  We continue by collecting the $O(\epsilon^3)$ terms from \eqref{BFormulaMultiscale}, which yields the condition
\begin{align*}
(I - \nht_0)b^{(3)} & = \nht^{(1)}b^{(2)}  - [b^{(2)}\partial_{\alpha_0}, \nht^{(0)}](I + \nht^{(0)})\znew^{(1)}\kvec \\
& \quad - [\partial_{t_1}, \nht^{(1)}_1](I + \nht^{(0)})\znew^{(1)}\kvec  - [\partial_{t_0}, \nht^{(1)}_1](I + \nht^{(0)})\znew^{(2)}\kvec \\
& \quad - [\partial_{t_0}, \nht^{(1)}_1]\nht^{(1)}\znew^{(1)}\kvec  - [\partial_{t_0}, \nht^{(2)}_1 + \nht^{(2)}_2](I + \nht^{(0)})\znew^{(1)}\kvec \\
& \quad + (I - \nht^{(0)})\partial_{t_0}\mathcal{K}^{(2)}\znew^{(1)}\kvec  + (I - \nht^{(0)})\partial_{t_0}\mathcal{K}^{(1)}\znew^{(2)}\kvec \\
& \quad + (I - \nht^{(0)})\partial_{t_1}\mathcal{K}^{(1)}\znew^{(1)}\kvec - \nht^{(1)}\partial_{t_0}\mathcal{K}^{(1)}\znew^{(1)}\kvec
\end{align*}
Using the formula $$(I - \nht_0)(S_0 + S_3\kvec) = (I - \nht_0)(\ivec(\mathcal{R}_1S_3 - \mathcal{R}_2S_0) + \jvec(\mathcal{R}_1S_0 + \mathcal{R}_2S_3))$$ for scalar-valued $S_0 S_3$, we can solve for $b^{(3)}$ as a $\ivec, \jvec$-valued function.  This gives us consistency for the formula for $b$ up to residual terms of physical size $O(\epsilon^4)$.
Similarly, since the formula for $\mathcal{A}$ requires extracting the $\kvec$-component of some expression, we may always find a formula for $\mathcal{A}^{(3)}$ that makes \eqref{AFormulaMultiscale} consistent to terms of size $O(\epsilon^4)$.  

We can now consider the $O(\epsilon^4)$ contributions from \eqref{WaterWaveEvolutionMultiscale}.  An explicit calculation of $\znew^{(4)}$ would be taxing on the author and reader alike.  Luckily, the precise form of $\znew^{(4)}$ is irrelevant: we need only show that we can find some choice of correctors that yields a residual of physical size $O(\epsilon^5)$.  Thus we instead give a general argument that the structure of $G^{(4)}$ is such that we can always solve for $\znew^{(4)}$.

In order to further analyze the terms of size $O(\epsilon^4)$, we appeal to the idea of \textbf{total order} and \textbf{total phase}.  The idea is that the power of $\epsilon$ and the multiple of $\jvec\phi$ in the phase of any term in the formal calculation can be read off from the function of slow variables in each term.  We would like to formally describe maps $\mathfrak{o}$ and $\mathfrak{p}$ from the set of terms that can appear in the formal calculation into the integers with the property that all terms appearing in the formal calculation are of the form
\begin{equation}\label{OrderPhaseProperty}
\epsilon^{\mathfrak{o}(F)}Fe^{\mathfrak{p}(F)\jvec\phi}
\end{equation}

To describe this formally, consider the monoid\footnote{A \textbf{monoid} is an algebraic structure with an associative multiplication and an identity element.} of functions $\mathfrak{A}$ with pointwise multiplication, and generated by the functions $T_1 T_2 \ldots T_n 1$ where each $T_j$ is either the operator of multiplication by $A$, $B$, $\overline{A}$, $\overline{B}$ or one of the operators $\partial_X, \partial_Y, \partial_T, \mathcal{R}_1, \mathcal{R}_2$.  Every term appearing in the formal calculation can be expressed as a member of this monoid.  The total order $\mathfrak{o} : \mathfrak{A} \to (\mathbb{N}, +)$ and the total phase $\mathfrak{p} : \mathfrak{A} \to (\mathbb{Z}, +)$ are respectively monoid homomorphisms\footnote{That is $\mathfrak{o}$ and $\mathfrak{p}$ send pointwise products of functions to the sums of the images of their factors.} satisfying $\mathfrak{o}(T_1 T_2 \ldots T_n 1) = \prod_{j = 1}^n \mathfrak{o}(T_j)$ and $\mathfrak{p}(T_1 T_2 \ldots T_n 1) = \prod_{j = 1}^n \mathfrak{p}(T_j)$, where $$\mathfrak{o}(T) = \begin{cases} 0, & T = \mathcal{R}_1, \mathcal{R}_2 \\ 1, & T = A, \overline{A}, \partial_X, \partial_Y \\ 2, & T = B, \overline{B}, \partial_T \end{cases} \qquad \qquad \mathfrak{p}(T) = \begin{cases} -1, & T = \overline{A}, \overline{B} \\ 0, & T = \mathcal{R}_1, \mathcal{R}_2, \partial_X, \partial_Y, \partial_T \\ 1, & T = A, B \end{cases}$$  It is tedious but straightforward to verify that \eqref{OrderPhaseProperty} holds by construction for all terms up to $O(\epsilon^4)$; we leave the proof to the reader.

We are now ready to analyze the $O(\epsilon^4)$ terms.

\begin{lemma}\label{AbstractEpsilon4}
Let $\mathcal{P}^{(0)}(I - \nht^{(0)})\znew^{(4)}\kvec = G^{(4)}$ be as above.  Then $G^{(4)}$ is of the form
\begin{align}\label{Epsilon4TermsSummary}
& \left(2\jvec B_T - \omega^{\prime\prime}B_{XX} + 2\omega^{\prime\prime}B_{YY}\right)e^{\jvec\phi}\ivec + (I - \nht_0)\Bigl((F_3 - |\mathcal{D}|M^{(3)})\kvec + F_0\Bigr) \notag \\
& + \sum_{0 < |n| \leq 3} S_ne^{n\jvec\phi}\ivec + \sum_{0 < |m| \leq 3} S^\prime_m e^{m\jvec\phi} + \frac12 ( I + \nht_0)F
\end{align}
where the $S_m, S^\prime_n, F_j$ $1, \jvec$-valued are sums of members of $\mathfrak{A}$ with total order at most four and further satisfying:
\begin{itemize}
\item[(a)]{The $F_j$ are scalar-valued for $j = 0, 3$.}
\item[(b)]{$F_3$ is independent of $M^{(3)}$ and is a sum of terms in $\mathfrak{A}$ having the form either (i) a pure derivative in $\partial_X$, $\partial_Y$ or (ii) a product of two or more factors of $A$, $B$, $\overline{A}$, $\overline{B}$ and their derivatives in $X, Y$ and Riesz transforms of such products.}
\item[(c)]{$S_1$ does not depend on derivatives of $B$.}
\end{itemize}
Hence we may always choose $M = |\mathcal{D}|^{-1}F_3$ and $B$ so that
\begin{equation}\label{AuxiliaryEquationVector}
2\jvec B_T - \omega^{\prime\prime}B_{XX} + 2\omega^{\prime\prime}B_{YY} = -S_1
\end{equation}
and $\znew^{(4)}$ so that only terms of the form
\begin{equation}\label{BigButOKResidualPart}
(I - \nht_0)F_0 + \frac12 (I + \nht_0)F + \sum_{-3 \leq n < 0} S_n e^{n\jvec\phi}\ivec + \sum_{0 < |m| \leq 3} S^\prime_m e^{m\jvec\phi}
\end{equation}
remain.
\end{lemma}

\begin{proof}
Consider the equation $$\mathcal{P}^{(0)}(I - \nht^{(0)})\znew^{(4)}\kvec = \sum_{j_1 + j_2 + j_3 = 4}^{j_3 \neq 4} -\mathcal{P}^{(j_1)}(I - \nht)^{(j_2)}\znew^{(j_3)}\kvec := G^{(4)}$$ where in particular $G^{(4)}$ is calculated through $b^{(3)}$ and $\mathcal{A}^{(3)}$.  Notice that we can always replace $A_T$ everywhere with only $A$ and its derivatives in $X$ and $Y$ using \eqref{HNLSQuaternion}.  Observe first that the only place that $\partial_T$ and $\partial_X^2$, $\partial_Y^2$ act on $B$ is through the term $- \mathcal{P}^{(2)}(I - \nht^{(0)})\znew^{(2)}\kvec$.  Simplifying this term contributes $\left(2\jvec B_T - \omega^{\prime\prime}B_{XX} + 2\omega^{\prime\prime}B_{YY}\right)e^{\jvec\phi}\ivec$ which we have isolated.  Similarly, since $M^{(3)}$ appears only in $\znew^{(3)}$ and is a function of slow variables alone, the only place it appears is through the term $-\mathcal{P}^{(1)}(I - \nht^{(0)})\znew^{(3)}\kvec$, which when calculated simplifies to $-(I - \nht_0)|\mathcal{D}|M^{(3)}\kvec$.  Denote the rest of $G^{(4)}$ by
$$G^{(4)}_0 = G^{(4)} - \left(2\jvec B_T - \omega^{\prime\prime}B_{XX} + 2\omega^{\prime\prime}B_{YY}\right)e^{\jvec\phi}\ivec + (I - \nht_0)|\mathcal{D}|M\kvec$$

Next, since the operators appearing in $G^{(4)}_0$ are concatenations of either multiplications by wave packets of the form $Fe^{n\jvec\phi}$ for $n = -1, 0, 1$, multiplication by $\ivec$, or operating by the flat Hilbert transform $\nht_0 = -\jvec\mathcal{R}_1 + \ivec\mathcal{R}_2$, it follows that $G^{(4)}_0$ is of the form
\begin{equation*}
F  + \sum_{0 < |n| \leq 3} S_ne^{n\jvec\phi}\ivec + \sum_{0 < |m| \leq 3} S^\prime_m e^{m\jvec\phi}
\end{equation*}
where the $F$, $S_n$, $S_m^\prime$ are functions of the slow variables alone.  Observe that for scalar-valued functions $F, G$ we have the identity $$(I - \nht_0)(F\ivec + G\jvec) = (I - \nht_0)\Bigl((-\mathcal{R}_1F - \mathcal{R}_2G)\kvec + (\mathcal{R}_1F - \mathcal{R}_2G)\Bigr)$$  This allows us to write $(I - \nht_0)F = (I - \nht_0)(F_0 + F_3\kvec)$ as above, which gives us (a).  In particular, the terms $\sum_{n = 2}^3 S_n e^{n\jvec\phi}\ivec$ can be accounted for by an appropriate choice of $\znew^{(4)}$.

Now all terms of the form $Fe^{m\jvec\phi}$ appearing in $G^{(4)}_0$ satisfy $\mathfrak{o}(F) = 4$; those terms comprising $S_1$ also satisfy $\mathfrak{p}(F) = 1$.  Showing (b) reduces to ruling out the existence of a term in $F_3$ having only one occurrence of $A$, $B$, $\overline{A}$, or $\overline{B}$.  But the only such members of $\mathfrak{A}$ possible are $A$, $B$, $\overline{A}$, or $\overline{B}$ operated on by finitely many Riesz transforms $\mathcal{R}_1, \mathcal{R}_2$, and all such terms have $\mathfrak{o}(F) \leq 2$.  Similarly, to show (c), note that (up to Riesz transforms) if a term $F$ has a derivative of $B$ as one of its factors, its other factors could have total order at most 1.  However, up to Riesz transforms this forces the other factor to be $A$ or $\overline{A}$.  But then this term cannot appear in $S_1$ since then $\mathfrak{p}(F)$ would be even.
\end{proof}

\begin{remark}\label{ResidualFormallyEpsilon4}
No possible choice of $M^{(3)}$, $B$, and $\znew^{(4)}$ can account for the terms \eqref{BigButOKResidualPart}.  In this sense, \eqref{WaterWaveEvolutionMultiscale} is only formally consistent up to a residual of physical size $O(\epsilon^4)$.  However, in the process of deriving the energy inequality, we will show using an almost-orthogonality argument that these terms do not spoil the energy estimates.
\end{remark}

In summary, there exists some choice of $M^{(3)}$, $B$ and $\znew^{(4)}$ as in Lemma \ref{AbstractEpsilon4} so that if we take as our approximate solution to the system \eqref{WaterWaveEvolutionMultiscale}-\eqref{AFormulaMultiscale} as the following:
\begin{align}\label{TildeLambdaFormula}
\tilde{\lambda} & = \epsilon \ivec Ae^{\jvec\phi} + \epsilon^2\left(\ivec Be^{\jvec\phi} + \frac{1}{2}k|A|^2\kvec - \frac{\jvec}{k}\Im(A_Ye^{\jvec\phi})\right) + \epsilon^3\lambda^{(3)} + \epsilon^4\znew^{(4)}\kvec
\end{align}
where $\lambda^{(3)}$ is given by \eqref{Lambda3Formula} and $\znew^{(4)}$ is given as in Lemma \ref{AbstractEpsilon4},
\begin{align}\label{TildeBFormula}
\tilde{b} & = \epsilon^2(-k\omega|A|^2\ivec) + \epsilon^3 b^{(3)} + O(\epsilon^4)
\end{align}
where $b^{(3)} = (\sum_{|i| \leq 3} S_i e^{i\jvec\phi}) + (\sum_{|i| \leq 3} S_i e^{i\jvec\phi})^\dagger$ for some functions $S_i$ of slow variables alone,
\begin{align}\label{TildeAFormula}
\tilde{\mathcal{A}} & = 1 + \epsilon^3\mathcal{A}^{(3)} + O(\epsilon^4)
\end{align}
where $\mathcal{A}^{(3)} = \Re \sum_{|i| \leq 3} F_i e^{i\jvec\phi}$ for functions $F_i$ of slow variables alone.  Further define
\begin{equation}
\tilde{D}_t = \partial_t + (\tilde{b} \cdot \mathcal{D}), \qquad \tilde{\mathcal{P}} = \tilde{D}_t^2 - \tilde{\zeta}_\beta\partial_\alpha + \tilde{\zeta}_\alpha\partial_\beta
\end{equation}
then the equations \eqref{WaterWaveEvolutionMultiscale}-\eqref{AFormulaMultiscale} are satisfied up to the following residuals:
\begin{align}\label{WaterWaveEvolutionApprox}
\tilde{\mathcal{P}}(I - \tilde{\nht})\tilde{\znew}\kvec & = [\tilde{D}_t, \tilde{\nht}]\tilde{D}_t\tilde{\zeta}^\dagger \notag \\
& + \iint K(\tilde{\zeta} - \tilde{\zeta}^\prime)(\tilde{D}_t \tilde{\zeta} - \tilde{D}_t^\prime \tilde{\zeta}^\prime) \times (\partial_\beta^\prime \tilde{D}_t^\prime \tilde{\zeta}^\prime \partial_{\alpha^\prime} - \partial_\alpha^\prime \tilde{D}_t^\prime \tilde{\zeta}^\prime\partial_{\beta^\prime})\tilde{\znew}^\prime\kvec \, d\alpha^\prime, d\beta^\prime \\
& + \iint \tilde{D}_t K(\tilde{\zeta} - \tilde{\zeta}^\prime)(\tilde{D}_t\tilde{\zeta} - \tilde{D}_t^\prime\tilde{\zeta}^\prime) \times (\tilde{\zeta}^\prime_{\beta^\prime}\partial_{\alpha^\prime} - \tilde{\zeta}^\prime_{\alpha^\prime}\partial_{\beta^\prime})\tilde{\znew}^\prime\kvec \, d\alpha^\prime, d\beta^\prime \notag \\
& + \epsilon^4\left((I - \nht_0)F_0 + \sum_{m = 1}^2 S^\prime_m e^{m\jvec\phi}\right) + O(\epsilon^5) \notag
\end{align}
where $F_0$ is a scalar-valued function of slow variables alone, and the $S_m^\prime$ are $1, \jvec$-valued functions of slow variables alone,
\begin{equation}\label{LambdaFromZApprox}
\tilde{\lambda} = (I + \tilde{\nht})\tilde{\znew}\kvec - \tilde{\mathcal{K}}\tilde{\znew}\kvec + O(\epsilon^4)
\end{equation}
\begin{equation}\label{BFormulaApprox}
(I - \tilde{\nht})\tilde{b} = -[\tilde{D}_t, \tilde{\nht}](I + \tilde{\nht})\tilde{\znew}\kvec + (I - \tilde{\nht})\tilde{D}_t\tilde{\mathcal{K}}\tilde{\znew}\kvec + O(\epsilon^4)
\end{equation}
\begin{align}\label{AFormulaApprox}
(I - \tilde{\mathcal{K}})\tilde{\mathcal{A}} & = \Bigl\{1 + [\tilde{D}_t, \tilde{\nht}]\tilde{D}_t\tilde{\zeta} + [\tilde{\mathcal{A}}(\tilde{\mathcal{N}} \times \nabla), \tilde{\nht}] (I + \tilde{\nht})\tilde{\znew}\kvec  \notag \\
& \qquad + (I - \tilde{\nht})\left(-\tilde{\mathcal{A}}\tilde{\zeta}_\beta \times (\partial_\alpha\tilde{\mathcal{K}}\tilde{\znew}\kvec) + \tilde{\mathcal{A}}\tilde{\zeta}_\alpha \times (\partial_\beta\tilde{\mathcal{K}}\tilde{\znew}\kvec) + \tilde{\mathcal{A}}(\tilde{\lambda}_\alpha \times \tilde{\lambda}_\beta)\right)\Bigr\}_3 \\
& \qquad\qquad + O(\epsilon^4)\notag
\end{align}
The largest number of derivatives falling on $A$ and $B$ in the above residuals occurs in the residual of \eqref{WaterWaveEvolutionApprox} and is $9$ and $8$, respectively.  Denote the right hand side of \eqref{WaterWaveEvolutionApprox} by $\tilde{G}$.  If we operate on \eqref{WaterWaveEvolutionApprox} by $\tilde{D}_t$ then we also have
\begin{equation}\label{DtWaterWaveEvolutionApprox}
\tilde{\mathcal{P}}\tilde{D}_t(I -\tilde{\nht})\tilde{\znew}\kvec =
[\tilde{\mathcal{P}}, \tilde{D}_t](I - \tilde{\nht})\tilde{\znew}\kvec + \tilde{D}_t\tilde{\mathcal{G}} + O(\epsilon^5)
\end{equation}
The largest number of derivatives falling on $A$ or $B$ in the residual of \eqref{DtWaterWaveEvolutionApprox} is hence $11$ and $10$, respectively.  As in \eqref{NewEvolutionEquationsDt}, we rewrite the commutator $[\tilde{D}_t, \tilde{\mathcal{P}}]$ by changing variables.  Introduce the approximate change of variables $\tilde{\kappa}$ as the solution to the following ODE:
$$\begin{cases} \tilde{\kappa}_t = \tilde{b} \circ \tilde{\kappa} \\ \tilde{\kappa}(0) = \alpha \ivec + \beta \jvec \end{cases}$$  Then by construction we have $\tilde{D}_tU_{\tilde{\kappa}}^{-1} = U_{\tilde{\kappa}}^{-1}\partial_t$ and hence

\begin{equation}\label{CommutatorDtWithP}
[\tilde{D}_t, \tilde{\mathcal{P}}] = \tilde{\mathcal{A}}U_{\tilde{\kappa}}^{-1}\left(\frac{\tilde{\mathfrak{a}}_t}{\tilde{\mathfrak{a}}}\right)(\tilde{\mathcal{N}} \times \nabla) + (\partial_\beta \tilde{D}_t\tilde{\zeta} \partial_\alpha - \partial_\alpha \tilde{D}_t\tilde{\zeta} \partial_\beta)
\end{equation}

Finally, we record the formula

\begin{align}\label{TildeatOveraFormula}
U_{\tilde{\kappa}}^{-1}\left(\frac{\tilde{\mathfrak{a}}_t}{\tilde{\mathfrak{a}}}\right)\kvec & = \frac{\tilde{D}_t\tilde{\mathcal{A}}}{\tilde{\mathcal{A}}}\kvec - U_{\tilde{\kappa}}^{-1}\left(\frac{\partial_t J(\tilde{\kappa})\kvec}{J(\tilde{\kappa})}\right) \notag \\
& = \frac{\tilde{D}_t\tilde{\mathcal{A}}}{\tilde{\mathcal{A}}}\kvec - U_{\tilde{\kappa}}^{-1}\left(\frac{-\tilde{\kappa}_\beta \times \tilde{\kappa}_{t\alpha} + \tilde{\kappa}_\alpha \times \tilde{\kappa}_{t\beta}}{J(\tilde{\kappa})}\right)  \\
& = \frac{\tilde{D}_t\tilde{\mathcal{A}}}{\tilde{\mathcal{A}}}\kvec + \left(\jvec \times \tilde{b}_\alpha - \ivec \times \tilde{b}_\beta\right) \notag
\end{align}
from which it is clear by \eqref{TildeBFormula} and \eqref{TildeAFormula} that $U_{\tilde{\kappa}}^{-1}\left(\tilde{\mathfrak{a}}_t/\tilde{\mathfrak{a}}\right)$ consists of terms of size at most $O(\epsilon^3)$.

	\subsection{Analysis of the Approximate Solution and Hilbert Transform}

Most of the terms in the higher order corrections in \eqref{TildeLambdaFormula} can be estimated trivially in Sobolev space, given that $A$ and $B$ are in a suitably regular Sobolev space.  However, there are terms in the correctors of the form $$\epsilon^3|\mathcal{D}|^{-1}S$$ where $S$ depends on $A$, $B$, and their derivatives.  Such terms are not in general in $L^2$ unless extra conditions are put on $A$ and $B$.  The most obvious condition is to restrict $A \in \dot{H}^{-1} \cap H^s$, but this implies an undesirable mean-zero condition on $A$.  Instead, since Lemma \ref{AbstractEpsilon4} shows that $S$ is either a pure derivative or a product of two or more instances of $A$, $B$ and their derivatives, we opt to insist that some of the lower-order derivatives of $A$ and $B$ decay at infinity at a mild algebraic rate.  This, along with Proposition \ref{SobolevEmbeddings} and the following well-posedness result, gives us the required control in $L^2$.

\begin{proposition}\label{HNLSLocallyWellPosed}
Let $s \geq 6$ and $0 \leq \delta \leq 1$ be given.  Suppose that $A_0 \in H^s \cap H^3(\delta)$ and $B_0 \in H^{s - 3} \cap L^2(\delta)$.  Then there exists a $\mathscr{T} > 0$ depending on $\|A_0\|_{H^s \cap H^3(\delta)}$ and $\|B_0\|_{H^{s - 3} \cap L^2(\delta)}$ so that both the initial value problem consisting of the HNLS equation \eqref{HNLSVector} with initial data $A(0) = A_0$ and the initial value problem consisting of \eqref{AuxiliaryEquationVector} with $B(0) = B_0$ have unique solutions $A \in C([0, \mathscr{T}], H^s \cap H^3(\delta))$ and $B \in H^{s - 3} \cap L^2(\delta)$ respectively.  In particular, if the solution $A$ cannot be continued past a time $\mathscr{T}$, then $\|A(T)\|_{W^{\lceil s/2 \rceil, \infty}} \not\in L^\infty([0, \mathscr{T}))$.
\end{proposition}

\begin{proof}
To avoid cluttering the proof, we observe that by rescaling we may without loss of generality ignore positive constants depending on $k$.  Denote the linear propagator of HNLS by $e^{LT}$.  Then by a routine energy estimate $e^{LT}$ is unitary in every $H^s$ for any real $s$.  Using Duhamel's formula on HNLS with $s \geq 2$ gives
\begin{align*}
A(T) & = e^{LT} A(0) + \int_0^T e^{L(T - t)} A(t)|A(t)|^2 \, dt
\end{align*}
from which by a routine contraction mapping argument in $H^s$ we have local well posedness of $A$ in $H^s$.  We gain decay by rewriting HNLS as
$$(\langle (X, Y) \rangle^\delta A)_T - \jvec (\langle (X, Y) \rangle^\delta A)_{XX} + \jvec (\langle (X, Y) \rangle^\delta A)_{YY} = \jvec [\partial_X^2, \langle (X, Y) \rangle^\delta]A - \jvec [\partial_Y^2, \langle (X, Y) \rangle^\delta]A$$
Since $0 \leq \delta \leq 1$, the right hand side of this equation is a linear combination of $A$ and its first derivatives with bounded coefficients.  Therefore the usual energy estimate applied to the above equation along with its derivatives up to the third order yields
$$\frac{d}{dt} \|\langle (X, Y) \rangle^\delta A(T)\|_{H^3}^2 \leq C \|\langle (X, Y) \rangle^\delta A(T)\|_{H^3}^2 \|A\|_{C([0, \mathscr{T}], W^{4, \infty})}$$ and so Gr\"onwall's inequality implies $$\|\langle (X, Y) \rangle^\delta A(T)\|_{H^3} \leq C\|\langle (X, Y) \rangle^\delta A(0)\|_{H^3} e^{\|A\|_{C([0, \mathscr{T}], W^{4, \infty})}T}$$

Recall that by Lemma \ref{Epsilon4TermsSummary}, $B$ satisfies an equation of the form (again dropping positive constants) 
\begin{align}\label{BEquation}
B_T = \jvec B_{XX} - \jvec B_{YY} + F_1(A)B + F_2(A) \overline{B} + F_3(A)
\end{align} where $F_1$ and $F_2$ are polynomial functions of $A$ and its first and second derivatives, and $F_3$ is a polynomial function of $A$ and its first through third derivatives.  Because no derivatives of $B$ appear in the nonlinearity, we may use the same Duhamel argument to show local well posedness of $B$ in $H^{s - 3}$ given $A_0$ is as in (a).  Finally, the decay in $B$ follows in the same way as the decay in $A$; three derivatives of $A$ are needed to decay in order to control $\|B\|_{L^2(\delta)}$ since the coefficients and source terms in \eqref{BEquation} contain up to three derivatives of $A$.
\end{proof}

\begin{remark}\label{HowManyDerivatives}
Since the largest number of derivatives needed in $A$ and $B$ in order to guarantee that the residual of \eqref{WaterWaveEvolutionApprox} be in $H^s$ is $s + 11$ and $s + 10$, respectively, we must require that $A \in C([0, \mathscr{T}], H^{s + 13} \cap H^3(\delta))$ and $B \in C([0, \mathscr{T}], H^{s + 10} \cap L^2(\delta))$.  We will always take $B_0 = 0$ in the sequel.
\end{remark}

\begin{lemma}\label{TildeHDifferenceEstimates}
Suppose $A_0 \in H^{s + 13} \cap H^3(0^+)$.  Then for $\epsilon_0 > 0$ chosen sufficiently small,
$$\|(\nht_{\tilde{\zeta}} - \tilde{\nht})f\|_{H^s} \leq C\epsilon^4\|f\|_{H^s} \qquad \qquad \|(\nht_{\tilde{\zeta}} - \tilde{\nht})f\|_{H^s} \leq C\epsilon^3\|f\|_{W^{s, \infty}}$$
\end{lemma}

\begin{proof}
This amounts to analyzing in more detail the expansion of the factors \eqref{CauchyKernelExpansion}, \eqref{NormalVectorExpansion} of the kernel of $\nht_{\tilde{\zeta}}$.  Abbreviate the expansion \eqref{NormalVectorExpansion} by $\tilde{\mathcal{N}} = \tilde{\zeta}_\alpha \times \tilde{\zeta}_\beta = \mathcal{N}_0 + \mathcal{N}_1 + \mathcal{N}_2$, where $\mathcal{N}_j$ is homogeneous of degree $j$ in $\lambda - \lambda^\prime$.  Similarly, abbreviate the expansion \eqref{CauchyKernelExpansion} by
\begin{equation*}
\frac{\tilde{\zeta} - \tilde{\zeta}^\prime}{|\tilde{\zeta} - \tilde{\zeta}^\prime|^3} = K_0 + K_1 + K_2 + K_3 + R_4
\end{equation*} where $K_j$ is homogeneous of degree $j$ in $\lambda - \lambda^\prime$, and $R_4$ is the remainder term arising from the power series expansion in \eqref{CauchyKernelExpansion}.  By construction, we have that the kernel $\tilde{\nht}$ consists of those terms of the operator
\begin{equation}\label{TildeHFirstTry}
\iint \sum_{i + j \leq 3} K_i \mathcal{N}_j^\prime f^\prime \, d\alpha^\prime\, d\beta^\prime,
\end{equation} with formal powers of $\epsilon$ of size $O(\epsilon^3)$ and lower orders when the substitution $\tilde{\lambda} = \sum_{j = 1}^4 \epsilon^j\lambda^{(j)}$ is made and expanded.  Therefore our task is to estimate the following singular integrals:
\begin{equation*}
\mathcal{I}_1 = \iint \sum_{\stackrel{i \leq 3, \, j \leq 2}{i + j \geq 4}} K_i \mathcal{N}_j^\prime f^\prime \, d\alpha^\prime\, d\beta^\prime \qquad \text{and} \qquad \mathcal{I}_2 = \iint R_4 \mathcal{N}^\prime f^\prime \, d\alpha^\prime\, d\beta^\prime
\end{equation*} along with $\mathcal{I}_3$, which are the contributions of formal size $O(\epsilon^4)$ and higher order terms of \eqref{TildeHFirstTry}.

First we estimate $\mathcal{I}_1$.  If we write $\lambda - \lambda^\prime = -(P - P^\prime)\frac{(P - P^\prime)(\lambda - \lambda^\prime)}{|P - P^\prime|^2}$, we can express the general term $K_i \mathcal{N}_j^\prime$ as a product of $\frac{P - P^\prime}{|P - P^\prime|^3}$ along with the factors
\begin{equation}\label{DifferenceQuotientFactors}
-\frac{(P - P^\prime)(\tilde{\lambda} - \tilde{\lambda}^\prime)}{|P - P^\prime|^2} \qquad\qquad \frac{(P - P^\prime) \cdot (\tilde{\lambda} - \tilde{\lambda}^\prime)}{|P - P^\prime|^2} \qquad\qquad \frac{|\tilde{\lambda} - \tilde{\lambda}^\prime|^2}{|P - P^\prime|^2}
\end{equation}
\begin{equation*}
\kvec \qquad\qquad (\tilde{\lambda}_\alpha \times \jvec) + (\ivec \times \tilde{\lambda}_\beta) \qquad\qquad (\tilde{\lambda}_\alpha \times \tilde{\lambda}_\beta)
\end{equation*} in appropriate combinations so that the degree of homogeneity in $\tilde{\lambda}$ is $i + j$.  Each of the difference quotient expressions is of degree zero, and so the resulting kernels are of the type in Theorem \ref{SingularIntegralL2Estimate}(a). We can therefore estimate
\begin{align*}
\|\mathcal{I}_1\|_{H^s} & \leq C\sum_{i + j \geq 4}\|\nabla\tilde{\lambda}\|_{W^{s, \infty}}^{i + j}\|f\|_{H^s} \quad \text{or} \quad C\sum_{i + j \geq 4}\|\nabla\tilde{\lambda}\|_{W^{s, \infty}}^{i + j - 1}\|\nabla\tilde{\lambda}\|_{H^s}\|f\|_{W^{s, \infty}} \\
& \leq C\epsilon^4\|f\|_{H^s} \quad \text{or} \quad C\epsilon^3\|f\|_{W^{s, \infty}}
\end{align*} where the constant $C$ depends on $\|A_0\|_{H^{s + 13} \cap H^3(0^+)}$.

To estimate $\mathcal{I}_3$, make the substitution $\tilde{\lambda} = \sum_{j = 1}^4 \epsilon^j \lambda^{(j)}$ in each of the expressions in \eqref{DifferenceQuotientFactors} and expand.  This yields a set of factors of the same form as in \eqref{DifferenceQuotientFactors}, but with $\epsilon^j \lambda^{(j)}$ replacing $\tilde{\lambda}$.  Estimating the operators that result from the finite number of terms that contribute $O(\epsilon^4)$ or higher as we did when estimating $\mathcal{I}_1$ yields the same bounds.  We omit the details.

To estimate $\mathcal{I}_2$, denote the difference quotients $$2\frac{(P - P^\prime) \cdot (\tilde{\lambda} - \tilde{\lambda}^\prime)}{|P - P^\prime|^2} + \frac{|\tilde{\lambda} - \tilde{\lambda}^\prime|^2}{|P - P^\prime|^2} =: Q$$  If we choose $\epsilon_0 > 0$ so small so that $\|\nabla\tilde{\lambda}\|_{L^\infty} < \frac14$, we have by the Lipschitz bound $|\tilde{\lambda} - \tilde{\lambda}^\prime| \leq \|\nabla\tilde{\lambda}\|_{L^\infty}|P - P^\prime|$ that $|Q| < \frac23$.  In this case the Taylor expansion of $(1 + Q)^{-\frac32}$ is valid.  Observe that, for some universal constant $C_3$, we have the following integral expression for the remainder $R_4$: $$R_4 = \frac{\tilde{\zeta} - \tilde{\zeta}^\prime}{|\tilde{\zeta} - \tilde{\zeta}^\prime|^3}\int_0^Q C_3(1 + \tau)^{-9/2}(Q - t)^3 \, d\tau$$  Following the proof of Proposition \ref{SingularIntegralHsEstimate}, we take up to $j$ derivatives of the integral factor with respect to $(\partial_\alpha + \partial_{\alpha^\prime})$ or $(\partial_\beta + \partial_{\beta^\prime})$ for $0 \leq j \leq s$.  When we do so, we get by the Chain Rule and by Differentiation under the Integral Sign a sum of terms of the form
\begin{equation}\label{RemainderKernel}
\left(C_j^\prime\prod_{i = 1}^n (\partial_\alpha + \partial_{\alpha^\prime})^{j_i}(\partial_\beta + \partial_{\beta^\prime})^{l_i}Q\right)\int_0^Q C_m(1 + \tau)^{-\frac92}(Q - \tau)^m \, d\tau
\end{equation}
where $m = 0, 1, 2, 3$, $C_0 = 0$, $m + n = 3$, and $(3 - m) + \sum_{i = 1}^n (|j_i| + |l_i|) \leq s$.  Terms of this form with $m = 0$ can be estimated directly using Theorem \ref{SingularIntegralL2Estimate}.  Note that we can always estimate these integral terms in \eqref{RemainderKernel} coarsely by 
\begin{equation}\label{CoarseRemainderKernelBound}
\left|\int_0^Q C_m(1 + \tau)^{-\frac92}(Q - \tau)^m \, d\tau\right| \leq C|Q|^{m + 1}
\end{equation}
To write the terms with $m \geq 1$ in the correct form to apply Theorem \ref{SingularIntegralL2Estimate}, we repeatedly integrate by parts to arrive at the following formula for $N \geq 1$:\footnote{Here we use the generalized binomial coefficient $\binom{r}{n} = \frac{1}{n!}r(r - 1)(r - 2)\cdots(r - n + 1)$.  By convention we stipulate that $\binom{r}{0} = 1$.}
\begin{align*}
\int_0^Q C_m(1 + \tau)^{-\frac92}(Q - \tau)^m \, d\tau & = C_m \sum_{n = 1}^N (-1)^{n - 1} \binom{-\frac92}{n - 1}\binom{m + n - 1}{n - 1}^{-1} \frac{Q^{m + n}}{m + n} \\
& \quad + (-1)^{N} C_m \int_0^Q \binom{-\frac92}{N}\binom{m + N}{N}^{-1} (1 + \tau)^{-\frac92 - N}(Q - \tau)^{m + N} \, d\tau
\end{align*}
Using the bound $|\binom{r}{n}| \leq C_r n^{-1}$ along with \eqref{CoarseRemainderKernelBound}, we see that upon choosing $\epsilon_0 > 0$ sufficiently small, the above series converges pointwise in $Q$, giving the exact series representation
\begin{align*}
\int_0^Q C_m(1 + \tau)^{-\frac92}(Q - \tau)^m \, d\tau & = C_m \sum_{n = 0}^N (-1)^n\binom{-\frac92}{n}\binom{m + n}{n}^{-1} \frac{Q^{m + n + 1}}{m + n + 1}
\end{align*}
This yields an infinite series of singular integrals.  Using Theorem \ref{SingularIntegralL2Estimate}, we have upon choosing $\epsilon_0 > 0$ to be sufficiently small that the remainder is bounded by
\begin{align*}
\|\mathcal{I}_2\|_{H^s} & \leq C \sum_{n = 1}^\infty 2^{n + m}(1 + n^8)\|\nabla\tilde{\lambda}\|_{W^{s, \infty}}^{4 + n} \|f\|_{H^s} \\
& \leq C\epsilon^4\|f\|_{H^s}
\end{align*}
Similarly, one has $\|\mathcal{I}_2\|_{H^s} \leq C\|\nabla\tilde{\lambda}\|_{W{s, \infty}}^3 \|\nabla\tilde{\lambda}\|_{H^s}\|f\|_{W^{s, \infty}} \leq C\epsilon^3\|f\|_{W^{s, \infty}}$.
\end{proof}

\begin{lemma}\label{HMinusHTildeEstimates}
The following estimates hold:
\begin{itemize}
\item[(a)]{
$
\|(\nht_{\tilde{\zeta}} - \tilde{\nht})f\|_{H^s} \leq \quad C\epsilon^4\|f\|_{H^s} \quad \text{and} \quad C\epsilon^3\|f\|_{W^{s, \infty}}
$
}
\item[(b)]{
$
\|[D_t, \nht_{\tilde{\zeta}} - \tilde{\nht}]f\|_{H^s} \leq \quad C\epsilon^4\|f\|_{H^s} \quad \text{and} \quad C\epsilon^3\|f\|_{W^{s, \infty}}
$
}
\item[(c)]{
$
\|[D_t, [D_t, \nht_{\tilde{\zeta}} - \tilde{\nht}]]f\|_{H^s} \leq \quad C\epsilon^4\|f\|_{H^s} \quad \text{and} \quad C\epsilon^3 \|f\|_{W^{s, \infty}}
$
}
\end{itemize}
\end{lemma}

\begin{proof}
We successively distribute the operator $D_t$ through the kernel of $\nht_{\tilde{\zeta}} - \tilde{\nht}$ as derived in Lemma \ref{TildeHDifferenceEstimates}.  In the resulting singular integrals we will have to estimate quantities of the form $D_t \tilde{f}$ and $D_t^2 \tilde{f}$, where $\tilde{f} = \partial \znew^{(j)}, \partial \lambda^{(j)}$ for $j = 1, 2, 3, 4$ and $\partial = \partial_\alpha, \partial_\beta$.  By rewriting $D_t\tilde{f} = \tilde{f}_t + (b \cdot \mathcal{D})\tilde{f}$ and 
$$D_t^2 \tilde{f} = \partial_t^2 \tilde{f} + 2(b \cdot \mathcal{D})\partial_t\tilde{f} + ((D_tb) \cdot \mathcal{D})\tilde{f} + (b \cdot (b \cdot \mathcal{D})\mathcal{D})\tilde{f}$$ we can estimate these terms using the coarse bound $b$ and $D_t b$ through Proposition \ref{BADtAFormulas} by $\|b\|_{H^s}, \|D_t b\|_{H^s} \leq C\epsilon^2$.  As a consequence, the constant $C$ depends on $\|A_0\|_{H^{s + 13} \cap H^3(0^+)}$.  The details are left to the reader.
\end{proof}

\section{The Remainder Quantities and Their Estimates}

In this section we construct equations for quantities related to the remainder
\begin{equation}\label{rDefinition}
r = \lambda - \tilde{\lambda}
\end{equation}
and its derivative from the equation \eqref{NewEvolutionEquations} and \eqref{WaterWaveEvolutionApprox}-\eqref{DtWaterWaveEvolutionApprox}.  In order to derive equations with nonlinearities of cubic and higher orders, we first work with the related quantities
\begin{equation}\label{RhoDefinition}
\rho = \frac{1}{2}(I - \nht)\left((I - \nht)\znew - (I - \nht_{\tilde{\zeta}})\tilde{\znew}\right)\kvec
\end{equation}
as a proxy for the quantity $r$ and
\begin{equation}\label{SigmaDefinition}
\sigma = \frac{1}{2}(I - \nht)\left(D_t(I - \nht)\znew - \tilde{D}_t(I - \nht_{\tilde{\zeta}})\tilde{\znew}\right)\kvec
\end{equation}
as a proxy for the quantity $D_t r$.  Indeed, in parallel with the derivation of \eqref{NewEvolutionEquations} we can write
\begin{align}\label{RhoEquation}
\mathcal{P}\rho & = -\frac{1}{2}[\mathcal{P}, \nht]\left((I - \nht)\znew\kvec - (I - \nht_{\tilde{\zeta}})\tilde{\znew}\kvec\right) + \frac{1}{2}(I - \nht)\mathcal{P}\left((I - \nht)\znew\kvec - (I - \nht_{\tilde{\zeta}})\tilde{\znew}\kvec\right) \notag \\
& = -\frac{1}{2}[\mathcal{P}, \nht]\left((I - \nht)\znew\kvec - (I - \nht_{\tilde{\zeta}})\tilde{\znew}\kvec\right) \notag \\
& \quad - \frac{1}{2}(I - \nht)(\mathcal{P} - \tilde{\mathcal{P}})(I - \nht_{\tilde{\zeta}})\tilde{\znew}\kvec \notag \\
& \quad + \frac12(I - \nht)\left(G - \tilde{\mathcal{P}}(I - \nht_{\tilde{\zeta}})\tilde{\znew}\kvec\right) \notag \\
& := G_\rho
\end{align}
and
\begin{align}\label{SigmaEquation}
\mathcal{P}\sigma & = -\frac12[\mathcal{P}, \nht]\left(D_t(I - \nht)\znew\kvec - \tilde{D}_t(I - \nht_{\tilde{\zeta}})\tilde{\znew}\kvec\right) \notag \\
& \quad - \frac12(I - \nht)\left((\mathcal{P} - \tilde{\mathcal{P}})\tilde{D}_t(I - \nht_{\tilde{\zeta}})\tilde{\znew}\kvec\right) \notag \\
& \quad + \frac12(I - \nht)\left([\mathcal{P}, D_t](I - \nht)\znew\kvec - [\tilde{\mathcal{P}}, \tilde{D}_t](I - \nht_{\tilde{\zeta}})\tilde{\znew}\kvec\right) \notag \\
& \quad + \frac12(I - \nht)\left(D_t\mathcal{P}(I - \nht)\znew\kvec - \tilde{D}_t\tilde{\mathcal{P}}(I - \nht_{\tilde{\zeta}})\tilde{\znew}\kvec\right) \notag \\
& \quad := G_\sigma
\end{align}
We spend the first part of this section controlling the above nonlinearities in terms of the quantity
\begin{equation}\label{EDefinition}
E(t) := \|\,|\mathcal{D}|^\frac12r\|_{H^{s + \frac12}}^2 + \|D_t r\|_{H^{s + \frac12}} + \|D_t^2 r\|_{H^s}^2 
\end{equation}

While the energy constructed directly from \eqref{RhoEquation}, \eqref{SigmaEquation} controls $E$, it does not provide a priori bounds for long enough times to provide a justification for HNLS.  This is due to third-order terms appearing in the energy inequality arising from the quadratic null-form nonlinearities.  Rather than state the full energy that we use immediately, we will derive it as a small perturbation of the energy derived directly from \eqref{RhoEquation}, \eqref{SigmaEquation} in the course of eliminating the null-form nonlinearities using a combination of the method of normal forms and third-order corrections to the energy itself.

So that we are certain of the conditions under which the above quantities and operators are well-defined, we explicitly state and, in this section explicitly assume the following

\begin{assumption}\label{APrioriBound}
Let $s \geq 6$.  For some $A_0 \in H^{s + 13} \cap H^3(0^+)$, let $A, B, \tilde{\zeta}$ be as in Proposition \ref{HNLSLocallyWellPosed}.  Then we suppose that there is an interval $[0, T_0]$ on which there exists a solution $\zeta$ to the equations \eqref{NewEvolutionEquations} satisfying $$\sup_{0 \leq t \leq T_0} E(t)^\frac12 \leq C\epsilon^2$$  As a consequence this implies that $$\sup_{0 \leq t \leq T_0} \|\,|\mathcal{D}|^\frac12(\zeta - P)\|_{W^{s - 1, \infty}} + \|D_t\zeta\|_{W^{s - 1, \infty}} + \|D_t^2 \zeta\|_{W^{s - 2, \infty}} \leq C\epsilon$$ and in particular that the chord-arc condition $$\frac{1}{C} \leq \frac{|(\alpha, \beta) - (\alpha^\prime, \beta^\prime)|}{|\zeta(\alpha, \beta) - \zeta(\alpha^\prime, \beta^\prime)|} \leq C \hspace{.75cm} \text{for all } (\alpha, \beta, t) \neq (\alpha^\prime, \beta^\prime, t)$$ holds for all $0 \leq t \leq T_0$.
\end{assumption}

	\subsection{Relations and Estimates between Remainders of Quantities}

This section is devoted to deriving relations between remainder quantities.  In particular we will show that auxiliary remainder quantities such as $b - \tilde{b}$, $\mathcal{A} - \tilde{\mathcal{A}}$, etc., are suitably bounded in terms of $E$.

\begin{lemma}\label{IsolateIJValued}
Let $f$ be an $\ivec, \jvec$-valued function, and let $g$ be a $\kvec$-valued function.  Then for sufficiently small $\epsilon_0 > 0$ we have $$\|f\|_{H^s} \leq C\|(I - \nht)f\|_{H^s}$$ and $$\|g\|_{H^s} \leq C\|(I - \nht)g\|_{H^s}$$
\end{lemma}

\begin{proof}
The proof of the estimates are similar, and so we will only show the first.  Since $f$ is $\ivec,\jvec$-valued, $f^\dagger = f$ and so we have
$$f = \frac12(\nht - \nht^\dagger)f + \frac12(I - \nht)f + \left(\frac12 (I - \nht)f\right)^\dagger$$ from which $\|f\|_{L^2} \leq \|(I - \nht)f\|_{L^2} + \|(\nht - \nht^\dagger)f\|_{L^2}$.
Now since $\nht^\dagger = \nht_{\zeta^\dagger}$, we can use Proposition \ref{CommutatorEstimates} to find that
\begin{align*}
\|f\|_{H^s} & \leq C(\epsilon + \||\mathcal{D}|(\znew - \tilde{\znew})\|_{H^s}\|f\|_{H^s} + \|(I - \nht)f\|_{H^s} \\
& \leq C \epsilon \|f\|_{H^s} + \|(I - \nht)f\|_{H^s}
\end{align*} and the bound now follows for $\epsilon_0 > 0$ chosen sufficiently small. 
\end{proof}

We record the operator differences:

\begin{equation}
D_t - \tilde{D}_t = (b - \tilde{b}) \cdot \mathcal{D}
\end{equation}

\begin{align}\label{DiffDt2Formula}
D_t^2 - \tilde{D}_t^2 & = D_t(D_t - \tilde{D}_t) + (D_t - \tilde{D}_t)\tilde{D}_t \notag \\
& = D_t(b - \tilde{b}) \cdot \mathcal{D} + (b - \tilde{b}) \cdot (D_t\mathcal{D} + \mathcal{D}\tilde{D}_t) 
\end{align}

\begin{align}\label{DiffANTimesNablaFormula}
\mathcal{A}(\mathcal{N} \times \nabla) - \tilde{\mathcal{A}}(\tilde{\mathcal{N}} \times \nabla) = (\mathcal{A} - \tilde{\mathcal{A}})(\zeta_\beta\partial_\alpha - \zeta_\alpha \partial_\beta) + \tilde{\mathcal{A}}(r_\beta\partial_\alpha - r_\alpha\partial_\beta)
\end{align}

\begin{align}\label{DiffPFormula}
\mathcal{P} - \tilde{\mathcal{P}} & = D_t(b - \tilde{b}) \cdot \mathcal{D} + (b - \tilde{b}) \cdot (D_t\mathcal{D} + \mathcal{D}\tilde{D}_t) \notag \\
& \quad - (\mathcal{A} - \tilde{\mathcal{A}})(\zeta_\beta\partial_\alpha - \zeta_\alpha\partial_\beta) - \tilde{\mathcal{A}}(r_\beta \partial_\alpha - r_\alpha \partial_\beta)
\end{align}

as well as the commutator

\begin{align}\label{DtCommuteP}
[D_t, \mathcal{P}] & = -(D_t\mathcal{A})(\mathcal{N} \times \nabla) \notag \\
& \quad - \mathcal{A}(\zeta_\beta(b_\alpha \cdot \mathcal{D}) - \zeta_\alpha(b_\beta \cdot \mathcal{D})) \notag \\
& \quad - \mathcal{A}(D_t\zeta_\beta \partial_\alpha - D_t\zeta_\alpha \partial_\beta) 
\end{align}

\begin{proposition}\label{DiffQuantitiesEstimates}
Let $s \geq 6$.  Then for sufficiently small $\epsilon_0 > 0$ we have
\begin{itemize}
\item[(a)]{$\|b - \tilde{b}\|_{H^{s + \frac12}} \leq C(E + \epsilon E^\frac12 + \epsilon^3)$}
\item[(b)]{$\|D_t(b - \tilde{b})\|_{H^s} \leq C(E + \epsilon E^\frac12 + \epsilon^3)$}
\item[(c)]{$\|\mathcal{A} - \tilde{\mathcal{A}}\|_{H^s} \leq C(E + \epsilon E^\frac12 + \epsilon^3)$}
\end{itemize}
\end{proposition}

\begin{proof}
In this proof, let $\mathcal{O}(\epsilon^n)$ represent a term with Sobolev norm of size $O(\epsilon^n)$.  To prove (a), recall the formula for $b$ given in Proposition \ref{BADtAFormulas} as well as the approximate equation \eqref{BFormulaApprox}: $$(I - \tilde{\nht})\tilde{b} = -[\tilde{D}_t, \tilde{\nht}](I + \tilde{\nht})\tilde{\znew}\kvec + (I - \tilde{\nht})\tilde{D}_t\tilde{\mathcal{K}}\tilde{\znew}\kvec + \mathcal{O}(\epsilon^3)$$  We subtract these and express the difference in such a way as to show the explicit dependence on approximate and remainder quantities and operators:
\begin{align}\label{IMinusHBMinusBTilde}
(I - \nht)(b - \tilde{b}) & = - [\tilde{D}_t, \tilde{\nht} - \nht](I + \tilde{\nht})\tilde{\znew}\kvec \notag \\
& \quad - [\tilde{D}_t, \nht](\tilde{\nht} - \nht)\tilde{\znew}\kvec \notag \\
& \quad - [\tilde{D}_t, \nht](I + \nht)(\tilde{\znew} - \znew)\kvec \notag \\
& \quad - [(\tilde{b} - b) \cdot \mathcal{D}, \nht](I + \nht)\znew\kvec\notag  \\
& \quad + (\tilde{\nht} - \nht)\tilde{D}_t\tilde{\mathcal{K}}\tilde{\znew}\kvec \notag \\
& \quad + (I - \nht)((\tilde{b} - b) \cdot \mathcal{D})\tilde{\mathcal{K}}\tilde{\znew}\kvec \\
& \quad + (I - \nht)D_t(\tilde{\mathcal{K}} - \mathcal{K})\tilde{\znew}\kvec \notag \\
& \quad + (I - \nht)D_t\mathcal{K}(\tilde{\znew} - \znew)\kvec \notag \\
& \quad + \mathcal{O}(\epsilon^3) \notag \\
& = I_1 + I_2 + \cdots + I_8 + \mathcal{O}(\epsilon^3)\notag 
\end{align}

First, using Proposition \ref{CommutatorEstimates} along with Proposition \ref{ComplexInterpolation} with the multiplication mapping $T(g) = gh$, we have that $$\|I_6\|_{H^{s + \frac12}} \leq \|b - \tilde{b}\|_{H^{s + \frac12}}\|\tilde{\mathcal{K}}\tilde{\znew}\|_{W^{s + 2, \infty}} \leq C\epsilon\|b - \tilde{b}\|_{H^{s + \frac12}}$$  Using Proposition \ref{CommutatorEstimates}, we see that $I_4$ is bounded by $C(E^{1/2} + \epsilon)\|b - \tilde{b}\|_{H^{s + \frac12}}$.  Next, by Lemmas \ref{CommutatorEstimates} and \ref{HMinusHTildeEstimates} we have
$$\|I_1\|_{H^{s + \frac12}}, \, \|I_2\|_{H^{s + \frac12}}, \, \|I_5\|_{H^{s + \frac12}} \leq C(E^\frac12 + \epsilon^3)\epsilon \leq C(\epsilon E^\frac12 + \epsilon^3)$$  Similarly, since $\mathcal{K} = \Re(\nht)$, the same estimates imply that $\|I_7\|_{H^{s + \frac12}} \leq C(\epsilon E^\frac12 + \epsilon^3)$.  The remaining terms satisfy $$\|I_3\|_{H^{s + \frac12}}, \, \|I_8\|_{H^{s + \frac12}} \leq C(E^\frac12 + \epsilon)E^\frac12$$  Applying Lemma \ref{IsolateIJValued} gives
$$\|b - \tilde{b}\|_{H^{s + \frac12}} \leq C(E^\frac12 + \epsilon)\|b - \tilde{b}\|_{H^{s + \frac12}} + C(E + \epsilon E^\frac12 + \epsilon^3),$$ from which (a) follows provided we choose $\epsilon_0 > 0$ to be sufficiently small by the a priori bound on $E$.

To show (b), we write
$$(I - \nht)D_t(b - \tilde{b}) = -[D_t, \nht](b - \tilde{b}) + D_t(I - \nht)(b - \tilde{b})$$
By Proposition \ref{CommutatorEstimates} and (a), the first of these terms is bounded by $C(E + \epsilon E^\frac12 + \epsilon^3)$ in $H^s$.  Applying $D_t$ to \eqref{IMinusHBMinusBTilde} and commuting $D_t$ past the various operators in \eqref{IMinusHBMinusBTilde} yields a sum of terms that can be estimated using (a), Proposition \ref{CommutatorEstimates} and Lemma \ref{HMinusHTildeEstimates}.  This yields bounds of the form
$$\|D_t(I - \nht)(b - \tilde{b})\|_{H^s} \leq C(E^\frac12 + \epsilon)\|D_t(b - \tilde{b})\|_{H^s} + C(E + \epsilon E^\frac12 + \epsilon^3) + \mathcal{O}(\epsilon^3)$$  Then by Lemma \ref{IsolateIJValued} we have
$\|D_t(b - \tilde{b})\|_{H^s} \leq C(E^\frac12 + \epsilon)\|D_t(b - \tilde{b})\|_{H^s} + C(E + \epsilon E^\frac12 + \epsilon^3)$, from which (b) follows after possibly choosing smaller $\epsilon_0 > 0$.

For the proof of (c), we begin with the formula \eqref{AFormulaMultiscale}:
\begin{align*}
(I - \mathcal{K})\mathcal{A} & = \Bigl\{1 + [D_t, \nht]D_t\zeta + [\mathcal{A}(\mathcal{N} \times \nabla), \nht] (I + \nht)\znew\kvec  \\
& \qquad + (I - \nht)\left(-\mathcal{A}\zeta_\beta \times (\partial_\alpha\mathcal{K}\znew\kvec) + \mathcal{A}\zeta_\alpha \times (\partial_\beta\mathcal{K}\znew\kvec) + \mathcal{A}(\lambda_\alpha \times \lambda_\beta)\right)\Bigr\}_3 \notag
\end{align*} 
we subtract its approximate version \eqref{AFormulaApprox} and arrive at the following sum of terms, where we have isolated the occurence of $\mathcal{A} - \tilde{\mathcal{A}}$:
\begin{align}\label{DiffAFormula}
& \quad (I - \mathcal{K})(\mathcal{A} - \tilde{\mathcal{A}}) \notag \\
& = \Bigl((I - \mathcal{K})(\mathcal{A} - 1) - (I - \tilde{\mathcal{K}})(\tilde{\mathcal{A}} - 1)\Bigr) + (\mathcal{K} - \tilde{\mathcal{K}})(\tilde{\mathcal{A}} - 1) \notag \\
& = (\mathcal{K} - \tilde{\mathcal{K}})(\tilde{\mathcal{A}} - 1) \notag \\
& \quad + \Biggl\{[D_t, \nht]D_t\zeta - [\tilde{D}_t, \tilde{\nht}]\tilde{D}_t\tilde{\zeta} \\
& \qquad + [(\mathcal{A} - \tilde{\mathcal{A}})(\mathcal{N} \times \nabla), \nht] (I + \nht)\znew\kvec \notag \\
& \qquad + [\tilde{\mathcal{A}}(\mathcal{N} \times \nabla), \nht] (I + \nht)\znew\kvec - [\tilde{\mathcal{A}}(\tilde{\mathcal{N}} \times \tilde{\nabla}), \tilde{\nht}] (I + \tilde{\nht})\tilde{\znew}\kvec \notag \\
& \qquad + (I - \nht)\left(-(\mathcal{A} - \tilde{\mathcal{A}})\zeta_\beta \times (\partial_\alpha\mathcal{K}\znew\kvec) + (\mathcal{A} - \tilde{\mathcal{A}})\zeta_\alpha \times (\partial_\beta\mathcal{K}\znew\kvec) + (\mathcal{A} - \tilde{\mathcal{A}})(\lambda_\alpha \times \lambda_\beta)\right) \notag \\
& \qquad + \Biggl( (I - \nht)\left(-\tilde{\mathcal{A}}\zeta_\beta \times (\partial_\alpha\mathcal{K}\znew\kvec) + \tilde{\mathcal{A}}\zeta_\alpha \times (\partial_\beta\mathcal{K}\znew\kvec) + \tilde{\mathcal{A}}(\lambda_\alpha \times \lambda_\beta)\right)  \notag \\
& \qquad \qquad - (I - \tilde{\nht})\left(-\tilde{\mathcal{A}}\tilde{\zeta}_\beta \times (\partial_\alpha\tilde{\mathcal{K}}\tilde{\znew}\kvec) + \tilde{\mathcal{A}}\tilde{\zeta}_\alpha \times (\partial_\beta\tilde{\mathcal{K}}\tilde{\znew}\kvec) + \tilde{\mathcal{A}}(\tilde{\lambda}_\alpha \times \tilde{\lambda}_\beta)\right)\Biggr)\Biggr\}_3\notag + \mathcal{O}(\epsilon^3) \notag \\
& = J_1 + J_2 + \cdots + J_6 + \mathcal{O}(\epsilon^3) \notag
\end{align}
As in (a) we have $$\|J_3\|_{H^s}, \, \|J_5\|_{H^s} \leq C(E^\frac12 + \epsilon)\|\mathcal{A} - \tilde{\mathcal{A}}\|_{H^s}$$ and by expanding the other terms into a sum of terms involving only approximate and remainder quantities using the method of \eqref{IMinusHBMinusBTilde}, we estimate the other terms as $$\|J_1\|_{H^s}, \, \|J_2\|_{H^s}, \, \|J_4\|_{H^s}, \, \|J_6\|_{H^s} \leq C(E + \epsilon E^\frac12 + \epsilon^3)$$
Thus we have 
\begin{align*}
\|\mathcal{A} - \tilde{\mathcal{A}}\|_{H^s} & \leq \|\mathcal{K}(\mathcal{A} - \tilde{\mathcal{A}})\|_{H^s} + C\|\mathcal{A} - \tilde{\mathcal{A}}\|_{H^s}(E^\frac12 + \epsilon) + C(E + \epsilon E^\frac12 + \epsilon^3) \\
& \leq C\|\mathcal{A} - \tilde{\mathcal{A}}\|_{H^s}(E^\frac12 + \epsilon) + C(E + \epsilon E^\frac12 + \epsilon^3),
\end{align*}
from which (c) follows by choosing $\epsilon_0 > 0$ sufficiently small.
\end{proof}

The energy we construct in the next section is in terms of quantities such as $D_t\partial^j\rho$, $D_t\partial^j \sigma$ in $L^2$.  We must show that we can bound these quantities in $L^2$ by $E^\frac12$ plus an acceptable error depending on $\epsilon$.  We will be aided by the commutator identities

\begin{equation}\label{DtCommutePartials}
[\partial^j, D_t] = \sum_{j_1 + j_2 = j}^{j_1 \neq 0} (\partial^{j_1} b) \cdot \mathcal{D}\partial^{j_2}
\end{equation}

\begin{equation}\label{Dt2CommutePartials}
[\partial^j, D_t^2] = [\partial^j, D_t]D_t + D_t[\partial^j, D_t]
\end{equation}

\begin{proposition}\label{CarefulControlRhoSigma}
We have the following estimates:
\begin{itemize}
\item[(a)]{$\|\,|\mathcal{D}|^\frac12(\rho - r^\dagger)\|_{H^{s + \frac12}} \leq C\epsilon(E^\frac12 + \epsilon^2)$}
\item[(b)]{$\|\Re(\rho)\|_{H^{s + 1}} + \|\Re(\sigma)\|_{H^{s + \frac12}} \leq C\epsilon (E^\frac12 + \epsilon^2)$} 
\item[(c)]{$\|\sigma - D_t\rho\|_{H^{s + \frac12}} + \|\sigma + D_t r^\dagger\|_{H^{s + \frac12}} \leq C\epsilon(E^\frac12 + \epsilon^2)$}
\item[(d)]{$\|D_t\sigma - D_t^2\rho\|_{H^s} + \|D_t^2\rho - D_t^2r^\dagger\|_{H^s} \leq C\epsilon(E^\frac12 + \epsilon^2)$}
\item[(e)]{For $|j| \leq s$, $$\|D_t\partial^j\rho - \partial^jD_t\rho\|_{H^\frac12} + \|D_t^2 \partial^j \rho - \partial^j D_t^2 \rho\|_{L^2} + \|D_t\partial^j\sigma - \partial^j D_t\sigma\|_{L^2} \leq C\epsilon(E^\frac12 + \epsilon^2)$$}
\end{itemize}
\end{proposition}

\begin{proof}
We first establish some identities between $r$, $\rho$, $\sigma$, and their time derivatives.  From \eqref{RhoDefinition} we have
\begin{align}\label{ExpandRhoPt1}
\rho & = (I - \nht)\znew\kvec - \frac{1}{2}(I - \nht)(I - \nht_{\tilde{\zeta}})\tilde{\znew}\kvec \notag \\
& = (I - \nht)\znew\kvec - (I - \nht_{\tilde{\zeta}}) \tilde{\znew}\kvec + \frac{1}{2}(I + \nht)(I - \nht_{\tilde{\zeta}})\tilde{\znew}\kvec \notag \\
& = (I - \nht)\znew\kvec - (I - \nht_{\tilde{\zeta}}) \tilde{\znew}\kvec + \frac{1}{2}(\nht - \nht_{\tilde{\zeta}})(I - \nht_{\tilde{\zeta}})\tilde{\znew}\kvec 
\end{align}
Note first that by Proposition \ref{ChangeOfVariablesIJValued} this gives 
\begin{equation}\label{RhoScalarPart}
\Re(\rho) = \Re\left(\frac12(\nht - \nht_{\tilde{\zeta}})(I - \nht_{\tilde{\zeta}})\tilde{\znew}\right)
\end{equation}  
and so for $|j| \leq s + 1$, $\|\partial^j \Re(\rho)\|_{L^2} \leq C\epsilon (E^\frac12 + \epsilon^2)$.  Similarly, applying $D_t$ to \eqref{RhoScalarPart} and using Propositions \ref{DiffHilbertTransforms}, \ref{SingularIntegralHsEstimate}, and \ref{DiffQuantitiesEstimates} similarly along with Proposition \ref{ComplexInterpolation} gives $\|\Re(\sigma)\|_{H^{s + \frac12}} \leq C\epsilon E^\frac12$ for all $|j| \leq s$, which is (b).  We can also use \eqref{ExpandRhoPt1} along with \eqref{LambdaFromZMultiscale}-\eqref{LambdaFromZApprox} to make rigorous the heuristic that $\rho \sim -r^\dagger$:
\begin{align}\label{ExpandRhoPt2}
\rho + r^\dagger & = (I - \nht)\znew\kvec - (I - \nht_{\tilde{\zeta}}) \tilde{\znew}\kvec + \frac{1}{2}(\nht - \nht_{\tilde{\zeta}})(I - \nht_{\tilde{\zeta}})\tilde{\znew}\kvec \notag \\
& \quad + \left((I + \nht)\znew\kvec - (I + \tilde{\nht})\tilde{\znew}\kvec\right)^\dagger \notag \\
& \quad + \mathcal{K}\znew\kvec - \tilde{\mathcal{K}}\tilde{\znew}\kvec \notag \\
& \quad + \left((I + \tilde{\nht})\tilde{\znew}\kvec - \tilde{\mathcal{K}}\tilde{\znew}\kvec - \tilde{\lambda}\right)^\dagger \notag \\
& = (\nht^\dagger - \nht)\znew\kvec - (\nht_{\tilde{\zeta}}^\dagger - \nht_{\tilde{\zeta}}) \tilde{\znew}\kvec + \frac{1}{2}(\nht - \nht_{\tilde{\zeta}})(I - \nht_{\tilde{\zeta}})\tilde{\znew}\kvec \\
& \quad - (\tilde{\nht} - \nht_{\tilde{\zeta}})\tilde{\znew}\kvec + \mathcal{K}\znew\kvec - \tilde{\mathcal{K}}\tilde{\znew}\kvec \notag \\
& \quad + \left((I + \tilde{\nht})\tilde{\znew}\kvec - \tilde{\mathcal{K}}\tilde{\znew}\kvec - \tilde{\lambda}\right)^\dagger \notag \\
& = \left((\nht - \nht_{\tilde{\zeta}})^\dagger - (\nht - \nht_{\tilde{\zeta}})\right)\tilde{\znew}\kvec \notag \\
& \quad + \frac{1}{2}(\nht - \nht_{\tilde{\zeta}})(I - \nht_{\tilde{\zeta}})\tilde{\znew}\kvec - (\tilde{\nht} - \nht_{\tilde{\zeta}})\tilde{\znew}\kvec \notag \\
& \quad + \mathcal{K}(\znew - \tilde{\znew})\kvec + (\mathcal{K} - \mathcal{K}_{\tilde{\zeta}})\tilde{\znew}\kvec + (\mathcal{K}_{\tilde{\zeta}} - \tilde{\mathcal{K}})\tilde{\znew}\kvec \notag \\
& \quad + \left((I + \tilde{\nht})\tilde{\znew}\kvec - \tilde{\mathcal{K}}\tilde{\znew}\kvec - \tilde{\lambda}\right)^\dagger \notag
\end{align}
We first use \eqref{ExpandRhoPt2} to show that, for $1 \leq |j| \leq s + 1$
\begin{equation}\label{RhoPlusRDaggerBound}
\|\partial^j(\rho + r^\dagger)\|_{L^2} \leq C(\epsilon E^\frac12 + \epsilon^3)
\end{equation}
To do so, we estimate $\partial^j(\nht - \nht_{\tilde{\zeta}})\tilde{\znew}\kvec$ by writing it via Proposition \ref{DiffHilbertTransforms} and distributing derivatives in this integral expression as in Proposition \ref{SingularIntegralHsEstimate} to get $$\|\partial^j(\nht - \nht_{\tilde{\zeta}})\tilde{\znew}\kvec\|_{L^2} \leq C\|\nabla r\|_{H^s}\|\tilde{\znew}\|_{W^{s, \infty}_\epsilon} \leq C\epsilon E^\frac12$$
By Proposition \ref{TildeHDifferenceEstimates} we have $\|\partial^j(\nht_{\tilde{\zeta}} - \tilde{\nht})\tilde{\znew}\kvec\|_{H^1} \leq C\epsilon^4$.  Since $\mathcal{K} - \mathcal{K}_{\tilde{\zeta}} = \Re(\nht - \nht_{\tilde{\zeta}})$ and $\tilde{\mathcal{K}} - \mathcal{K}_{\tilde{\zeta}} = \Re(\tilde{\nht} - \nht_{\tilde{\zeta}})$, the estimates of terms involving $\mathcal{K} - \mathcal{K}_{\tilde{\zeta}}$ and $\tilde{\mathcal{K}} - \mathcal{K}_{\tilde{\zeta}}$ reduce to those already given.  Notice that by Lemma \ref{SingularIntegralHsEstimate} and the a priori bound of $\zeta$ we have $\|\partial^j\mathcal{K}(\znew - \tilde{\znew})\kvec\|_{L^2} \leq C(E^\frac12 + \epsilon)E^\frac12 \leq C\epsilon E^\frac12$.  Finally, the residual $\left((I + \tilde{\nht})\tilde{\znew}\kvec - \tilde{\mathcal{K}}\tilde{\znew}\kvec - \tilde{\lambda}\right)^\dagger$ is by \eqref{TildeLambdaFormula} bounded in $L^2$ by $C\epsilon^3$.  Estimate (a) follows in the same way by estimating in $\dot{H}^\frac12$ using Proposition \ref{CommutatorEstimates}.  We omit the details.

Next, we make rigorous the heuristic $\sigma \sim D_t\rho$ by writing $\sigma$ as
\begin{align}\label{ExpandSigma}
\sigma & = \frac12[D_t, \nht](I - \nht)\znew\kvec + D_t\frac12(I - \nht)^2\znew\kvec \notag \\
& \quad - \frac12[\tilde{D}_t, \nht](I - \nht_{\tilde{\zeta}})\tilde{\znew}\kvec + \tilde{D}_t\frac12(I - \nht)(I - \nht_{\tilde{\zeta}})\tilde{\znew}\kvec \notag \\
& = \frac12[D_t, \nht](I - \nht)\znew\kvec - \frac12[\tilde{D}_t, \nht](I - \nht_{\tilde{\zeta}})\tilde{\znew}\kvec \notag\\
& \quad - (D_t - \tilde{D}_t)\frac12(I - \nht)(I - \nht_{\tilde{\zeta}})\tilde{\znew}\kvec \notag \\
& \quad + D_t\rho
\end{align}
Again by using Propositions \ref{SingularIntegralHsEstimate}, \ref{CommutatorEstimates}, and \ref{DiffQuantitiesEstimates}, we have the estimate
\begin{align*}
\|\sigma - D_t\rho\|_{H^{s + \frac12}} & \leq C(E + \epsilon E^\frac12 + \epsilon^3)(E^\frac12 + \epsilon) + C\epsilon^2(E^\frac12 + \epsilon) + C(E + \epsilon E^\frac12 + \epsilon^3) \\
& \leq C\epsilon(E^\frac12 + \epsilon^2)
\end{align*}
which is (b).  Similarly, by applying $D_t$ to the right hand side of \eqref{ExpandRhoPt2} and using Proposition \ref{HMinusHTildeEstimates} and (a) on the result implies the bound
\begin{align*}
\|\sigma + D_t r^\dagger\|_{H^{s + \frac12}} & \leq \|\sigma - D_t \rho\|_{H^{s + \frac12}} + \|D_t(\rho + r^\dagger)\|_{H^{s + \frac12}} \\
& \leq C\epsilon(E^\frac12 + \epsilon^2)
\end{align*}
which gives (c).  In the same way, to show (d) we apply another $D_t$ to the right hand side of \eqref{ExpandRhoPt2} as well as a $D_t$ to \eqref{ExpandSigma} and use the same estimates to show that
$$\|D_t\sigma - D_t^2r^\dagger\|_{H^s}, \|D_t^2\rho - D_t^2r^\dagger\|_{H^s} \leq C\epsilon(E^\frac12 + \epsilon^2)$$

Finally we prove (e).  We first show the bounds on $D_t\partial^j \rho$, $D_t^2\partial^j \rho$, and $D_t\partial^j \sigma$ in $L^2$ for $0 \leq |j| \leq s$.  By \eqref{DtCommutePartials} and Proposition \ref{DiffQuantitiesEstimates} we have for $f = \rho, \sigma$ that
\begin{align*}
\|\partial^jD_t f - D_t \partial^j f\|_{L^2} & \leq \sum_{|i| \leq j} C_i \|(\partial^i b) \cdot \mathcal{D}\partial^{j - i}f\|_{L^2} \\
& \leq C\epsilon\|\nabla f\|_{H^{s - 1}}
\end{align*}
Then by part (a) this implies that $\|\partial^jD_t f - D_t \partial^j f\|_{L^2} \leq C\epsilon(E^\frac12 + \epsilon^2)$ when $f = \rho, \sigma$.  In the same way, if we consider only $f = r, \rho$ and use Proposition \ref{ComplexInterpolation} on the product mapping $T(f, g) = fg$, we have in the same way the estimate 
\begin{align*}
\|\partial^jD_t f - D_t \partial^j f\|_{H^\frac12} & \leq \sum_{j_1 + j_2 = j}^{j_1 \neq 0} \|(\partial^{j_1} b) \cdot \mathcal{D}\partial^{j_2}f\|_{H^\frac12} \\
& \leq C\epsilon\|\nabla f\|_{H^{s - \frac12}}
\end{align*}
which follows from Proposition \ref{CommutatorEstimates} and Proposition \ref{DiffQuantitiesEstimates}.  This in turn gives the bound $\|\partial^jD_t f - D_t \partial^j f\|_{H^\frac12} \leq C\epsilon(E^\frac12 + \epsilon^2)$.  Similarly, if we also use \eqref{Dt2CommutePartials}, we can write
\begin{align*}
\partial^j D_t^2 f - D_t^2 \partial^j f & = [\partial^j, D_t]D_t f + D_t \sum_{j_1 + j_2 = j}^{j_1 \neq 0} (\partial^{j_1} b) \cdot \mathcal{D}\partial^{j_2} f \\
& = 2 [\partial^j, D_t]D_t f + \sum_{j_1 + j_2 = j}^{j_1 \neq 0} ([D_t, \partial^{j_1}] b) \cdot \mathcal{D}\partial^{j_2} f \\
& \qquad\qquad + \sum_{j_1 + j_2 = j}^{j_1 \neq 0}  (\partial^{j_1} D_t b) \cdot \mathcal{D}\partial^{j_2} f + \sum_{j_1 + j_2 = j}^{j_1 \neq 0}  (\partial^{j_1} D_t b) \cdot [D_t, \mathcal{D}\partial^{j_2}] f
\end{align*}
From this expression and part (a) we obtain the following estimate for $f = \rho, \sigma$ $$\|\partial^jD_t^2f - D_t^2\partial^jf\|_{L^2} \leq C\epsilon(\|\nabla f\|_{H^{s - 1}} + \|D_t f\|_{H^s}) \leq C\epsilon(E^\frac12 + \epsilon^2)$$
Summing these estimates gives (e).
\end{proof}

Next, we show that we are able to control the change of variables $\kappa$ in terms of the quantity $\zeta$ and its derivatives alone.

\begin{proposition}\label{KappaControlledByZeta}
Let $\kappa = \kappa_1\ivec + \kappa_2\jvec$ be defined as in \eqref{KappaDefinition}, and let $T_0$ be the a priori existence time of $\zeta$.  Then
$$\sup_{0 \leq t \leq \min(T_0, \mathscr{T}\epsilon^{-2})} \|\nabla \kappa - I\|_{W^{s - 2, \infty}} \leq C\epsilon,$$ In particular, $\kappa : \mathbb{R}^2 \to \mathbb{R}^2$ is a global diffeomorphism on $[0, \min(T_0, \mathscr{T}\epsilon^{-2})]$ provided $\epsilon_0 > 0$ is chosen to be sufficiently small.
\end{proposition}

\begin{proof}
The proof rests on the fact that we can write $\kappa_t(\alpha, \beta, t) = b(\kappa(\alpha, \beta, t), t)$ and that we have the formula of Proposition \ref{BADtAFormulas}(a) for $b$ in terms of $\zeta$ and its derivatives.  Recall that $\kappa(\alpha, \beta, 0) = \alpha \ivec + \beta \jvec$.  By integrating with respect to $t$, we have
\begin{align}\label{KappaPicardFormula}
\kappa(\alpha, \beta, t) - (\alpha\ivec + \beta\jvec) & = \int_0^t \tilde{b}(\kappa(\alpha, \beta, s), s) \, ds + \int_0^t (b - \tilde{b})(\kappa(\alpha, \beta, s), s) \, ds
\end{align}
Recall the definition of $\tilde{b}$ as given by \eqref{TildeBFormula}, which shows that $\tilde{b} = S\epsilon^2 + F\epsilon^3$, where $S$ is a function of the slow variables alone.  Thus $\nabla \tilde{b}$ is of physical size $O(\epsilon^3)$, and so Proposition \ref{DiffQuantitiesEstimates} gives us the estimate
\begin{align*}
\|\nabla\kappa - I\|_{W^{s - 2, \infty}} & \leq \int_0^t \| (\nabla \tilde{b} \circ \kappa) \cdot \nabla \kappa\|_{W^{s - 2, \infty}} + \| (\nabla (b - \tilde{b}) \circ \kappa) \cdot \nabla \kappa\|_{W^{s - 2, \infty}} \, ds \\
& \leq \mathscr{T}\epsilon^{-2}C(1 + \|\nabla\kappa - I\|_{L^\infty})\epsilon^3 \\
& \leq C\epsilon(1 + \|\nabla\kappa - I\|_{W^{s - 2, \infty}}).
\end{align*}
from which the proposition follows upon choosing $\epsilon_0 > 0$ sufficiently small.
\end{proof}

\begin{remark}\label{BigHorizontalError}
Note that if we were to expect $E^\frac12 = O(\epsilon)$, then the above proof allows us to conclude only that $\nabla(\kappa - P) = O(1)$, which is insufficient to justify inverting $\kappa$ with the Inverse Function Theorem.  Alternatively, if $E^\frac12 = O(\epsilon^2)$ as we intend, then the $b - \tilde{b}$ component of \eqref{KappaPicardFormula} contributes an error of size $O(\epsilon)$.  This is the fact that prevents us from justifying asymptotics for the horizontal component of $\Xi$ in Lagrangian coordinates.  Observe that if we could show that $E^\frac12 = o(\epsilon^2)$, we could in fact justify asymptotics for the horizontal component of $\Xi$ directly.
\end{remark}

	\subsection{Preliminary Energy Identities}

We will build our energy using the following basic energy identity, following \cite{WuGlobal3D}:

\begin{proposition}\label{BasicEnergyIdentity}
Suppose $\theta \in \mathscr{S}(\mathbb{H}\,)$ satisfies $\theta = -\nht\theta$.  Introduce the energy
$$\mathscr{E}(\theta) = \iint \frac{1}{\mathcal{A}}|D_t\theta|^2 - \theta \cdot (\mathcal{N} \times \nabla)\theta \, d\alpha\, d\beta$$  Then
\begin{align*}
\frac{d\mathscr{E}(\theta)}{dt} & = \iint \frac{2}{\mathcal{A}}D_t \theta \cdot \left(D_t^2 - \mathcal{A}(\zeta_\beta\partial_\alpha - \zeta_\alpha \partial_\beta)\right)\theta \, d\alpha\, d\beta \\
& \qquad + \iint - \frac{1}{\mathcal{A}} U_{\kappa^{-1}}\left(\frac{\mathfrak{a}_t}{\mathfrak{a}}\right)|D_t\theta|^2 - \theta \cdot \left((\partial_\beta D_t\zeta)\theta_\alpha - (\partial_\alpha D_t\zeta)\theta_\beta\right) \, d\alpha \, d\beta
\end{align*}
Moreover, $$\iint - \theta \cdot (\mathcal{N} \times \nabla)\theta \, d\alpha\, d\beta \geq 0$$
\end{proposition}

\begin{proof}  The energy $\mathscr{E}$ is that of Lemma 3.2 of \cite{WuGlobal3D}.  The proof is the same except that, using integration by parts and Proposition \ref{SkewedTripleProduct}, we rewrite $$-\iint \theta \cdot (\mathcal{N} \times \nabla)D_t\theta \, d\alpha \, d\beta = -\iint (\mathcal{N} \times \nabla)\theta \cdot D_t\theta \, d\alpha \, d\beta$$
\end{proof}

Our starting point is to consider the energy
\begin{equation}\label{EnergyFirstTry}
\mathcal{E} = \sum_{|j| \leq s} \mathscr{E}(\rho^{[j]}) + \mathscr{E}(\sigma^{[j]})
\end{equation}
where we use the modified quantities
\begin{equation}\label{RhoSigmaJDefinition}
\theta^{[j]} = \frac12(I - \nht)\partial^j \theta
\end{equation}
to ensure that the energy is nonnegative.  However, using this energy contributes terms in the energy identity that are of third order, instead of the fourth and higher orders necessary for a suitable energy estimate.  Therefore we cannot use this energy directly.  Since the third order terms contributed by this energy motivate our choice of normal form transformation, we spend the rest of the section identifying these terms.  We will need the commutator identity
\begin{align}\label{ANTimesNablaCommutePartials}
[\partial^j, \mathcal{A}(\mathcal{N} \times \nabla)] & = \mathcal{A}\sum_{j_1 + j_2 = j}^{j_2 < j} \left((\partial^{j_1}\lambda_\beta)\partial_\alpha\partial^{j_2} - (\partial^{j_1}\lambda_\alpha)\partial_\beta\partial^{j_2}\right) \notag \\
& \quad + \left(\sum_{j_1 + j_2 = j}^{j_2 < j} (\partial^{j_1}(\mathcal{A} - 1))\partial^{j_2}\right)(\mathcal{N} \times \nabla)
\end{align}
Denote $\mathcal{P}\theta = G_\theta$; applying Proposition \ref{BasicEnergyIdentity}, we have
\begin{align}
\frac{d\mathcal{E}}{dt} & = \sum_{|j| \leq s} \iint \frac{2}{\mathcal{A}}D_t \rho^{[j]} \cdot \left(D_t^2 - \mathcal{A}(\zeta_\beta\partial_\alpha - \zeta_\alpha\partial_\beta)\right)\rho^{[j]} \, d\alpha\, d\beta \notag \\
& \qquad + \iint - \frac{1}{\mathcal{A}} U_{\kappa^{-1}}\left(\frac{\mathfrak{a}_t}{\mathfrak{a}}\right)|D_t\rho^{[j]}|^2 - \rho^{[j]} \cdot \left((\partial_\beta D_t\zeta)\rho^{[j]}_\alpha - (\partial_\alpha D_t\zeta)\rho^{[j]}_\beta\right) \, d\alpha \, d\beta \notag \\
& \qquad + \iint \frac{2}{\mathcal{A}}D_t \sigma^{[j]} \cdot \left(D_t^2 - \mathcal{A}(\zeta_\beta\partial_\alpha - \zeta_\alpha \partial_\beta)\right)\sigma^{[j]} \, d\alpha\, d\beta \notag \\
& \qquad + \iint - \frac{1}{\mathcal{A}} U_{\kappa^{-1}}\left(\frac{\mathfrak{a}_t}{\mathfrak{a}}\right)|D_t\sigma^{[j]}|^2 - \sigma^{[j]} \cdot \left((\partial_\beta D_t\zeta)\sigma^{[j]}_\alpha - (\partial_\alpha D_t\zeta)\sigma^{[j]}_\beta\right) \, d\alpha \, d\beta \notag
\end{align}
The first terms which are of third order are
\begin{equation}\label{QuadraticTermNo1}
\iint - \sigma^{[j]} \cdot \left((\partial_\beta D_t\zeta)\sigma^{[j]}_\alpha - (\partial_\alpha D_t\zeta)\sigma^{[j]}_\beta\right) \, d\alpha \, d\beta
\end{equation}
\begin{equation}\label{QuadraticTermNo2}
\iint - \rho^{[j]} \cdot \left((\partial_\beta D_t\zeta)\rho^{[j]}_\alpha - (\partial_\alpha D_t\zeta)\rho^{[j]}_\beta\right) \, d\alpha \, d\beta
\end{equation}
The other terms of third order must be extracted from the nonlinearities in the first and third lines above.  For $\theta = \rho, \sigma$ we first have
\begin{align*}
\left(D_t^2 - \mathcal{A}(\zeta_\beta\partial_\alpha - \zeta_\alpha\partial_\beta)\right)\theta^{[j]} & = -\frac{1}{2}[D_t^2 - \mathcal{A}(\zeta_\beta\partial_\alpha - \zeta_\alpha\partial_\beta), \nht]\partial^j\theta \\
& \quad + \frac{1}{2}(I - \nht)[D_t^2 - \mathcal{A}(\zeta_\beta\partial_\alpha - \zeta_\alpha\partial_\beta), \partial^j] \theta \\
& \quad + \frac12(I - \nht)\partial^jG_\theta
\end{align*}
The first line on the right hand side appears to contribute a term of third order, but we will show later that we can gain an extra order of smallness using an almost-orthogonality argument.  Using \eqref{ANTimesNablaCommutePartials}, we see that the second line contributes third order terms of the form
\begin{equation}\label{QuadraticTermNo5}
\iint 2 \frac{D_t\theta^{[j]}}{\mathcal{A}} \cdot-\frac12 (I - \nht)\left(\mathcal{A}\sum_{j_1 + j_2 = j}^{j_2 < j} \left((\partial^{j_1}\lambda_\beta)(\partial^{j_2}\theta_\alpha) - (\partial^{j_1}\lambda_\alpha)(\partial^{j_2}\theta_\beta)\right)\right) \, d\alpha \, d\beta
\end{equation}
We must treat the analysis of $G_\theta$ separately for $\theta = \rho, \sigma$.  Expanding the term $G_\rho$ using \eqref{RhoEquation} further yields
\begin{align*}
\frac12(I - \nht)\partial^jG_\rho & = -\frac{1}{4}(I - \nht)\partial^j[\mathcal{P}, \nht]\left((I - \nht)\znew\kvec - (I - \nht_{\tilde{\zeta}})\tilde{\znew}\kvec\right) \notag \\
& \quad - \frac{1}{4}(I - \nht)\partial^j(I - \nht)(\mathcal{P} - \tilde{\mathcal{P}})(I - \nht_{\tilde{\zeta}})\tilde{\znew}\kvec \notag \\
& \quad + \frac14(I - \nht)\partial^j(I - \nht)\left(G - \tilde{\mathcal{P}}(I - \nht_{\tilde{\zeta}})\tilde{\znew}\kvec\right) \notag
\end{align*}
Using \eqref{DiffPFormula}, we see that the second of these terms contributes another term of third order:
\begin{equation}\label{QuadraticTermNo4}
\iint 2 \frac{D_t\rho^{[j]}}{\mathcal{A}} \cdot \frac{1}{4}(I - \nht)\partial^j(I - \nht)\left(\tilde{\mathcal{A}}(r_\beta \partial_\alpha - r_\alpha \partial_\beta)(I - \nht_{\tilde{\zeta}})\tilde{\znew}\kvec\right) \, d\alpha \, d\beta
\end{equation}
To find the other third order terms, we expand using \eqref{SigmaEquation}:
\begin{align*}
\frac12(I - \nht)\partial^jG_\sigma & = -\frac14(I - \nht)[\mathcal{P}, \nht]\left(D_t(I - \nht)\znew\kvec - \tilde{D}_t(I - \nht_{\tilde{\zeta}})\tilde{\znew}\kvec\right) \notag \\
& \quad - \frac14(I - \nht)\partial^j(I - \nht)\left((\mathcal{P} - \tilde{\mathcal{P}})\tilde{D}_t(I - \nht_{\tilde{\zeta}})\tilde{\znew}\kvec\right) \notag \\
& \quad + \frac14(I - \nht)\partial^j(I - \nht)\left([\mathcal{P}, D_t](I - \nht)\znew\kvec - [\tilde{\mathcal{P}}, \tilde{D}_t](I - \nht_{\tilde{\zeta}})\tilde{\znew}\kvec\right) \notag \\
& \quad + \frac14(I - \nht)\partial^j(I - \nht)\left(D_t\mathcal{P}(I - \nht)\znew\kvec - \tilde{D}_t\tilde{\mathcal{P}}(I - \nht_{\tilde{\zeta}})\tilde{\znew}\kvec\right) \notag
\end{align*}
From the second line of the last expression above we have the third order term
\begin{equation}\label{QuadraticTermNo6}
\iint 2 \frac{D_t\sigma^{[j]}}{\mathcal{A}} \cdot \frac14(I - \nht)\partial^j(I - \nht)\left(\mathcal{A}(r_\beta\partial_\alpha - r_\alpha \partial_\beta)\tilde{D}_t(I - \nht_{\tilde{\zeta}})\tilde{\znew}\kvec\right) \, d\alpha \, d\beta
\end{equation}
and the remaining terms of third order are contained in the term
\begin{equation}\label{QuadraticTermNo7}
\iint 2 \frac{D_t\sigma^{[j]}}{\mathcal{A}} \cdot \frac14(I - \nht)\partial^j(I - \nht)\left([\mathcal{P}, D_t](I - \nht)\znew\kvec - [\tilde{\mathcal{P}}, \tilde{D}_t](I - \nht_{\tilde{\zeta}})\tilde{\znew}\kvec\right) \, d\alpha \, d\beta
\end{equation}

	\subsection{The Normal Form Calculation}

To account for the terms of third order isolated in the last section, we use the method of normal forms.  This method, due to Poincar\'e in the context of ODEs, was first pioneered for use in PDEs by Shatah \cite{ShatahNormalForm} and has since been used widely in the study of long-time existence questions for evolution equations.  In particular, this method has been used to great effect in justification of model equations for water waves by Schneider and Wayne in \cite{SchneiderWayveJustifyKdV} and \cite{SchneiderWayneNLSJustify}.  Rather than using the quantities $\theta = \rho^{[j]}, \sigma^{[j]}$ in \eqref{RhoSigmaJDefinition}, we use equivalent quantities formed by perturbing $\theta$ by a quantity $\mathcal{Q}_\theta$ intended to be quadratically small:
\begin{equation}
\Theta = \theta + \mathcal{Q}_\theta
\end{equation}
Each $\mathcal{Q}_\theta$ will be a sum of bilinear expressions involving $\theta$ corresponding to the quadratic nonlinearities which must be eliminated.  In order show systematically how to account for these quadratic contributions, we treat the simplest case first in detail and show its equivalence to $\rho$ and $\sigma$.  Then we will estimate in detail the many higher order terms neglected in deriving the formula for this normal form.   Finally, we perform the same process for the normal form corrections for the other quadratic terms, detailing only the steps that are significantly different from what has been seen at that point.

We begin with a heuristic construction of normal forms in \S 4.3.1; there we will see that the cost of constructing a normal form is roughly speaking one space derivative of smoothness.  Therefore, in order for the corrected quantities and the original quantities to have the same regularity, we must manipulate the quadratic terms to be accounted for to gain a space derivative.  

The third order terms \eqref{QuadraticTermNo5} clearly involve quantities which have a full derivative less than the total number of spatial derivatives controlled by the total energy, and so their normal forms are easily constructed in \S 4.3.2.  In \S4.3.3 we rewrite the third order terms \eqref{QuadraticTermNo4}, \eqref{QuadraticTermNo6} and \eqref{QuadraticTermNo7} using commutators so that the same is true of their resulting normal forms constructed in \S 4.3.3.  Finally, in \S4.3.4 we combine the terms \eqref{QuadraticTermNo1} and \eqref{QuadraticTermNo2} along with contributions arising from carefully chosen third order energy corrections, yielding terms that can be eliminated using normal forms.

\subsubsection{Heuristic Calculation of the Simplest Normal Form; Equivalence of Original and Transformed Unknowns}

To derive the equations that the correction $\mathcal{Q}_\theta$ must satisfy, we repeat the energy estimates of the previous section except using $\Theta$ and omitting terms which can be suitably estimated.  We have using Proposition \ref{BasicEnergyIdentity} that the energy identity will read in part:
\begin{align}\label{IsolateCancellationInNewEnergy}
\frac{d\mathscr{E}(\Theta)}{dt} & = \iint 2\frac{D_t\Theta}{\mathcal{A}} \cdot (I - \nht)\mathcal{P}(\theta + \mathcal{Q}_\theta) \, d\alpha \, d\beta + \cdots
\end{align}
For the moment, we will examine what is involved in accounting for a typical quadratic term: $$(I - \nht)(\tilde{\lambda}_\beta \theta_\alpha - \tilde{\lambda}_\alpha \theta_\beta)$$ Since we wish to use the Fourier transform to construct $\mathcal{Q}_\theta$, it will be more convenient for it to satisfy an equation involving only the ordinary time derivative $\partial_t$ instead of the more complicated convective derivative $D_t$.  To accomplish this without losing derivatives, we change variables by $\kappa$ in the above integral.  Introduce the quantity $\mathscr{Q}_\theta = \mathcal{Q}_\theta \circ \kappa$; then changing variables by $\kappa$ and using the identity $\jvec\partial_\alpha - \ivec\partial_\beta = \kvec\mathcal{D} = \nht_0|\mathcal{D}|$, we derive to leading order the integral
\begin{align}\label{NewEnergyBilinearManipulation}
& \quad \iint 2\frac{D_t\Theta}{\mathcal{A}} \cdot (I - \nht)\Biggl(\mathcal{P}\mathcal{Q}_\theta - \left(\tilde{\lambda}_\beta\partial_\alpha(\theta \circ \kappa) - \tilde{\lambda}_\alpha \partial_\beta(\theta \circ \kappa)\right)\Biggr) \, d\alpha \, d\beta \notag \\
& = \iint 2\frac{\partial_t(\Theta \circ \kappa)}{\mathfrak{a}} \cdot (I - \nht_0)\Biggl(\left(\partial_t^2 + |\mathcal{D}|\right)\mathscr{Q}_\theta \notag \\
& \hspace{2in} - \left(\mathcal{B}_{-k}\tilde{\lambda}_\beta\partial_\alpha(\theta \circ \kappa) - \mathcal{B}_{-k}\tilde{\lambda}_\alpha \partial_\beta(\theta \circ \kappa)\right)\Biggr) \, d\alpha \, d\beta + \text{h.o.t.}
\end{align}
where we have applied the mode filter $\mathcal{B}_{-k}$ of Lemma \ref{WavePacketCutoff}.\footnote{See the next section for a more precise accounting of the higher order terms that have been neglected here.}  Hence it suffices to construct $\mathscr{Q}_\theta$ so as to satisfy
\begin{equation}\label{BilinearFormEquation}
\left(\partial_t^2 + |\mathcal{D}|\right)\mathscr{Q}_\theta = \mathcal{B}_{-k}\tilde{\lambda}_\beta\partial_\alpha(\theta \circ \kappa) - \mathcal{B}_{-k}\tilde{\lambda}_\alpha \partial_\beta(\theta \circ \kappa)
\end{equation}
Observe that if we take the left-hand $\jvec$-Fourier transform $\mathcal{F}_\jvec^L[\cdot]$ of the right hand side we can rewrite it in the form
\begin{align}\label{PreparingQuadraticTerm}
& \quad \mathcal{F}_\jvec^L\left[\mathcal{B}_{-k}\tilde{\lambda}_\beta\partial_\alpha(\theta \circ \kappa) - \mathcal{B}_{-k}\tilde{\lambda}_\alpha \partial_\beta(\theta \circ \kappa)\right] \notag \\
& = \mathcal{F}_\jvec^L\left[\partial_\beta(\epsilon \mathcal{B}_{-k}\overline{A}e^{-\jvec\phi})\partial_\alpha(\ivec\theta \circ \kappa) - \partial_\alpha(\epsilon \mathcal{B}_{-k}\overline{A}e^{-\jvec\phi})\partial_\beta(\ivec\theta \circ \kappa)\right] + \text{h.o.t.} \notag \\
& = \frac{1}{(2\pi)^2} \iint \mathcal{F}_\jvec^L[\epsilon \mathcal{B}_{-k}\overline{A}e^{-\jvec\phi}]_{(\xi - \xi^\prime)}\Bigl((\xi_1 - \xi_1^\prime)\xi_2^\prime -(\xi_2 - \xi_2^\prime)\xi_1^\prime\Bigr)\mathcal{F}_\jvec^L[\ivec\theta \circ \kappa]_{(\xi^\prime)} \, d\xi^\prime + \text{h.o.t.} \\
& = \frac{1}{(2\pi)^2} \iint \mathcal{F}_\jvec^L[\epsilon \mathcal{B}_{-k}\overline{A}e^{-\jvec\phi}]_{(\xi - \xi^\prime)}\left(-k\xi_2^\prime \right)\mathcal{F}_\jvec^L[\ivec\theta \circ \kappa]_{(\xi^\prime)} \, d\xi^\prime + \text{h.o.t.} \notag
\end{align}
This suggests the following ansatz for $\mathscr{Q}_\theta$ in Fourier space:
\begin{align}\label{BilinearFormAnsatz}
\hat{\mathscr{Q}}_\theta = \hat{\mathscr{Q}}(\tilde{\lambda}, \theta)& = \frac{1}{(2\pi)^2} \iint \mathcal{F}_\jvec^L[\epsilon \mathcal{B}_{-k}\overline{A}e^{-\jvec\phi}]_{(\xi - \xi^\prime)}Q_0(\xi, \xi - \xi^\prime, \xi^\prime)\mathcal{F}_\jvec^L[\ivec\theta \circ \kappa]_{(\xi^\prime)} \, d\xi^\prime \notag \\
& \qquad \quad + \frac{1}{(2\pi)^2} \iint \mathcal{F}_\jvec^L[\epsilon \mathcal{B}_{-k}\overline{A}e^{-\jvec\phi}]_{(\xi - \xi^\prime)}Q_1(\xi, \xi - \xi^\prime, \xi^\prime)\mathcal{F}_\jvec^L[\ivec D_t \theta \circ \kappa]_{(\xi^\prime)} \, d\xi^\prime
\end{align}
Note that our unknown loses a half space derivative in constructing a normal form, since we are adding to $\theta$ a correction with the smoothness of $D_t\theta$.

If one substitutes \eqref{BilinearFormAnsatz} into \eqref{BilinearFormEquation}, and makes the substitutions $\partial_t^2(\theta \circ \kappa) \sim -|\mathcal{D}|(\theta \circ \kappa)$ and $|\xi - \xi^\prime + k\ivec| \sim 0$ which contributes negligible higher order terms, one finds that the kernels $Q_0, Q_1$ satisfy a system of the form:
\begin{align}\label{BilinearKernelSystem}
(|\xi^\prime - k\ivec| - |\xi^\prime| - |k\ivec|)Q_0 + 2\jvec\omega |\xi^\prime| Q_1 & = F_0\\
(|\xi^\prime - k\ivec| - |\xi^\prime| - |k\ivec|)Q_1 - 2\jvec\omega Q_0 & = F_1
\end{align}
with $F_0 = k\xi_2^\prime$ and $F_1 = 0$.  The solution to this system for general $F_0, F_1$ is
\begin{equation}\label{Q0GeneralSolution}
Q_0 = \frac{(|\xi^\prime - k\ivec| - |\xi^\prime| - |k\ivec|) F_0 - 2\jvec\omega |\xi^\prime|F_1}{(|\xi^\prime - k\ivec| - |\xi^\prime| - |k\ivec|)^2 - 4k|\xi^\prime|}
\end{equation}
\begin{equation}\label{Q1GeneralSolution}
Q_1 = \frac{2\jvec\omega F_0 + ((|\xi^\prime - k\ivec| - |\xi^\prime| - |k\ivec|)F_1}{(|\xi^\prime - k\ivec| - |\xi^\prime| - |k\ivec|)^2 - 4k|\xi^\prime|}
\end{equation}
and in the present case is
\begin{equation}\label{Q0ParticularSolution1}
Q_0 = \frac{((|\xi^\prime - k\ivec| - |\xi^\prime| - |k\ivec|) (-k\xi_2^\prime)}{(|\xi^\prime - k\ivec| - |\xi^\prime| - |k\ivec|)^2 - 4k|\xi^\prime|}
\end{equation}
\begin{equation}\label{Q1ParticularSolution1}
Q_1 = \frac{2\jvec\omega (-k\xi_2^\prime)}{(|\xi^\prime - k\ivec| - |\xi^\prime| - |k\ivec|)^2 - 4k|\xi^\prime|}
\end{equation}
Taking for granted this formal expression for the change of variables for the moment, we now turn to studying its regularity.  This amounts to studying the asymptotic behavior of the above kernels as $\xi, \xi^\prime \to \infty$, regarding $\xi - \xi^\prime$ as remaining bounded.  More precisely we have the
\begin{lemma}\label{BilinearFormEstimate}
Let $0 \leq q < 12$, let $p : \mathbb{R}^+ \to \mathbb{R}^+$ be in $L^1_{\text{loc}}$, let $S$ be a function of slow variables alone, and let $\mathscr{Q}$ be defined by
$$\mathcal{F}_\jvec^L[\mathscr{Q}] = \iint \mathcal{F}_\jvec^L[\epsilon \mathcal{B}_{-k}\overline{S}e^{-\jvec\phi}]_{(\xi - \xi^\prime)}Q(\xi, \xi - \xi^\prime, \xi^\prime)\mathcal{F}_\jvec^L[\theta]_{(\xi^\prime)} \, d\xi^\prime$$ and suppose that $|Q(\xi, \xi - \xi^\prime, \xi^\prime)| \leq |\xi - \xi^\prime - k\ivec |^q p(|\xi^\prime|)$.  Then $\|\mathscr{Q}\|_{L^2} \leq C\epsilon^{q + 1} \|S\|_{H^{q + 3}} \|p(|\mathcal{D}|)\theta\|_{L^2}$.
\end{lemma}
\begin{proof}
By Parseval's Identity and Young's Inequality we have 
\begin{align*}
\|\mathscr{Q} \|_{L^2} & \leq \| (|\xi - k\ivec|^q |\mathcal{F}_\jvec^L[\epsilon \mathcal{B}_{-k}\overline{S}e^{-\jvec\phi}]|) \star (p(|\xi|)| |\mathcal{F}_\jvec^L[\theta]|) \|_{L^2} \\
& \leq C\| |\mathcal{F}_\jvec^L[\epsilon \mathcal{B}_{-k}\overline{S}e^{-\jvec\phi}]|\|_{L^1} \| p(|\xi|) |\mathcal{F}_\jvec^L[\theta]| \|_{L^2}
\end{align*}
The lemma now follows since, as in Lemma \ref{WavePacketCutoff}, we can write
\begin{align*}
\| |\xi - k\ivec|^q |\mathcal{F}_\jvec^L[\epsilon \mathcal{B}_{-k}\overline{S}e^{-\jvec\phi}]|\|_{L^1} & = \left\||\xi - k\ivec|^q \frac{1}{\epsilon}\mathcal{F}_\jvec^L[\overline{S}]\left(\frac{\xi - k\ivec}{\epsilon}\right)\right\|_{L^1} \\
& = \epsilon^{q + 1} \left\|\mathcal{F}_\jvec^L[|\mathcal{D}|^q\overline{S}]\right\|_{L^1} \\
& \leq C\epsilon^{q + 1} \|S\|_{H^{q + 3}} 
\end{align*}
\end{proof}
In order to use this lemma, we must further analyze the singularities of $Q_0$ and $Q_1$, which in this context are called resonances.  In order to understand these resonances, we record the
\begin{lemma}\label{DenominatorEstimate}
\begin{itemize}
 \item[(a)]{There exists a universal constant $C_0$ so that $$\frac{1}{C_0}\frac{|\xi| + |\xi^\prime| + |\xi - \xi^\prime|}{|\xi|\,|\xi^\prime|\,|\xi - \xi^\prime|} \leq \frac{1}{(|\xi| -|\xi - \xi^\prime| - |\xi^\prime|)^2 - 4|\xi^\prime|\,|\xi - \xi^\prime|} \leq C_0\frac{|\xi| + |\xi^\prime| + |\xi - \xi^\prime|}{|\xi|\,|\xi^\prime|\,|\xi - \xi^\prime|}$$}
 \item[(b)]{$$|\, |\xi| - |\xi^\prime| - |\xi - \xi^\prime|\, | \leq \min\left(2|\xi - \xi^\prime|, |\xi^\prime|\right)$$}
 \item[(c)]{\begin{align*}
\left| \frac{\xi_1\xi_2^\prime - \xi_2 \xi_1^\prime}{(|\xi| - |\xi^\prime| - |\xi - \xi^\prime|)^2 - 4|\xi^\prime|\,|\xi - \xi^\prime|}\right| & \leq C_0\frac{|\xi_1\xi_2^\prime - \xi_2 \xi_1^\prime|}{|\xi|\,|\xi^\prime|} \\
& \quad + C_0\frac{|\xi_1(\xi_2 - \xi_2^\prime) - \xi_2(\xi_1 - \xi_1^\prime)|}{|\xi|\,|\xi - \xi^\prime|} \\
& \quad + C_0\frac{|\xi_1^\prime(\xi_2 - \xi_2^\prime) - \xi_2^\prime(\xi_1 - \xi_1^\prime)|}{|\xi^\prime|\,|\xi - \xi^\prime|}
\end{align*}}
\end{itemize}
\end{lemma}
\begin{proof}
Parts (a) and (b) are given in Appendix C of \cite{WuGlobal3D}.  Part (c) is a simple consequence of (a) and the triangle inequality.
\end{proof}
Applying Lemma \ref{DenominatorEstimate} to the kernels of the $\mathcal{Q}_\theta = \mathscr{Q}_\theta \circ \kappa^{-1}$ as defined in \eqref{BilinearFormAnsatz}-\eqref{Q0ParticularSolution1}-\eqref{Q1ParticularSolution1} implies that $Q_0$ and $Q_1$ are uniformly bounded in $\xi$ and $\xi^\prime$.  Then Lemma \ref{BilinearFormEstimate} immediately gives
\begin{proposition}\label{NaiveNormalFormEstimate}
For $\mathcal{Q}_\theta$ defined as above we have the estimates
\begin{equation*}
\|\,|\mathcal{D}|^\frac12\mathcal{Q}_\theta\|_{H^\frac12} + \|D_t\mathcal{Q}_\theta\|_{H^\frac12} + \|D_t^2\mathcal{Q}_\theta\|_{L^2} \leq C\epsilon(\|\,|\mathcal{D}|^\frac12\theta\|_{H^\frac12} + \|D_t\theta\|_{H^\frac12} + \|D_t^2 \theta\|_{L^2})
\end{equation*}
\end{proposition}

	\subsubsection{Construction of the Normal Form and Estimates of the Higher Order Corrections Corresponding to \eqref{QuadraticTermNo5}}
	
Here we construct the normal forms corresponding to \eqref{QuadraticTermNo5}, and also carefully control the higher order terms neglected in the heuristic calculation of the previous section.

First we write
\begin{align*}
\mathcal{A}(\partial^{j_1}\lambda_\beta \partial^{j_2}\partial_\alpha - \partial^{j_1}\lambda_\alpha \partial^{j_2}\partial_\beta)\sigma & = (\mathcal{A} - 1)(\partial^{j_1}\lambda_\beta \partial^{j_2}\sigma_\alpha - \partial^{j_1}\lambda_\alpha \partial^{j_2}\sigma_\beta) \\
& \quad + (\partial^{j_1}r_\beta \partial^{j_2} \sigma_\alpha - \partial^{j_1}r_\alpha \partial^{j_2}\sigma_\beta) \\
& \quad + (\partial^{j_1}(I - \mathcal{B}_{-k})\tilde{\lambda}_\beta \partial^{j_2}\sigma_\alpha - \partial^{j_1}(I - \mathcal{B}_{-k})\tilde{\lambda}_\alpha \partial^{j_2}\sigma_\beta) \\
& \quad + (\partial^{j_1}\mathcal{B}_{-k}\tilde{\lambda}_\beta \partial^{j_2}\sigma_\alpha - \partial^{j_1}\mathcal{B}_{-k}\tilde{\lambda}_\alpha \partial^{j_2}\sigma_\beta)
\end{align*}
Since $j_2 < j$, all of the above terms except for the last are controlled by $C(E^\frac12 + \epsilon^2)^2$ by Propositions \ref{SingularIntegralHsEstimate} and \ref{DiffQuantitiesEstimates}.  Hence it suffices to eliminate the term 
\begin{equation*}
(I - \nht)\left(\sum_{j_1 + j_2 = j}^{j_2 < j} \left((\partial^{j_1}\mathcal{B}_{-k}\tilde{\lambda}_\beta)(\partial^{j_2}\sigma_\alpha) - (\partial^{j_1}\mathcal{B}_{-k}\tilde{\lambda}_\alpha)(\partial^{j_2}\sigma_\beta)\right)\right)
\end{equation*}
To do so, we construct bilinear terms $\mathcal{Q}_{\eqref{QuadraticTermNo5}}^{(j_1, j_2)}$ defined by \eqref{BilinearFormAnsatz} with $\theta = \partial^{j_2} \sigma$ and $\tilde{\lambda}$ replaced by $\partial^{j_1} \tilde{\lambda}$.  The derivation \eqref{NewEnergyBilinearManipulation} is more precisely given by
\begin{align*}
& \quad U_\kappa(I - \nht)\Biggl(\left(D_t^2 - (\zeta_\beta \partial_\alpha - \zeta_\alpha \partial_\beta)\right)\mathcal{Q}(\partial^{j_1}\tilde{\lambda}, \partial^{j_2}\sigma) - \left(\partial^{j_1}\mathcal{B}_{-k}\tilde{\lambda}_\beta\partial^{j_2}\sigma_\alpha - \partial^{j_1}\mathcal{B}_{-k}\tilde{\lambda}_\alpha \partial^{j_2}\sigma_\beta\right)\Biggr) \notag \\
& = (I - \oht)\Biggl(\left(\partial_t^2 - \frac{\Xi_\beta}{J(\kappa)}\partial_\alpha + \frac{\Xi_\alpha}{J(\kappa)}\partial_\beta\right)\mathscr{Q}(\partial^{j_1}\tilde{\lambda}, \partial^{j_2}\sigma) \notag \\
& \qquad\qquad\qquad  - \frac{1}{J(\kappa)}\left(\partial_\beta(\partial^{j_1}\mathcal{B}_{-k}\tilde{\lambda} \circ \kappa)\partial_\alpha(\partial^{j_2}\sigma \circ \kappa) - \partial_\alpha(\partial^{j_1}\mathcal{B}_{-k}\tilde{\lambda} \circ \kappa) \partial_\beta(\partial^{j_2}\sigma \circ \kappa)\right)\Biggr) \notag \\
& = (I - \oht)\Biggl(\left(\partial_t^2 + |\mathcal{D}|\right)\mathscr{Q}(\partial^{j_1}\tilde{\lambda}, \partial^{j_2}\sigma) - \left(\mathcal{B}_{-k} \tilde{\lambda}_\beta\partial_\alpha(\partial^{j_2}\sigma \circ \kappa) - \mathcal{B}_{-k} \tilde{\lambda}_\alpha \partial_\beta(\partial^{j_2}\sigma \circ \kappa)\right)\Biggr) \notag 
\end{align*}
\begin{align}\label{NastyHigherOrderTerms}
& + (I - \oht)\left(-\left(\frac{\Xi_\beta}{J(\kappa)} - \jvec\right)\partial_\alpha + \left(\frac{\Xi_\alpha}{J(\kappa)} - \ivec\right)\partial_\beta\right)\mathscr{Q}(\partial^{j_1}\tilde{\lambda}, \partial^{j_2}\sigma) \notag \\
& + (\oht - \nht_0)(\kvec\mathcal{D} + |\mathcal{D}|)\mathscr{Q}(\partial^{j_1}\tilde{\lambda}, \partial^{j_2}\sigma) \\
& - \left(\frac{1}{J(\kappa)} - 1\right)\left(\partial_\beta(\partial^{j_1}\mathcal{B}_{-k}\tilde{\lambda} \circ \kappa)\partial_\alpha(\partial^{j_2}\sigma \circ \kappa) - \partial_\alpha(\partial^{j_1}\mathcal{B}_{-k}\tilde{\lambda} \circ \kappa) \partial_\beta(\partial^{j_2}\sigma \circ \kappa)\right) \notag \\
& - \left(\partial_\beta(\partial^{j_1}\mathcal{B}_{-k}\tilde{\lambda} \circ \kappa - \partial^{j_1}\mathcal{B}_{-k}\tilde{\lambda})\partial_\alpha(\partial^{j_2}\sigma \circ \kappa) - \partial_\alpha(\partial^{j_1}\mathcal{B}_{-k}\tilde{\lambda} \circ \kappa - \partial^{j_1}\mathcal{B}_{-k}\tilde{\lambda}) \partial_\beta(\partial^{j_2}\sigma \circ \kappa)\right) \notag
\end{align}
Using Propositions \ref{KappaControlledByZeta}, \ref{NaiveNormalFormEstimate}, and the Mean Value Theorem, the higher-order terms given by the last four lines above are bounded by $C\epsilon^2 E^\frac12$ in $L^2$.  It is here that we need a full derivative of smoothness in order to justify constructing a normal form.  As $j_2 < j$, the equation \eqref{SigmaEquation} also yields the estimate
$$\|(\partial_t^2 + |\mathcal{D}|)(\partial^{j_2}\sigma \circ \kappa)\|_{L^2} \leq C(E^\frac12 + \epsilon)E^\frac12$$

Finally, recall that in the process of applying $\partial_t^2$ and $|\mathcal{D}|$ to $\mathscr{Q}_\theta$ we introduced error by replacing $\xi- \xi^\prime$ by $k\ivec$.  This is justified by Lemma \ref{BilinearFormEstimate}, since $|\xi| - |\xi^\prime - k\ivec| \leq |\xi - \xi^\prime + k\ivec|$.  The estimates of the error terms resulting from replacing $|\xi - \xi^\prime|$ with $k$ and $\partial_t \widehat{\overline{A}e^{-\jvec\phi}}(\xi - \xi^\prime)$ by $-\jvec\omega \widehat{\overline{A}e^{-\jvec\phi}}(\xi - \xi^\prime)$ is justified similarly.  Notice that  since the quantities $\sigma^{[j_2]}$ have one fewer derivative than the total energy controls, the terms resulting from making such replacements on derivatives of $\widehat{\overline{A}e^{-\jvec\phi}}(\xi - \xi^\prime)$ always yield error terms controlled by $$C\epsilon^2(\|\sigma^{[j_2]}\|_{L^2} + \|D_t\sigma^{[j_2]}\|_{L^2} + \|D_t^2\sigma^{[j_2]}\|_{L^2}) \leq C\epsilon^2 E^\frac12$$  Hence with this choice of $\mathcal{Q}_{\eqref{QuadraticTermNo5}}^{(j_1, j_2)}$ added to $\sigma^{[j_2]}$ the term \eqref{QuadraticTermNo5} is eliminated from the energy identity at the expense of finitely many terms of size at most $C(E^\frac12 + \epsilon^2)^2$ in $L^2$

	\subsubsection{The Normal Forms for \eqref{QuadraticTermNo4}, \eqref{QuadraticTermNo6}, and \eqref{QuadraticTermNo7}}

The quadratic term \eqref{QuadraticTermNo5} was successfully eliminated since its regularity was one derivative less than that of the total energy.  While the quadratic terms \eqref{QuadraticTermNo4}, \eqref{QuadraticTermNo6}, and \eqref{QuadraticTermNo7} seem to be rougher than this, one can increase their regularity by exploiting the commutator structure of these terms.  Since \eqref{QuadraticTermNo7} contains terms having the worst regularity, we will derive the normal form for this term only in detail and only mention the necessary changes needed to treat the others.  We will see that the case $|j| < s$ follows by the same argument used to treat the top order derivatives, and so for the remainder of the section we assume $|j| = s$.

First we prepare \eqref{QuadraticTermNo7} to motivate our choice of bilinear form as we did in \eqref{PreparingQuadraticTerm}.  Let $\sim$ indicate that we have omitted terms of size at most $C(E^\frac12 + \epsilon^2)^3$.  Then we expand \eqref{QuadraticTermNo7}:
\begin{align*}
& \quad \iint \frac{D_t\sigma^{[j]}}{\mathcal{A}} \cdot \frac12(I - \nht)\partial^j(I - \nht)\left([\mathcal{P}, D_t](I - \nht)\znew\kvec - [\tilde{\mathcal{P}}, \tilde{D}_t](I - \nht_{\tilde{\zeta}})\tilde{\znew}\kvec\right) \, d\alpha \, d\beta \\
& \sim \iint \frac{D_t\sigma^{[j]}}{\mathcal{A}} \cdot (I - \nht)\partial^j\left([\mathcal{P}, D_t](I - \nht)\znew\kvec - [\tilde{\mathcal{P}}, \tilde{D}_t](I - \nht_{\tilde{\zeta}})\tilde{\znew}\kvec\right) \, d\alpha \, d\beta
\end{align*}

We expand using \eqref{DtCommuteP} and break the estimates into cases, depending on whether all of the derivatives $\partial^j$ above fall on the factor $(I - \nht_{\tilde{\zeta}})\tilde{\znew}\kvec$:

\begin{align*}
& \sim \iint \frac{D_t\sigma^{[j]}}{\mathcal{A}} \cdot -(I - \nht)\left((D_t\mathcal{A})(\zeta_\beta \partial_\alpha - \zeta_\alpha \partial_\beta)\partial^j(I - \nht)\znew\kvec - \tilde{D}_t\tilde{\mathcal{A}}(\tilde{\zeta}_\beta \partial_\alpha - \tilde{\zeta}_\alpha \partial_\beta)\partial^j(I - \nht_{\tilde{\zeta}})\tilde{\znew}\kvec\right) \, d\alpha \, d\beta \\
& + \iint \frac{D_t\sigma^{[j]}}{\mathcal{A}} \cdot -(I - \nht)\left(\mathcal{A}(D_t \zeta_\beta \partial_\alpha - D_t \zeta_\alpha \partial_\beta)\partial^j(I - \nht)\znew\kvec - \tilde{\mathcal{A}}(\tilde{D}_t \tilde{\zeta}_\beta \partial_\alpha - \tilde{D}_t \tilde{\zeta}_\alpha \partial_\beta)\partial^j(I - \nht_{\tilde{\zeta}})\tilde{\znew}\kvec\right) \, d\alpha \, d\beta \\
& + \sum_{j_1 + j_2 = j}^{j_1 \neq 0} \iint \frac{D_t\sigma^{[j]}}{\mathcal{A}} \cdot -(I - \nht)\Biggl(\mathcal{A}\partial^{j_1}U_\kappa^{-1}\left(\frac{\mathfrak{a}_t}{\mathfrak{a}}\right)(\zeta_\beta \partial_\alpha - \zeta_\alpha \partial_\beta)\partial^{j_2}(I - \nht)\znew\kvec \\
& \hspace{5.5cm} - \tilde{\mathcal{A}}\partial^{j_1}U_{\tilde{\kappa}}^{-1}\left(\frac{\tilde{\mathfrak{a}}_t}{\tilde{\mathfrak{a}}}\right)(\tilde{\zeta}_\beta \partial_\alpha - \tilde{\zeta}_\alpha \partial_\beta)\partial^{j_2}(I - \nht_{\tilde{\zeta}})\tilde{\znew}\kvec\Biggr) \, d\alpha \, d\beta \\
& + \sum_{j_1 + j_2 = j}^{j_1 \neq 0} \iint \frac{D_t\sigma^{[j]}}{\mathcal{A}} \cdot -(I - \nht)\Bigl(\mathcal{A}(\partial^{j_1}\partial_\beta D_t \zeta \partial_\alpha - \partial^{j_1}\partial_\alpha D_t \zeta \partial_\beta)\partial^{j_2}(I - \nht)\znew\kvec \\
& \hspace{5.5cm} - \tilde{\mathcal{A}}(\partial^{j_1}\partial_\beta \tilde{D}_t \tilde{\zeta} \partial_\alpha - \partial^{j_1} \partial_\alpha \tilde{D}_t \tilde{\zeta} \partial_\beta)\partial^{j_2}(I - \nht_{\tilde{\zeta}})\tilde{\znew}\kvec\Biggr) \, d\alpha \, d\beta \\
& := \mathscr{I}_1 + \mathscr{I}_2 + \mathscr{I}_3 + \mathscr{I}_4 
\end{align*}
We further manipulate these terms in steps.

\vspace{0.5cm}

\noindent \textbf{Estimates of $\mathscr{I}_1$.}

\vspace{0.5cm}

Taking a derivative with respect to $D_t$ of the formula \eqref{DiffAFormula} and estimating as in Proposition \ref{DiffQuantitiesEstimates}(c) yields the estimates $\|D_t\mathcal{A}\|_{H^2} \leq C(E + \epsilon E^\frac12 + \epsilon^2)$ and $\|D_t(\mathcal{A} - \tilde{\mathcal{A}})\|_{H^2} \leq C(E + \epsilon E^\frac12 + \epsilon^3)$.  Now decomposing all of the quantities in $\mathscr{I}_1$ into sums of approximation and remainder quantities shows that it suffices to account for the integral
$$\iint \frac{D_t\sigma^{[j]}}{\mathcal{A}} \cdot -(I - \nht)(D_t\mathcal{A} - \tilde{D}_t\tilde{\mathcal{A}})(\jvec \partial_\alpha - \ivec \partial_\beta)\partial^j(I - \nht_{\tilde{\zeta}})\tilde{\znew}\kvec) \, d\alpha \, d\beta$$  But now writing $$D_t\mathcal{A} - \tilde{D}_t\tilde{\mathcal{A}} = D_t(\mathcal{A} - \tilde{\mathcal{A}}) + \bigl((b - \tilde{b}) \cdot \mathcal{D}\bigr)(\tilde{\mathcal{A}} - 1)$$ shows that this last integral is bounded by $C(E^\frac12 + \epsilon^2)^3$.

\vspace{0.5cm}

\noindent \textbf{Estimates of $\mathscr{I}_3$.}

\vspace{0.5cm}

Recall the formula from Proposition \eqref{BADtAFormulas}, which we rewrite slightly using Proposition \ref{CommutatorIdentities} as
\begin{align}\label{GoodDtAFormula2}
(I - \nht)U_\kappa^{-1}(\mathfrak{a}_t(\Xi_\alpha \times \Xi_\beta)) & = 2D_t[D_t, \nht]D_t\zeta
\end{align}
\begin{align*}
& \quad - \iint D_t K(\zeta^\prime - \zeta)(D_t\zeta - D_t^\prime\zeta^\prime) \times (\zeta_{\beta^\prime}\partial_{\alpha^\prime} - \zeta_{\alpha^\prime}\partial_{\beta^\prime}) D_t^\prime\zeta^\prime \, d\alpha \, d\beta \\
& \quad - \iint K(\zeta^\prime - \zeta)\Bigl(((D_t\zeta - D_t^\prime\zeta^\prime) \times \partial_{\beta^\prime}D_t^\prime \zeta^\prime)\partial_{\alpha^\prime}D_t^\prime \zeta^\prime - ((D_t\zeta - D_t^\prime\zeta^\prime) \times \partial_{\alpha^\prime}D_t^\prime \zeta^\prime)\partial_{\beta^\prime}D_t^\prime \zeta^\prime \Bigr) d\alpha \, d\beta \notag
\end{align*}

By taking the third component of \eqref{GoodDtAFormula2} and using the identity $\mathcal{A} \circ \kappa = \mathfrak{a}J(\kappa)$, we can write $U_\kappa^{-1}(\mathfrak{a}_t)$ as an expression whose leading order terms are just the right hand side of \eqref{GoodDtAFormula2}.  By Proposition \ref{KappaControlledByZeta}, the only terms of the right hand side of \eqref{GoodDtAFormula2} that are not controlled by $C(E + \epsilon E^\frac12)$ are purely approximate contributions.  The only $O(\epsilon^2)$ contributions from \eqref{GoodDtAFormula2} are through the commutator $2\partial_{t_0}[\partial_{t_0}, \nht^{(1)}_1]\partial_{t_0}\lambda^{(1)} = 0$, and $U_{\tilde{\kappa}}^{-1}\left(\tilde{\mathfrak{a}}_t/\tilde{\mathfrak{a}}\right)$ consists of terms of size at most $O(\epsilon^3)$ by \eqref{TildeatOveraFormula}.  Hence it suffices to account for the purely approximate contributions of the term $U_\kappa^{-1}\left(\mathfrak{a}_t/\mathfrak{a}\right) - U_{\tilde{\kappa}}^{-1}\left(\tilde{\mathfrak{a}}_t/\tilde{\mathfrak{a}}\right)$ of size $O(\epsilon^3)$; denote these terms by $\sum_{|i| = 3} F_ie^{i\jvec\phi}$.  Since these contributions are scalar-valued, we can choose the $F_i$ to be $1, \jvec$-valued.  Since $\partial^{j_1}F_0 = O(\epsilon^4)$ when $j_1 \neq 0$, extracting the leading order of $\mathscr{I}_3$ shows that we need only bound
\begin{align*}
 \sum_{j_1 + j_2 = j}^{j_1 \neq 0} k\iint D_t\mathbf{S}^{[j]} \cdot \frac12(I - \nht_0)\epsilon^4\left(\partial^{j_1}\left(\sum_{i= -1, 1, 2, 3} F_ie^{i\jvec\phi} \right)\partial^{j_2}(Ae^{\jvec\phi}\ivec)\right) \, d\alpha \, d\beta
\end{align*}
Note that the right hand factor of the above inner product contains no factor of the form $Se^{\jvec\phi}$.  For the term corresponding to $F_{-1}$ above, note that if we add the following higher order contribution to the energy:
\begin{equation}\label{Term7EnergyCorrection}
\sum_{j_1 + j_2 = j}^{j_1 \neq 0} -k\iint \mathbf{S}^{[j]} \cdot \frac12(I - \nht_0)\epsilon^4\Bigl(\partial^{j_1}(F_{-1}e^{-\jvec\phi})\partial^{j_2}(Ae^{\jvec\phi})\ivec\Bigr) \, d\alpha \, d\beta
\end{equation} then changing variables and taking a time derivative as usual shows that we can eliminate the term corresponding to $i = -1$ at the expense of a term which is to leading order of the form
$$\sum_{j_1 + j_2 = j}^{j_1 \neq 0} -k\iint \mathbf{S}^{[j]} \cdot \frac12(I - \nht_0)\epsilon^4\partial_t\Bigl(\partial^{j_1}(F_{-1}e^{-\jvec\phi})\partial^{j_2}(Ae^{\jvec\phi})\ivec\Bigr) \, d\alpha \, d\beta$$
which is at most of size $CE^\frac12\epsilon^4$ since we have taken a time derivative of a function of slow variables alone.

The remaining terms are non resonant; if we denote these terms by $\epsilon^4\sum_{i = 2}^4 S_i e^{i\jvec\phi}$, then by adding to $\mathbf{S}^{[j]}$ the term 
\begin{equation}\label{SomeCorrectionsToS}
\epsilon^4\sum_{i = 2}^4 \frac{S_i e^{i\jvec\phi}}{ik - (i\omega)^2}
\end{equation}
the rest of the non resonant terms are eliminated as in \eqref{IsolateCancellationInNewEnergy} up to terms of size $C(E^\frac12 + \epsilon^2)^3$.

\vspace{0.5cm}

\noindent \textbf{Estimates of $\mathscr{I}_2$ and $\mathscr{I}_4$.}

\vspace{0.5cm}

The term $\mathscr{I}_2$ will be treated in the same way as $\mathscr{I}_4$, except that it is a term with a small number of derivatives.  Thus we will only provide estimates for $\mathscr{I}_4$ in detail.  Consider anew the term

\begin{align}\label{TermNo7Again}
\mathscr{I}_4 & = \sum_{j_1 + j_2 = j}^{j_1 \neq 0} \iint \frac{D_t\sigma^{[j]}}{\mathcal{A}} \cdot -(I - \nht)\Bigl(\mathcal{A}(\partial^{j_1}\partial_\beta D_t \zeta \partial_\alpha - \partial^{j_1}\partial_\alpha D_t \zeta \partial_\beta)\partial^{j_2}(I - \nht)\znew\kvec \notag \\
& \hspace{5.5cm} - \tilde{\mathcal{A}}(\partial^{j_1}\partial_\beta \tilde{D}_t \tilde{\zeta} \partial_\alpha - \partial^{j_1} \partial_\alpha \tilde{D}_t \tilde{\zeta} \partial_\beta)\partial^{j_2}(I - \nht_{\tilde{\zeta}})\tilde{\znew}\kvec\Biggr) \, d\alpha \, d\beta
\end{align}

As it stands, the term $\partial^j \partial D_t\zeta$ for $\partial = \partial_\alpha, \partial_\beta$ cannot be directly controlled by the energy; hence we rewrite the integrand so as to gain enough regularity to be able to construct a normal form transformation.  The idea is to introduce a commutator with the Hilbert transform in order to use Proposition \ref{SingularIntegralHsEstimate} to remove one derivative.  Specifically, using $D_t\zeta = \nht D_t\zeta$ and omitting directly controlled terms, we write this term as

\begin{align}\label{I4WrittenLikeACommutator}
& \quad -(I - \nht)\Bigl(\mathcal{A}(\partial^{j_1}\partial_\beta D_t \zeta \partial_\alpha - \partial^{j_1}\partial_\alpha D_t \zeta \partial_\beta)\partial^{j_2}(I - \nht)\znew\kvec\Bigr) \notag \\
& = -\Bigl(\mathcal{A}(\partial^{j_1}\partial_\beta \nht D_t \zeta \partial_\alpha - \partial^{j_1}\partial_\alpha D_t \zeta \partial_\beta)\partial^{j_2}(I - \nht)\znew\kvec\Bigr) \notag \\
& \quad + \nht \Bigl(\mathcal{A}(\partial^{j_1}\partial_\beta D_t \zeta \partial_\alpha - \partial^{j_1}\partial_\alpha \nht D_t \zeta \partial_\beta)\partial^{j_2}(I - \nht)\znew\kvec\Bigr) \\
& \sim -\Bigl( \partial^{j_1} \nht(\partial_\beta  D_t \zeta) \partial_\alpha -  \partial^{j_1} \nht(\partial_\alpha D_t \zeta) \partial_\beta \partial^{j_2}(I - \nht)\znew\kvec\Bigr) \notag \\
& \quad + \nht \Bigl((\partial^{j_1}\partial_\beta D_t \zeta \partial_\alpha - \partial^{j_1}\partial_\alpha D_t \zeta \partial_\beta)\partial^{j_2}(I - \nht)\znew\kvec\Bigr) \notag
\end{align}
Similarly, we can rewrite the approximate version of this term from $\mathscr{I}_4$ using a commutator in the exact same way with the exception of the fact that it is only true that $(I - \tilde{\nht})\tilde{D}_t\tilde{\zeta}$ is $O(\epsilon^2)$.  This contributes the following extra term:
\begin{align*}
-\tilde{\mathcal{A}}\Bigl((\partial^{j_1}(I - \tilde{\nht})(\partial_\beta \tilde{D}_t \tilde{\zeta}) \partial_\alpha - \partial^{j_1}(I - \tilde{\nht})(\partial_\alpha  \tilde{D}_t \tilde{\zeta}) \partial_\beta)\partial^{j_2}(I - \tilde{\nht})\tilde{\znew}\kvec\Bigr)
\end{align*}
It is easy to check that this term will be of physical size $O(\epsilon^5)$ if $\partial^j$ contains any $\partial_\beta$ derivatives; the integral contributed by these terms is therefore bounded by $CE^\frac12\epsilon^4$.  Thus it suffices to assume for this term that $\partial^j = \partial_\alpha^j$, and in this case we calculate that
\begin{align}\label{AnotherCorrectionToS}
& \quad -\tilde{\mathcal{A}}\Bigl((\partial_\alpha^{j_1}(I - \tilde{\nht})(\partial_\beta \tilde{D}_t \tilde{\zeta}) \partial_\alpha - \partial_\alpha^{j_1}(I - \tilde{\nht})( \partial_\alpha \tilde{D}_t \tilde{\zeta}) \partial_\beta)\partial_\alpha^{j_2}(I - \tilde{\nht})\tilde{\znew}\kvec\Bigr) \notag \\
& = \epsilon^4 F_2e^{2\jvec\phi}\ivec + O(\epsilon^5)
\end{align}
with $F_2$ being some $1, \jvec$-valued function of slow variables alone.  Since this term is non-resonant, we eliminate it with a higher order correction to $\mathbf{S}^{[j]}$ as we did with the term $\mathscr{I}_3$.

If we now subtract from \eqref{I4WrittenLikeACommutator} its approximate version and expand into expressions involving only approximate and remainder quantities as usual, all terms can be controlled by $C(E^\frac12 + \epsilon^2)^2$ except for the following:
\begin{align*}
& -\Bigl( \partial^{j_1} \nht_{\tilde{\zeta}}(\partial_\beta  D_t r) \partial_\alpha -  \partial^{j_1} \nht_{\tilde{\zeta}}(\partial_\alpha D_t r) \partial_\beta \partial^{j_2}(I - \tilde{\nht})\tilde{\znew}\kvec\Bigr) \notag \\
& + \nht_{\tilde{\zeta}} \Bigl((\partial^{j_1}\partial_\beta D_t r \partial_\alpha - \partial^{j_1}\partial_\alpha D_t r \partial_\beta)\partial^{j_2}(I - \tilde{\nht})\tilde{\znew}\kvec\Bigr) \notag \\
& - \Bigl( \partial^{j_1} \tilde{\nht}(\partial_\beta  \tilde{D}_t \tilde{\zeta}) \partial_\alpha -  \partial^{j_1} \tilde{\nht}(\partial_\alpha \tilde{D}_t \tilde{\zeta}) \partial_\beta \Bigr) \partial^{j_2}\left((I - \nht)\znew\kvec - (I - \tilde{\nht})\tilde{\znew}\kvec\right) \notag \\
& + \nht \Biggl( \Bigl((\partial^{j_1}\partial_\beta \tilde{D}_t \tilde{\zeta} \partial_\alpha - \partial^{j_1}\partial_\alpha \tilde{D}_t \tilde{\zeta} \partial_\beta)\Bigr)\partial^{j_2}\left((I - \nht)\znew\kvec - (I - \tilde{\nht})\tilde{\znew}\kvec\right)\Biggr) \notag
\end{align*}
Since $j_2 < j$, the latter pair of terms can be accounted for using the method of normal forms exactly as in the previous sections.  In anticipation of changing variables in the former pair of terms with respect to $\kappa$, we rewrite them as
\begin{align*}
& \sim \partial_\beta(\nht_{\kappa^{-1}} \partial^{j_1} D_t r)\partial_\alpha(\partial^{j_2}\tilde{\lambda}^\dagger \circ \kappa^{-1}) (J(\kappa) \circ \kappa^{-1}) - \nht_{\kappa^{-1}}\left(\partial_\beta (\partial^{j_1} D_t r) \, \partial_\alpha(\partial^{j_2}\tilde{\lambda}^\dagger \circ \kappa^{-1})(J(\kappa) \circ \kappa^{-1}) \right) \\
& - \partial_\alpha(\nht_{\kappa^{-1}} \partial^{j_1} D_t r)\partial_\beta(\partial^{j_2}\tilde{\lambda}^\dagger \circ \kappa^{-1})(J(\kappa) \circ \kappa^{-1}) + \nht_{\kappa^{-1}}\left(\partial_\alpha (\partial^{j_1} D_t r) \, \partial_\beta(\partial^{j_2}\tilde{\lambda}^\dagger \circ \kappa^{-1})(J(\kappa) \circ \kappa^{-1})\right)
\end{align*}
with a difference controlled in $L^2$ by $C(E^\frac12 + \epsilon^2)^2$, by virtue of Propositions \ref{SingularIntegralHsEstimate}, \ref{DiffQuantitiesEstimates}, and \ref{KappaControlledByZeta}, along with the Mean Value Theorem to control $\tilde{\lambda}^\dagger \circ \kappa^{-1} - \tilde{\lambda}^\dagger$ and identity \eqref{DtCommutePartials}.  Thankfully, when we change variables by $\kappa$ in the integral \eqref{TermNo7Again}, this expression simplifies by Proposition \ref{CarefulControlRhoSigma}(b) to
\begin{align*}
& \sim \partial_\beta \nht_0 (\partial^{j_1} D_t r \circ \kappa)\partial_\alpha\partial^{j_2}\tilde{\lambda}^\dagger - \nht_0\Bigl(\partial_\beta (\partial^{j_1} D_t r \circ \kappa) \, \partial_\alpha\partial^{j_2}\tilde{\lambda}^\dagger\Bigr) \\
&  \quad - \partial_\alpha \nht_0 (\partial^{j_1} D_t r \circ \kappa)\partial_\beta\partial^{j_2}\tilde{\lambda}^\dagger + \nht_0\Bigl(\partial_\alpha (\partial^{j_1} D_t r \circ \kappa) \, \partial_\beta\partial^{j_2}\tilde{\lambda}^\dagger\Bigr)
\end{align*}
Since $\tilde{\lambda}_\beta = O(\epsilon^2)$, it suffices to retain only the first line.  Then, writing $\nht_0 = -\jvec\mathcal{R}_1 + \ivec\mathcal{R}_2$ and $\tilde{\lambda}^\dagger = \epsilon\ivec\overline{A}e^{-\jvec\phi} + O(\epsilon^2)$, we expand the expression to read
\begin{align*}
& \sum_{l = 1, 2} \Biggl( \partial_\beta \mathcal{R}_l (\ivec \kvec^{2 - l} \partial^{j_1}  D_t r \ivec \circ \kappa)\partial_\alpha\partial^{j_2}\epsilon\overline{A}e^{-\jvec\phi} - \mathcal{R}_l\Bigl(\partial_\beta (\ivec\kvec^{2 - l}\partial^{j_1} D_t r \ivec \circ \kappa) \, \partial_\alpha\partial^{j_2}\epsilon\overline{A}e^{-\jvec\phi}\Bigr)\Biggr)
\end{align*}
Unlike in our past uses of the normal form transformation, the approximate wave packet term occurs as the right factor as opposed to the left factor.  Therefore it will more convenient to use the right-hand $\jvec$-Fourier transform.  Since the Riesz transforms $\mathcal{R}_l$ have scalar-valued kernels, one can regard them as acting on the right.  Therefore taking the right $\jvec$-Fourier transform of this expression and using \eqref{RieszPotentialFormulas} gives us
\begin{equation}\label{ManipulatedTermNo7}
\sum_{l = 1, 2} \frac{1}{(2\pi)^2} \iint \Biggl( \mathcal{F}_\jvec^R[\ivec \kvec^{2 - l} \partial^{j_1}  D_t r \ivec \circ \kappa)]_{(\xi^\prime)}\jvec k \xi_2^\prime\left(\frac{\xi_l^\prime}{|\xi^\prime|} - \frac{\xi_l}{|\xi|}\right)\mathcal{F}_\jvec^R[\epsilon\partial^{j_2}\overline{A}e^{-\jvec\phi}]_{(\xi - \xi^\prime)}\Biggr) d\xi^\prime
\end{equation}
To avoid dealing with singularities near zero in frequency, we also make another simplification: We decompose this integral into the domains $|\xi^\prime| \geq 4k$ and $|\xi^\prime| \leq 4k$; since $|\xi - \xi^\prime|$ is bounded away from 0 and $\infty$, we can trade derivatives in the low frequency $|\xi^\prime| \leq 4k$ for constants.  Normal form transformations can be constructed for these low-frequency contributions just as for terms with $|j| < s$ without needing to use the commutator structure as we have done above.

Hence without loss of generality we consider only frequencies $|\xi^\prime| \geq 4k$.  On this region, we have since $|\xi - \xi^\prime| \leq \frac32 k$ that
\begin{equation*}
\left|\frac{\xi_l^\prime}{|\xi^\prime|} - \frac{\xi_l}{|\xi|} \right| = \left|\frac{\xi_l^\prime}{|\xi^\prime|}\frac{|\xi| - |\xi^\prime|}{|\xi|} + \frac{\xi_l^\prime - \xi_l}{|\xi|}\right| \leq C\frac{|\xi - \xi^\prime|}{\frac12|\xi^\prime| + \frac12 k}  \leq \frac{C}{|\xi^\prime|}
\end{equation*}
which gives the promised gain of one derivative.  With this, we adopt the ansatz
\begin{align*}
\mathcal{F}_\jvec^R[\mathscr{Q}_{\eqref{QuadraticTermNo7}}^{(j_1, j_2)}] & = \sum_{l = 1, 2} \frac{1}{(2\pi)^2} \iint_{|\xi^\prime| \geq 4k} \Biggl( \mathcal{F}_\jvec^R[\ivec \kvec^{2 - l} \partial^{j_1}  D_t r \ivec \circ \kappa)]_{(\xi^\prime)}Q_0^l(\xi^\prime - k\ivec, \xi^\prime)\mathcal{F}_\jvec^R[\epsilon\partial^{j_2}\overline{A}e^{-\jvec\phi}]_{(\xi - \xi^\prime)}\Biggr) d\xi^\prime \\
& \quad + \sum_{l = 1, 2} \frac{1}{(2\pi)^2} \iint_{|\xi^\prime| \geq 4k} \Biggl( \mathcal{F}_\jvec^R[\ivec \kvec^{2 - l} \partial^{j_1}  D^2_t r \ivec \circ \kappa)]_{(\xi^\prime)}Q_1^l(\xi^\prime - k\ivec, \xi^\prime)\mathcal{F}_\jvec^R[\epsilon\partial^{j_2}\overline{A}e^{-\jvec\phi}]_{(\xi - \xi^\prime)}\Biggr) d\xi^\prime
\end{align*}
Substituting this ansatz into \eqref{BilinearFormEquation} using \eqref{ManipulatedTermNo7} as the forcing term yields the same formal system \eqref{Q0GeneralSolution}-\eqref{Q1GeneralSolution} as in the last section, and so are given by
\begin{equation}\label{Q0ParticularSolution7}
Q_0^l = \frac{((|\xi^\prime - k\ivec| - |\xi^\prime| - |k\ivec|)(\jvec k\xi_2^\prime)}{(|\xi^\prime - k\ivec| - |\xi^\prime| - |k\ivec|)^2 - 4k|\xi^\prime|} \left(\frac{\xi_l^\prime}{|\xi^\prime|} - \frac{\xi_l}{|\xi|}\right)
\end{equation}
\begin{equation}\label{Q1ParticularSolution7}
Q_1^l = \frac{(-2\omega k\xi_2^\prime)}{(|\xi^\prime - k\ivec| - |\xi^\prime| - |k\ivec|)^2 - 4k|\xi^\prime|}\left(\frac{\xi_l^\prime}{|\xi^\prime|} - \frac{\xi_l}{|\xi|}\right)
\end{equation}
Hence, denoting $\mathcal{Q}_{\eqref{QuadraticTermNo7}}^{(j_1, j_2)} = \mathscr{Q}_{\eqref{QuadraticTermNo7}}^{(j_1, j_2)} \circ \kappa^{-1}$ as before, Lemma \ref{BilinearFormEstimate} gives us the estimate
\begin{equation}
\||\mathcal{D}|^\frac12\mathcal{Q}_{\eqref{QuadraticTermNo7}}^{(j_1, j_2)}\|_{H^\frac12} + \|D_t\mathcal{Q}_{\eqref{QuadraticTermNo7}}^{(j_1, j_2)}\|_{H^\frac12} + \|D_t^2\mathcal{Q}_{\eqref{QuadraticTermNo7}}^{(j_1, j_2)}\|_{L^2} \leq C(\epsilon E^\frac12 + \epsilon^3)
\end{equation}
Moreover, because of the above gain in regularity, the higher order error terms are controlled just as in the previous section as well; in particular by Proposition \ref{CarefulControlRhoSigma} we have the estimate $$\|(\partial_t^2 + |\mathcal{D}|)(D_t r \circ \kappa)\|_{H^{s - 1}} \leq C\|(\partial_t^2 + |\mathcal{D}|)(\sigma^\dagger \circ \kappa)\|_{H^{s - 1}} + C\epsilon E^\frac12 \leq C\epsilon E^\frac12$$ and $$\|(\partial_t^2 + |\mathcal{D}|)(D_t^2 r \circ \kappa)\|_{H^{s - 1}} \leq C\|(\partial_t^2 + |\mathcal{D}|)(D_t\sigma^\dagger \circ \kappa)\|_{H^{s - 1}} + C\epsilon E^\frac12 \leq C\epsilon E^\frac12$$
 The rest of the details are left to the reader.

\vspace{0.5cm}

\noindent \textbf{Estimates of \eqref{QuadraticTermNo4} and \eqref{QuadraticTermNo6}.}

\vspace{0.5cm}

Finally, we mention the modifications to this construction that need to be made in order to account for the term \eqref{QuadraticTermNo6}; this term is easier to estimate because it has more regularity than \eqref{QuadraticTermNo7}.  The treatment of \eqref{QuadraticTermNo4} is almost identical, and so we omit it.  First note that by Proposition \ref{CarefulControlRhoSigma} we have
$r \sim -\rho^\dagger$ and $\rho^\dagger \sim \nht \rho^\dagger$ and hence it suffices to estimate
\begin{align*}
& \sum_{j_1 + j_2 = j} \iint \frac{D_t\sigma^{[j]}}{\mathcal{A}} \cdot -\frac12(I - \nht)\Biggl(\mathcal{A}(\partial_\beta \nht (\partial^{j_1}\rho^\dagger)\partial_\alpha - \partial_\alpha \nht(\partial^{j_1}\rho^\dagger) \partial_\beta)\tilde{D}_t\partial^{j_2}(I - \nht_{\tilde{\zeta}})\tilde{\znew}\kvec \\
& \hspace{4cm} - \mathcal{A}\nht \Bigl(((\partial^{j_1}\rho^\dagger_\beta)\partial_\alpha - (\partial^{j_1}\rho^\dagger_\alpha) \partial_\beta)\tilde{D}_t\partial^{j_2}(I - \nht_{\tilde{\zeta}})\tilde{\znew}\kvec\Bigr)\Biggr) \, d\alpha \, d\beta
\end{align*}
The normal form to account for this term can now be constructed exactly as for the term $\mathscr{I}_4$.  If we denote this normal form by $\mathcal{Q}_{\eqref{QuadraticTermNo6}}$, we again arrive at the estimates
\begin{equation}
\||\mathcal{D}|^\frac12\mathcal{Q}_{\eqref{QuadraticTermNo6}}\|_{H^\frac12} + \|D_t\mathcal{Q}_{\eqref{QuadraticTermNo6}}\|_{H^\frac12} + \|D_t^2\mathcal{Q}_{\eqref{QuadraticTermNo6}}\|_{L^2} \leq C(\epsilon E^\frac12 + \epsilon^3)
\end{equation}
The higher order terms can be estimated in a less careful fashion than the higher order terms corresponding to $\mathscr{I}_4$.  We omit the details.

	\subsubsection{The Normal Forms for \eqref{QuadraticTermNo1} and \eqref{QuadraticTermNo2}}

Unlike the other quadratic terms, \eqref{QuadraticTermNo1} and \eqref{QuadraticTermNo2} arise in the energy inequality in such a way as to need significant preparation to be accounted for by a normal form transformation.  In order to do so we must introduce higher order corrections to the original energy.

As in previous sections it suffices to account for terms of the form
\begin{equation}
\iint (\theta \circ \kappa) \cdot \tilde{\lambda}_{\alpha t}\partial_\beta(\theta \circ \kappa) \, d\alpha \, d\beta = \iint (\theta \circ \kappa) \cdot -\tilde{\lambda}_{\alpha t}\mathcal{R}_2|\mathcal{D}|(\theta \circ \kappa) \, d\alpha \, d\beta
\end{equation}
where $\theta = \rho^{[j]}, \sigma^{[j]}$.

We first rewrite the term corresponding to $\theta = \rho^{[j]}$ with $0 < |j| \leq s$ so that we can use a normal form transformation to account for it.  As usual we write $\sim$ to indicate that we have omitted terms that are controlled by $C(E^\frac12 + \epsilon^2)^3$.  Since we have the estimate $\|(\partial_t^2 + |\mathcal{D}|)\rho^{[j]}\|_{L^2} \leq C\epsilon E^\frac12$, we have
\begin{align*}
& \quad \iint -(\rho^{[j]} \circ \kappa) \cdot \tilde{\lambda}_{\alpha t}\mathcal{R}_2|\mathcal{D}|(\rho^{[j]} \circ \kappa) \, d\alpha \, d\beta \\
& \sim \iint (\rho^{[j]} \circ \kappa) \cdot \tilde{\lambda}_{\alpha t}\mathcal{R}_2(\rho^{[j]} \circ \kappa)_{tt} \, d\alpha \, d\beta
\end{align*}
Suppose we add the correction
\begin{equation}\label{RhoEnergyCorrection}
\iint (\rho^{[j]} \circ \kappa) \cdot -\tilde{\lambda}_{\alpha t}\mathcal{R}_2(\rho^{[j]} \circ \kappa)_t \, d\alpha \, d\beta
\end{equation}
to our energy.  Then taking a time derivative of this correction \eqref{RhoEnergyCorrection}, adding the result to the above term, and using Proposition \ref{SkewedTripleProduct} and $\mathcal{R}_2^* = -\mathcal{R}_2$ yields the terms
\begin{align*}
& \iint (\rho^{[j]} \circ \kappa)_t \cdot -\tilde{\lambda}_{\alpha t}\mathcal{R}_2(\rho^{[j]} \circ \kappa)_t \, d\alpha \, d\beta + \iint (\rho^{[j]} \circ \kappa)_t \cdot -\mathcal{R}_2(\tilde{\lambda}_{\alpha tt}(\rho^{[j]} \circ \kappa)) \, d\alpha \, d\beta
\end{align*}
These terms are not yet removable by the method of normal forms, since the crucial cancellation in Fourier space is not present.  However, we can alter these terms in a straightforward way so that the null structure is apparent at the expense of suitably small error terms.  First, add terms identical to the above except that the sign is reversed and the substitutions $\partial_\alpha \leftrightarrow \partial_\beta$ and $\mathcal{R}_1 \leftrightarrow \mathcal{R}_2$ are made:
\begin{align*}
& \iint (\rho^{[j]} \circ \kappa)_t \cdot \tilde{\lambda}_{\beta t}\mathcal{R}_1(\rho^{[j]} \circ \kappa)_t \, d\alpha \, d\beta + \iint (\rho^{[j]} \circ \kappa)_t \cdot \mathcal{R}_1(\tilde{\lambda}_{\beta tt}(\rho^{[j]} \circ \kappa)) \, d\alpha \, d\beta
\end{align*}
Since the $\beta$ derivatives fall on $\tilde{\lambda}$ in these terms, they are controlled by $C(E^\frac12 + \epsilon^2)^3$.  Now associate these terms so that the null structure is present.  Using Lemmas \ref{BilinearFormEstimate} and \ref{DenominatorEstimate}, we see that the terms 
$$\iint (\rho^{[j]} \circ \kappa)_t \cdot \Biggl(\tilde{\lambda}_{\beta t}\mathcal{R}_1(\rho^{[j]} \circ \kappa)_t - \tilde{\lambda}_{\alpha t}\mathcal{R}_2(\rho^{[j]} \circ \kappa)_t\Biggr)  \, d\alpha \, d\beta $$
have the null structure needed to complete the normal form construction.  The remaining pair of terms must be altered further.  Since $|j| > 0$, there exists a multi-index $j_1$ of length 1 so that we can write $j = j_1 + j^\prime$.  Then using Propositions \ref{SingularIntegralHsEstimate}(b) and \ref{KappaControlledByZeta}, we can extract a derivative from the term $(\rho^{[j]} \circ \kappa)_t$ with the estimate $\|(\rho^{[j]} \circ \kappa)_t - \partial^{j_1}(\rho^{[j^\prime]} \circ \kappa)_t\|_{L^2} \leq C\epsilon E^\frac12$.  This allows us to rewrite the remaining terms as
\begin{align*}
& \quad \iint (\rho^{[j]} \circ \kappa)_t \cdot \Biggl( \mathcal{R}_1(\tilde{\lambda}_{\beta tt}(\rho^{[j]} \circ \kappa)) - \mathcal{R}_2(\tilde{\lambda}_{\alpha tt}(\rho^{[j]} \circ \kappa)) \Biggr) \, d\alpha \, d\beta \\
& \sim \iint (\rho^{[j^\prime]} \circ \kappa)_t \cdot -\Biggl( \partial^{j_1}\mathcal{R}_1(\tilde{\lambda}_{\beta tt}(\rho^{[j]} \circ \kappa)) - \partial^{j_1}\mathcal{R}_2(\tilde{\lambda}_{\alpha tt}(\rho^{[j]} \circ \kappa)) \Biggr) \, d\alpha \, d\beta \\
& = \iint (\rho^{[j^\prime]} \circ \kappa)_t \cdot -\mathcal{R}_{j_1}\Biggl( \tilde{\lambda}_{\beta tt}(\rho^{[j]} \circ \kappa)_\alpha -\tilde{\lambda}_{\alpha tt}(\rho^{[j]} \circ \kappa)_\beta \Biggr) \, d\alpha \, d\beta
\end{align*}
which is now in the proper form to be accounted for using the method of normal forms using Lemmas \ref{BilinearFormEstimate} and \ref{DenominatorEstimate}.  The terms in the sequel that will be accounted for by normal forms can all be treated in the same way.

The case $\theta = \rho$ must be handled slightly differently since we do not have control of $\rho$ in $L^2$.  We add the correction
\begin{equation}\label{RhoLowEnergyCorrection}
\iint (\rho \circ \kappa) \cdot -\tilde{\lambda}_\alpha(\rho \circ \kappa)_\beta \, d\alpha \, d\beta
\end{equation}
to the energy; taking a time derivative and combining with the third order term corresponding to $\theta = \rho$ contributes the following terms in the energy inequality:
\begin{align*}
& \quad \iint (\rho \circ \kappa)_t \cdot -(\tilde{\lambda}_\alpha\partial_\beta + \partial_\beta\tilde{\lambda}_\alpha)(\rho \circ \kappa) \, d\alpha \, d\beta \\
& \sim \iint (\rho \circ \kappa)_t \cdot -2\tilde{\lambda}_\alpha(\rho \circ \kappa)_\beta \, d\alpha \, d\beta
\end{align*}
which are again accounted for using normal forms.

We now turn to the terms corresponding to $\theta = \sigma^{[j]}$.  Observe that when $|j| < s$ we can proceed as above at the expense of adding analogous corrections to the energy.  This method fails when $|j| = s$ because we can no longer control $(\partial_t^2 + |\mathcal{D}|)\sigma^{[j]}$ with our energy; we therefore proceed more carefully.  Consider the following candidate for correction to the energy:
\begin{equation}\label{SigmaEnergyCorrection}
\frac12 \iint -\frac{1}{\mathfrak{a}}(\sigma^{[j]} \circ \kappa)_t \cdot \tilde{\lambda}_\alpha\mathcal{R}_2 (\sigma^{[j]} \circ \kappa)_t + (\sigma^{[j]} \circ \kappa) \cdot \frac12\left((N \times \nabla)\tilde{\lambda}_\alpha\mathcal{R}_2 + \tilde{\lambda}_\alpha\mathcal{R}_2(N \times \nabla)\right)(\sigma^{[j]} \circ \kappa) \, d\alpha \, d\beta
\end{equation}
If we take a derivative with respect to $t$ of this correction, we have of the first term by Proposition \ref{DiffQuantitiesEstimates} that
\begin{align}\label{EnergyCorrectionCalc1}
& \quad \frac{d}{dt} \iint -\frac{1}{\mathfrak{a}}(\sigma^{[j]} \circ \kappa)_t \cdot \tilde{\lambda}_\alpha\mathcal{R}_2 (\sigma^{[j]} \circ \kappa)_t \, d\alpha \, d\beta \notag \\
& \sim \iint -\frac{1}{\mathfrak{a}}(\sigma^{[j]} \circ \kappa)_{tt} \cdot \tilde{\lambda}_\alpha\mathcal{R}_2 (\sigma^{[j]} \circ \kappa)_t \, d\alpha \, d\beta \\
& \quad + \iint -\frac{1}{\mathfrak{a}}(\sigma^{[j]} \circ \kappa)_t \cdot \tilde{\lambda}_\alpha\mathcal{R}_2 (\sigma^{[j]} \circ \kappa)_{tt} \, d\alpha \, d\beta \notag \\
& \quad + \iint -\frac{1}{\mathfrak{a}}(\sigma^{[j]} \circ \kappa)_t \cdot \tilde{\lambda}_{\alpha t}\mathcal{R}_2 (\sigma^{[j]} \circ \kappa)_t \, d\alpha \, d\beta \notag
\end{align}
If we add the correction
\begin{equation*}
\iint \frac{1}{\mathfrak{a}}(\sigma^{[j]} \circ \kappa) \cdot \tilde{\lambda}_{\alpha t}\mathcal{R}_2(\sigma^{[j]} \circ \kappa)_t \, d\alpha \, d\beta
\end{equation*}
to the energy and combine its derivative with respect to $t$ with the last of the terms above, we have after an application of Proposition \ref{SkewedTripleProduct} that only the following terms remain:
\begin{align*}
& \qquad \iint \frac{1}{\mathfrak{a}}(\sigma^{[j]} \circ \kappa)_t \cdot \mathcal{R}_2(\tilde{\lambda}_{\alpha tt} (\sigma^{[j]} \circ \kappa)) \, d\alpha \, d\beta \\
& \quad + \iint \frac{1}{\mathfrak{a}}(\sigma^{[j]} \circ \kappa) \cdot \tilde{\lambda}_{\alpha t}\mathcal{R}_2 (\sigma^{[j]} \circ \kappa)_{tt} \, d\alpha \, d\beta
\end{align*}
The first of these terms is accountable with a normal form transformation.  The second we rewrite using the estimate $\|(\partial_t^2 - \mathfrak{a}(N \times \nabla))(\sigma^{[j]} \circ \kappa)\|_{L^2} \leq C\epsilon E^\frac12$ as
\begin{align*}
& \quad \iint \frac{1}{\mathfrak{a}}(\sigma^{[j]} \circ \kappa) \cdot \tilde{\lambda}_{\alpha t}\mathcal{R}_2 (\sigma^{[j]} \circ \kappa)_{tt} \, d\alpha \, d\beta \\
& \sim \iint \left(\frac{1}{\mathfrak{a}}(\sigma^{[j]} \circ \kappa)\right) \cdot \left(\tilde{\lambda}_{\alpha t}\mathcal{R}_2 (\mathfrak{a}(N \times \nabla)(\sigma^{[j]} \circ \kappa))\right) \, d\alpha \, d\beta \\
& \sim \iint -\left(\frac{1}{\mathfrak{a}}(\sigma^{[j]} \circ \kappa)\right) \cdot \left(\tilde{\lambda}_{\alpha t}\mathcal{R}_2 |\mathcal{D}|(\sigma^{[j]} \circ \kappa)\right) \, d\alpha \, d\beta \\
& \sim \iint -(\sigma^{[j]} \circ \kappa) \cdot \tilde{\lambda}_{\alpha t}\mathcal{R}_2 |\mathcal{D}|(\sigma^{[j]} \circ \kappa) \, d\alpha \, d\beta 
\end{align*}
In particular, note that the error terms neglected from the second to the third line above containing $$\left(\mathfrak{a}(N \times \nabla) + |D|\right)(\sigma^{[j]} \circ \kappa)$$ are controlled using Proposition \ref{HalfDerivativeEstimates} and complex interpolation after decomposing similar to \eqref{NastyHigherOrderTerms}.  This is possible because $(\sigma^{[j]} \circ \kappa)/\mathfrak{a}$ is controlled in $H^\frac12$ by the energy.

Next, taking a derivative with respect to $t$ of the second term of \eqref{SigmaEnergyCorrection} yields the expression
\begin{align}\label{EnergyCorrectionCalc2}
& \quad \frac{d}{dt} \iint (\sigma^{[j]} \circ \kappa) \cdot \frac12\left((N \times \nabla)\tilde{\lambda}_\alpha\mathcal{R}_2 + \tilde{\lambda}_\alpha\mathcal{R}_2(N \times \nabla)\right)(\sigma^{[j]} \circ \kappa) \, d\alpha \, d\beta \notag \\
& = \iint (\sigma^{[j]} \circ \kappa)_t \cdot \frac12\left((N \times \nabla)\tilde{\lambda}_\alpha\mathcal{R}_2 + \tilde{\lambda}_\alpha\mathcal{R}_2(N \times \nabla)\right)(\sigma^{[j]} \circ \kappa) \, d\alpha \, d\beta \notag \\
& \quad + \iint (\sigma^{[j]} \circ \kappa) \cdot \frac12\left((N \times \nabla)\tilde{\lambda}_\alpha\mathcal{R}_2 + \tilde{\lambda}_\alpha\mathcal{R}_2(N \times \nabla)\right)(\sigma^{[j]} \circ \kappa)_t \, d\alpha \, d\beta \\
& \quad + \iint (\sigma^{[j]} \circ \kappa) \cdot \frac12\left((N \times \nabla)\tilde{\lambda}_{\alpha t}\mathcal{R}_2 + \tilde{\lambda}_{\alpha t}\mathcal{R}_2(N \times \nabla)\right)(\sigma^{[j]} \circ \kappa) \, d\alpha \, d\beta \notag \\
& \quad + \iint (\sigma^{[j]} \circ \kappa) \cdot \frac12\left([\partial_t, (N \times \nabla)]\tilde{\lambda}_\alpha\mathcal{R}_2 + \tilde{\lambda}_\alpha\mathcal{R}_2[\partial_t, (N \times \nabla)]\right)(\sigma^{[j]} \circ \kappa) \, d\alpha \, d\beta \notag
\end{align}
The fourth term above is directly controlled by $C(E^\frac12 + \epsilon^2)^3$ thanks to Proposition \ref{HalfDerivativeEstimates}.  The third term must be rewritten further; indeed, using Proposition \ref{SkewedTripleProduct} we have
\begin{align*}
(\sigma^{[j]} \circ \kappa) \cdot \frac12\left(\tilde{\lambda}_{\alpha t}\mathcal{R}_2(N \times \nabla)\right)(\sigma^{[j]} \circ \kappa) & = (N \times \nabla)(\sigma^{[j]} \circ \kappa) \cdot \frac12\mathcal{R}_2\tilde{\lambda}_{\alpha t}(\sigma^{[j]} \circ \kappa)
\end{align*}
from which we have
\begin{align*}
& \quad \iint (\sigma^{[j]} \circ \kappa) \cdot \frac12\left(\tilde{\lambda}_{\alpha t}\mathcal{R}_2(N \times \nabla)\right)(\sigma^{[j]} \circ \kappa) \, d\alpha \, d\beta \\
& = \iint (\sigma^{[j]} \circ \kappa) \cdot \frac12(N \times \nabla)\mathcal{R}_2\tilde{\lambda}_{\alpha t}(\sigma^{[j]} \circ \kappa) \, d\alpha \, d\beta \\
& \sim \iint \frac{1}{\mathfrak{a}}(\rho^{[j]} \circ \kappa)_t \cdot \frac12(N \times \nabla)[\mathcal{R}_2, \tilde{\lambda}_{\alpha t}](\sigma^{[j]} \circ \kappa) \, d\alpha \, d\beta \\
& \quad + \iint (\sigma^{[j]} \circ \kappa) \cdot \frac12(N \times \nabla)\tilde{\lambda}_{\alpha t}\mathcal{R}_2(\sigma^{[j]} \circ \kappa) \, d\alpha \, d\beta
\end{align*}
and so because $[\mathcal{R}_2, \tilde{\lambda}_{\alpha t}]$ gains one derivative, we can eliminate the first term of the last expression with a normal form transformation.  What remains of the third term from \eqref{EnergyCorrectionCalc2} can then be further rewritten as
\begin{align*}
& \quad \iint (\sigma^{[j]} \circ \kappa) \cdot \left(\tilde{\lambda}_{\alpha t}\mathcal{R}_2(N \times \nabla)\right)(\sigma^{[j]} \circ \kappa) \, d\alpha \, d\beta \\
& \sim \iint -(\sigma^{[j]} \circ \kappa) \cdot \tilde{\lambda}_{\alpha t}\mathcal{R}_2|\mathcal{D}|(\sigma^{[j]} \circ \kappa) \, d\alpha \, d\beta
\end{align*}
Similarly, the first and second terms from \eqref{EnergyCorrectionCalc2} can be rewritten as
\begin{align*}
& = \iint (\sigma^{[j]} \circ \kappa)_t \cdot \tilde{\lambda}_\alpha\mathcal{R}_2(N \times \nabla)(\sigma^{[j]} \circ \kappa) \, d\alpha \, d\beta \\
& \quad + \iint (N \times \nabla)(\sigma^{[j]} \circ \kappa) \cdot \tilde{\lambda}_\alpha\mathcal{R}_2(\sigma^{[j]} \circ \kappa)_t \, d\alpha \, d\beta
\end{align*}
at the expense of terms containing commutators $[\mathcal{R}_2, \tilde{\lambda}_\alpha]$ which are therefore susceptible to elimination by normal form.  These terms then combine with the outstanding terms contributed from \eqref{EnergyCorrectionCalc1} to yield the quantities
\begin{equation*}
\iint \frac{1}{\mathfrak{a}}(\sigma^{[j]} \circ \kappa)_t \cdot \tilde{\lambda}_\alpha\mathcal{R}_2(\mathcal{P}\sigma^{[j]} \circ \kappa) \, d\alpha \, d\beta + \iint \frac{1}{\mathfrak{a}}(\mathcal{P}\sigma^{[j]} \circ \kappa) \cdot \tilde{\lambda}_\alpha\mathcal{R}_2(\sigma^{[j]} \circ \kappa)_t \, d\alpha \, d\beta
\end{equation*}
which is controlled by $C(E^\frac12 + \epsilon^2)^3$.  In summary, we have shown that the time derivative of \eqref{SigmaEnergyCorrection} is
\begin{equation*}
\iint -(\sigma^{[j]} \circ \kappa) \cdot \tilde{\lambda}_{\alpha t}\mathcal{R}_2|\mathcal{D}|(\sigma^{[j]} \circ \kappa) \, d\alpha \, d\beta
\end{equation*}
up to terms that are either directly controlled by $C(E^\frac12 + \epsilon^2)^3$ or which can be accounted for using a normal form transformation.  Therefore adding the negative of \eqref{SigmaEnergyCorrection} to the energy eliminates the remaining third order terms corresponding to $\theta = \sigma^{[j]}$.

\begin{remark}
The quadratic corrections $\mathcal{Q}^{[j]}$ of the normal form transformations constructed in this section no longer control the quantities $$\|D_t^2 \mathcal{Q}^{[j]}\|_{L^2}, \qquad \|\,|\mathcal{D}|^\frac12D_t\mathcal{Q}^{[j]}\|_{L^2}, \qquad \|\,|\mathcal{D}|\mathcal{Q}^{[j]}\|_{L^2}$$  This is because we allowed ourselves to construct normal forms from unknowns of the same regularity of $\sigma$ in the above section, which is a half-derivative less regular than those constructed in previous sections.  However, in the sequel we will see that we only need control over the quantities $$\|D_t\mathcal{Q}^{[j]}\|_{L^2}, \qquad \|\,|\mathcal{D}|^\frac12\mathcal{Q}^{[j]}\|_{L^2}$$
\end{remark}

\section{Justification of HNLS in Transformed and Eulerian Coordinates}

We have now constructed all of the quantities needed to begin the proof of Theorem \ref{MainTheorem}.  We proceed in three steps. The first is to use the quantities constructed in \S4 to construct the energy in earnest and show that it implies an appropriate a priori bound on the remainder.  Next, we must take the approximate solution $\tilde{\zeta}$ at time $t = 0$ and construct from it initial data for the water wave system \eqref{WaterWaveLagrangeFirstPass}-\eqref{XiTAnalyticFirstPass} that satisfies the compatibility conditions of this system and is sufficiently close to the approximate solution at $t = 0$.  Since the system \eqref{WaterWaveLagrangeFirstPass}-\eqref{XiTAnalyticFirstPass} has a local well-posedness theory in $H^s$, the next step is to establish $O(\epsilon^{-2})$ existence times for the full problem with this initial data followed by bootstrapping to the approximate solution using our a priori bound.  Finally we demonstrate Theorem \ref{EulerianTheorem} by showing that the initial data given in the hypothesis can be used to construct suitable initial data in the sense of Theorem \ref{MainTheorem}.

	\subsection{Construction of the Corrected Quantities, and the A Priori Energy Bound}

With all of our normal forms and third order energy corrections constructed, we at last have the
\begin{definition}\label{NewEnergy}
Fix some $s \geq 6$, and let $\mathcal{Q}_\rho^{[j]}$, $\mathcal{Q}_\sigma^{[j]}$ denote the sum of all of the normal forms constructed in \S\S 4.3.1-4 for $\rho^{[j]}$ and $\sigma^{[j]}$, respectively, along with the corrections to $\mathbf{S}^{[j]}$ given by \eqref{SomeCorrectionsToS} and \eqref{AnotherCorrectionToS}.  Then define $\mathbf{R}^{[j]} = \rho + \mathcal{Q}_\rho^{[j]}$ and $\mathbf{S}^{[j]} = \sigma + \mathcal{Q}_\sigma^{[j]}$.  With the $\mathscr{E}$ defined as in Proposition \ref{BasicEnergyIdentity}, set $\mathfrak{E}(t) := \sum_{|j| \leq s} \mathscr{E}(\mathbf{R}^{[j]}) + \mathscr{E}(\mathbf{S}^{[j]}) + \mathfrak{E}_3$, where $\mathfrak{E}_3$ consists of \eqref{Term7EnergyCorrection} as well as the sum of the third-order corrections derived in \S 4.3.4.
\end{definition}
Then our work in \S 4 demonstrates that using Proposition \ref{NaiveNormalFormEstimate} along with the fact that $|\mathfrak{E}_3| \leq C\epsilon(E^\frac12 + \epsilon^2)^2$, we have the conclusive estimates
\begin{corollary}\label{NewBoundsOld}
For $s \geq 6$ and $\epsilon_0 > 0$ chosen sufficiently small, for all $0 < \epsilon < \epsilon_0$ we have the estimates
\begin{align*}
\sum_{|j| \leq s} \|D_t\rho^{[j]}\|_{L^2}^2 + \|D_t\sigma^{[j]}\|_{L^2}^2 + \|\,|\mathcal{D}|\rho^{[j]}\|_{L^2}^2 & \leq C\sum_{|j| \leq s} \|D_t\mathbf{R}^{[j]}\|_{L^2}^2 + \|D_t\mathbf{S}^{[j]}\|_{L^2}^2 + C\epsilon^5
\end{align*}
\begin{align*}
\sum_{|j| \leq s} \|\,|\mathcal{D}|^\frac12\rho^{[j]}\|_{L^2}^2 + \|\,|\mathcal{D}|^\frac12\sigma^{[j]}\|_{L^2}^2 & \leq C\sum_{|j| \leq s} \|\,|\mathcal{D}|^\frac12\mathbf{R}^{[j]}\|_{L^2}^2 + \|\,|\mathcal{D}|^\frac12\mathbf{S}^{[j]}\|_{L^2}^2 + C\epsilon^5
\end{align*}
\begin{align*}
E & \leq C\mathfrak{E} + C\epsilon^5
\end{align*}
\end{corollary}
We now present the following energy inequality.

\begin{proposition}\label{APrioriEnergyBound}
Let $s \geq 6$ be given, and let $T_0 > 0$ be the existence time given in the a priori assumption on $\zeta$.  Suppose that $\mathfrak{E}(0) \leq M_0\epsilon^4$.  Then for every $\iota > 0$ we have the a priori bound $$\sup_{0 \leq t \leq \min(\mathscr{T}\epsilon^{-2}, T_0)} \mathfrak{E}(t)^\frac12 \leq C\epsilon^{2 - \iota}$$ where $C$ depends on $k$, $s$, $M_0$, $\|A_0\|_{H^{s + 13} \cap H^3(0^+)}$ and $\iota$.
\end{proposition}

\begin{proof}
We would like to demonstrate first that the inequality $$\frac{d\mathfrak{E}}{dt} \leq C(E^\frac12 + \epsilon^2)^3$$ holds.  Begin by expanding $\frac{d\mathfrak{E}}{dt}$ through Definition \ref{NewEnergy} and Proposition \ref{BasicEnergyIdentity} as follows:
\begin{align}\label{PrelimEnergyIdentity}
\frac{d\mathfrak{E}}{dt} & = \sum_{|j| \leq s} \iint \frac{2}{\mathcal{A}}D_t \mathbf{R}^{[j]} \cdot \left(D_t^2 - \mathcal{A}(\zeta_\beta\partial_\alpha - \zeta_\alpha\partial_\beta)\right)\mathbf{R}^{[j]} \, d\alpha\, d\beta \notag \\
& \qquad + \iint - \frac{1}{\mathcal{A}} U_{\kappa^{-1}}\left(\frac{\mathfrak{a}_t}{\mathfrak{a}}\right)|D_t\mathbf{R}^{[j]}|^2 - \frac{1}{\mathcal{A}} U_{\kappa^{-1}}\left(\frac{\mathfrak{a}_t}{\mathfrak{a}}\right)|D_t\mathbf{S}^{[j]}|^2  \, d\alpha \, d\beta \notag \\
& \qquad + \iint \frac{2}{\mathcal{A}}D_t \mathbf{S}^{[j]} \cdot \left(D_t^2 - \mathcal{A}(\zeta_\beta\partial_\alpha - \zeta_\alpha \partial_\beta)\right)\mathbf{S}^{[j]} \, d\alpha\, d\beta \notag \\
& \qquad + \frac{d\mathfrak{E}_3}{dt} + \iint  - \mathbf{R}^{[j]} \cdot \left(D_t\zeta_\beta\mathbf{R}^{[j]}_\alpha - D_t\zeta_\alpha\mathbf{R}^{[j]}_\beta\right) - \mathbf{S}^{[j]} \cdot \left((D_t\zeta_\beta)\mathbf{S}^{[j]}_\alpha - (D_t\zeta_\alpha)\mathbf{S}^{[j]}_\beta\right) \, d\alpha \, d\beta \; \notag
\end{align}
By Proposition \ref{DiffQuantitiesEstimates} and Corollary \ref{NewBoundsOld}, the second line is bounded by $C(E^\frac12 + \epsilon^2)^3$.  By our work in \S 4.3.4 we know that the fourth line is of size at most $C(E^\frac12 + \epsilon^2)^3$.  The remaining terms are analyzed by decomposing them as in \S4.2, except that in some of these terms occurrences of $D_t \rho^{[j]}$ and $D_t \sigma^{[j]}$ in \S 4.2 are here replaced by $D_t \mathbf{R}^{[j]}$ and $D_t \mathbf{S}^{[j]}$ where appropriate.  This replacement does not change the estimates thanks to Corollary \ref{NewBoundsOld}.  Generally we decompose all of the factors of the nonlinearities into sums of remainders and approximations.  Most of these terms are straightforwardly estimated using the estimates of \S 4.1.  We list only the terms requiring further treatment:

\begin{enumerate}
\item[(1)]{Ostensibly quadratic terms involving operators of the form $[T, \nht]\theta$ with $\theta = -\nht \theta$.}
\item[(2)]{Terms involving the residual of the equation $\mathcal{P}\theta = G_\theta$.}
\end{enumerate}

\vspace{0.5cm}

\textbf{(1)} There are several quadratic terms that can be quickly written as cubic using almost-orthogonality methods.  For example, for $\theta = \rho, \sigma$ and $\Theta = \mathbf{R}, \mathbf{S}$ respectively,
$$\iint -\frac{1}{\mathcal{A}} D_t\Theta^{[j]} \cdot [\mathcal{P}, \nht]\partial^j\theta \, d\alpha \, d\beta$$
As it stands, the commutator $[\mathcal{P}, \nht]\partial^j\theta$ is only quadratic.  To treat this term, we use the identity $(I - \nht)[T, \nht] = [T, \nht](I + \nht)$ to exploit almost-orthogonality.  Write 
\begin{align*}
\frac{1}{\mathcal{A}}D_t\Theta^{[j]} & = \frac{1}{\mathcal{A}} D_t\mathcal{Q}_\theta + \frac12(I - \nht)\frac{1}{\mathcal{A}}D_t\theta^{[j]} + \frac12 \frac{1}{\mathcal{A}}[D_t, \nht]\theta^{[j]} + \frac12 \left[\frac{1}{\mathcal{A}}, \nht \right] D_t\theta^{[j]}
\end{align*}
Then, letting $\sim$ denote that we have omitted terms of size $C(E^\frac12 + \epsilon^2)^3$, this yields
\begin{align*}
\iint -\frac{1}{\mathcal{A}} D_t\Theta^{[j]} \cdot [\mathcal{P}, \nht]\partial^j\theta \, d\alpha \, d\beta & \sim \iint -\frac12(I - \nht)\frac{1}{\mathcal{A}} D_t\theta^{[j]} \cdot [\mathcal{P}, \nht]\partial^j\theta \, d\alpha \, d\beta \\
& = \iint -\frac{1}{\mathcal{A}} D_t\theta^{[j]} \cdot \frac12(\nht- \nht^*)[\mathcal{P}, \nht]\partial^j\theta \, d\alpha \, d\beta \\
& \quad + \iint -\frac{1}{\mathcal{A}} D_t\theta^{[j]} \cdot [\mathcal{P}, \nht](I + \nht)\partial^j\theta \, d\alpha \, d\beta \\
& = \iint -\frac{1}{\mathcal{A}} D_t\theta^{[j]} \cdot \frac12(\nht- \nht^*)[\mathcal{P}, \nht]\partial^j\theta \, d\alpha \, d\beta \\
& \quad + \iint -\frac{1}{\mathcal{A}} D_t\theta^{[j]} \cdot [\mathcal{P}, \nht][\nht, \partial^j]\theta \, d\alpha \, d\beta
\end{align*}
All of these integrals are controlled by $C(E^\frac12 + \epsilon^2)^3$ by Propositions \ref{SingularIntegralHsEstimate}, \ref{DiffQuantitiesEstimates}, and \ref{CarefulControlRhoSigma}.  The other terms arising from the energy containing the commutator $[\mathcal{P}, \nht]$ are treated in the same way.

Similarly, the term $$\iint \frac{D_t\mathbf{R}^{[j]}}{\mathcal{A}} \cdot \frac12(I - \nht)\Bigl(G - \tilde{\mathcal{P}}(I - \nht_{\tilde{\zeta}})\tilde{\znew}\kvec\Bigr) \, d\alpha \, d\beta$$ contributes the following difference of commutators: $$\iint \frac{D_t\mathbf{R}^{[j]}}{\mathcal{A}} \cdot \frac12(I - \nht)\Bigl([D_t, \nht]D_t\zeta^\dagger - [\tilde{D}_t, \tilde{\nht}]\tilde{D}_t\tilde{\zeta}^\dagger\Bigr) \, d\alpha \, d\beta$$  As above, we can decompose this difference as a sum of differences involving approximate and remainder quantities and use an almost orthogonality argument to gain an extra factor of smallness.  This implies the above is bounded by $C(E^\frac12 + \epsilon^2)^3$.  The other terms arising from this quadratic commutator are treated similarly.

\vspace{0.5cm}

\textbf{(2)}  As mentioned in \S 3, the residual is only of appropriate size provided it is estimated in the context of the energy estimate.  Recall from Lemma \ref{AbstractEpsilon4} that the residual yields an integral of the form
\begin{align*}
\iint \frac{1}{\mathcal{A}} D_t \mathbf{R}^{[j]} \cdot \frac12 (I - \nht)\partial^j(I - \nht)\epsilon^4\left( (I - \nht_0)F_0 + \frac12 (I + \nht_0)F + \sum_{n = -3}^{-1} S_n e^{n\jvec\phi}\ivec + \sum_{0 < |m| \leq 3} S^\prime_m e^{m\jvec\phi}\right) \, d\alpha \, d\beta
\end{align*}
As in (1) we know that $\mathcal{A}^{-1}D_t\mathbf{R}^{[j]}$ is almost-holomorphic, and so it suffices to treat
\begin{align*}
\iint \frac{1}{\mathcal{A}} D_t \mathbf{R}^{[j]} \cdot \epsilon^4 \partial^j\left( (I - \nht_0)F_0 + \frac12 (I + \nht_0)F + \sum_{n = -3}^{-1} S_n e^{n\jvec\phi}\ivec + \sum_{0 < |m| \leq 3} S^\prime_m e^{m\jvec\phi}\right) \, d\alpha \, d\beta
\end{align*}
First, by Lemma \ref{AbstractEpsilon4} we know that $F_0$ is scalar-valued.  Therefore we can write
\begin{align*}
& \quad \iint \frac{1}{\mathcal{A}} D_t \mathbf{R}^{[j]} \cdot \epsilon^4 \partial^j\left((I - \nht_0)F_0\right) \, d\alpha \, d\beta \\
& = \iint (I - \nht_0)\frac{1}{\mathcal{A}} D_t \mathbf{R}^{[j]} \cdot \epsilon^4\partial^j F_0 \, d\alpha \, d\beta \\
& \sim \iint (\nht - \nht_0)\frac{1}{\mathcal{A}} D_t \mathbf{R}^{[j]} \cdot \epsilon^4\partial^j F_0 \, d\alpha \, d\beta + \iint \frac{1}{\mathcal{A}} D_t \Re(\mathbf{R}^{[j]}) \cdot \epsilon^4\partial^j F_0 \, d\alpha \, d\beta 
\end{align*}
and all of these integrals are of size at most $C(\mathfrak{E}^\frac12 + \epsilon^2)^3$.  A similar almost-orthogonality argument treats the term
\begin{align*}
& \quad \iint \frac{1}{\mathcal{A}} D_t \mathbf{R}^{[j]} \cdot \epsilon^4 \partial^j\frac12 (I + \nht_0)F \, d\alpha \, d\beta \\
& \sim \iint \frac{1}{\mathcal{A}} D_t \mathbf{R}^{[j]} \cdot \epsilon^4 \partial^j\frac12 (\nht_0 - \nht)^*F \, d\alpha \, d\beta
\end{align*}
which is now seen to be of size at most $C(\mathfrak{E}^\frac12 + \epsilon^2)^3$.  Since the terms $\sum_{n = -3}^{-1} S_n e^{n\jvec\phi}\ivec + \sum_{m = -3}^{-1} S^\prime_m e^{m\jvec\phi}$ are all almost-antiholomorphic, they are disposed of similarly.  Finally, taking only the term $m = 1$ above for simplicity, we know by Lemma \ref{FlatHilbertTransformExpansion} that $\overline{S}^\prime_1e^{-\jvec\phi} = \frac12(I + \nht_0)\overline{S}^\prime_1e^{-\jvec\phi}$ up to a term of size at most $O(\epsilon^2)$ in $L^2$.  Therefore we can use almost orthogonality to write
\begin{align*}
& \quad \iint \frac{1}{\mathcal{A}} D_t \mathbf{R}^{[j]} \cdot \epsilon^4\partial^j\left(S^\prime_1 e^{\jvec\phi}\right) \, d\alpha \, d\beta \\
& = \iint \frac{1}{\mathcal{A}} D_t \mathbf{R}^{[j]} \cdot \epsilon^4\partial^j\left(S^\prime_1 e^{\jvec\phi} + \overline{S}^\prime_1 e^{-\jvec\phi}\right) \, d\alpha \, d\beta - \iint \frac{1}{\mathcal{A}} D_t \mathbf{R}^{[j]} \cdot \epsilon^4\partial^j\left(\overline{S}^\prime_1 e^{-\jvec\phi}\right) \, d\alpha \, d\beta \\
& \sim \iint \frac{1}{\mathcal{A}} D_t \Re(\mathbf{R}^{[j]}) \cdot \epsilon^4\partial^j\left(S^\prime_1 e^{\jvec\phi} + \overline{S}^\prime_1 e^{-\jvec\phi}\right) \, d\alpha \, d\beta \\
& \quad - \iint \frac{1}{\mathcal{A}} D_t \mathbf{R}^{[j]} \cdot \frac12(\nht_0 - \nht)^*\epsilon^4\partial^j\left(\overline{S}^\prime_1 e^{-\jvec\phi}\right) \, d\alpha \, d\beta
\end{align*}
all of which are bounded by $C(\mathfrak{E}^\frac12 + \epsilon^2)^3$.  The estimates for the residual of the time derivative is essentially the same, so we do not estimate it explicitly.  Summing all of these estimates gives us (a).  

\vspace{0.5cm}

Applying Corollary \ref{NewBoundsOld} gives us the inequality
\begin{equation*}
\frac{d\mathfrak{E}}{dt} \leq C(\mathfrak{E}^\frac12 + \epsilon^2)^3
\end{equation*}
We begin a continuity argument.  By hypothesis we know that $\mathfrak{E}(0) = M_0\epsilon^4$.  Let $T^*$ be the first time at which $\mathfrak{E}(T^*) = 4M_0\epsilon^4$.  If $T^* \geq \mathscr{T}\epsilon^{-2}$, then choose $\mathscr{T}^\prime = \mathscr{T}$.  If not, then on the interval $[0, T^*]$ we have $\frac{d\mathfrak{E}}{dt} \leq C_0\epsilon^6$ from which we have immediately that $\mathfrak{E}(t) \leq \mathfrak{E}(0) + C_0t\epsilon^6$.  But then if we choose $\mathscr{T}^\prime$ so that $C_0\mathscr{T}^\prime \leq 2M_0$ and assume further that $T^* \leq \mathscr{T}^\prime\epsilon^{-2}$, we find that at $t = T^*$, $$4M_0\epsilon^4 = \mathfrak{E}(T^*) \leq \mathfrak{E}(0) + C_0\mathscr{T}^\prime\epsilon^4 \leq 3M_0\epsilon^4$$ which is a contradiction.

Finally, following the idea given in \cite{RauchLannesLogTime} and \cite{SalemRenormGroup}, we remove the restriction on the time $\mathscr{T}^\prime \leq \mathscr{T}$ in the above estimate by extending the validity of the a priori estimate by a logarithmic factor of $\epsilon$ at the expense of enlarging the error bound slightly.  We have just established that there exists a time $\mathscr{T}^\prime \leq \mathscr{T}$ so that $$\mathfrak{E}(t)^\frac12 \leq C_1\epsilon^2, \qquad 0 \leq t \leq \mathscr{T}^\prime\epsilon^{-2}$$  Introduce an $n \in \mathbb{N}$ to be fixed later.  Applying the above estimate $n$ times yields $$\mathfrak{E}(t) \leq C_1^n \epsilon^2, \qquad 0 \leq t \leq n\mathscr{T}^\prime\epsilon^{-2}$$  Suppose that $\iota > 0$ is chosen so that $C_1^n \leq \epsilon^{-\iota}$.  Then $n \leq \iota |\log(\epsilon)|/\log(C_1)$.  For $\epsilon_0$ chosen small enough the set of all positive $n$ satisfying this inequality is nonempty; fix $n$ to be the largest such satisfying this inequality.  Then we have $$\mathfrak{E}(t) \leq \epsilon^{2 - \iota}, \qquad 0 \leq t \leq \lfloor \iota |\log(\epsilon)|/\log(C_1) \rfloor \mathscr{T}^\prime\epsilon^{-2}$$  But now for any $\iota > 0$ and $\mathscr{T} > 0$ we please, we may choose $\epsilon_0 > 0$ so small depending on $\iota, \mathscr{T}^\prime, \|A_0\|_{H^{s + 13} \cap H^3(\delta)}$ so that $\lfloor \iota |\log(\epsilon)|/\log(C_1) \rfloor \mathscr{T}^\prime \geq \mathscr{T}$, as desired.
\end{proof}

	\subsection{Construction of Appropriate Initial Data}

Now that we have a suitable a priori estimate we can show that there is a solution to the water wave problem \eqref{WaterWaveLagrangeFirstPass}-\eqref{XiTAnalyticFirstPass} which remains closer than $O(\epsilon^2)$ in Sobolev space to the approximate solution for $O(\epsilon^{-2})$ times.  To start, we need the following local well-posedness result with blow-up alternative of \cite{WuLocal3D}:

\begin{theorem}\label{LocalWellPosednessWW}
(c.f. also Theorem 4.3 of \cite{WuGlobal3D})  Let $s \geq 5$ be given, and suppose that we are given initial data 
\begin{align*}
\Xi(0) - P & = \zeta_0 - P \in \dot{H}^\frac12 \cap \dot{H}^{s + 1} \\
\Xi_t(0) & = u_0 \in H^{s + \frac12} \\
\Xi_{tt}(0) &= w_0 \in H^s \\
\mathfrak{a}(0) & = \mathfrak{a}_0 \in H^s
\end{align*}
for the system \eqref{WaterWaveLagrangeFirstPass}-\eqref{XiTAnalyticFirstPass} which satisfies the compatibility conditions
$$\Re(\zeta_0) = 0 \qquad (I - \nht_{\zeta_0})u_0 = 0 \qquad w_0 + \kvec = \mathfrak{a}_0(\partial_\alpha \zeta_0 \times \partial_\beta \zeta_0)$$ with $\mathfrak{a}_0$ given by the formula in Proposition \ref{BADtAFormulas}(b), and suppose further that there are numbers $\nu_0, N_0 > 0$ so that $\zeta_0$ satisfies the chord-arc condition $$\nu_0 \leq \sup_{(\alpha, \beta) \neq (\alpha^\prime, \beta^\prime)} \frac{|\zeta_0(\alpha, \beta) - \zeta_0(\alpha^\prime, \beta^\prime)|}{|(\alpha, \beta) - (\alpha^\prime, \beta^\prime)|} \leq N_0$$ as well as that $|\partial_\alpha\zeta_0 \times \partial_\beta \zeta_0|^{-1} \leq N_0$.  Then there exists a time $T_0 > 0$ and constants $\nu, N$ depending on $\|\,|\mathcal{D}|^\frac12(\zeta_0 - P)\|_{H^{s + \frac12}}$, $\|u_0\|_{H^{s + \frac12}}$, and $\|w_0\|_{H^s}$ for which there is a solution $\Xi$ satisfying \eqref{WaterWaveLagrangeFirstPass}-\eqref{XiTAnalyticFirstPass} with the following properties for all $0 \leq t \leq T_0$:
\begin{itemize}
\item{$\partial_t^j \Xi \in C^{2 - j}([0, T_0], H^{s + 1 - j/2})$}
\item{$$\nu \leq \sup_{(\alpha, \beta) \neq (\alpha^\prime, \beta^\prime)} \frac{|\Xi(\alpha, \beta) - \Xi(\alpha^\prime, \beta^\prime)|}{|(\alpha, \beta) - (\alpha^\prime, \beta^\prime)|} \leq N$$}
\item{$|\Xi_\alpha \times \Xi_\beta|^{-1} \leq N$}
\end{itemize}
Moreover, if $T_0^*$ is the supremum of all such times $T_0$ on which the above solution $\Xi$ with the above properties exists, then either $T_0^* = \infty$ or $$\|\Xi_t\|_{W^{\lfloor s/2 \rfloor + 3, \infty}} + \|\Xi_{tt}\|_{W^{\lfloor s/2 \rfloor + 3, \infty}} + \sup_{(\alpha, \beta) \neq (\alpha^\prime, \beta^\prime)} \frac{|\Xi(\alpha, \beta) - \Xi(\alpha^\prime, \beta^\prime)|}{|(\alpha, \beta) - (\alpha^\prime, \beta^\prime)|} + \frac{1}{|\Xi_\alpha \times \Xi_\beta|} \not\in L^\infty[0, T_0^*)$$
\end{theorem}

Ideally we could take the approximate solution $\tilde{\zeta}$ of \S 3 and use $(\tilde{\zeta}, \tilde{\zeta}_t, \tilde{\zeta}_{tt}, \tilde{\mathcal{A}})$ as initial data. However, this candidate for the initial data need not satisfy the compatibility conditions in Theorem \ref{LocalWellPosednessWW}.  To rectify this, the a priori bound of the last section guarantees that if we can find data sufficiently close to the approximate solution, the resulting solution will remain close for the appropriate time scales.  Hence it suffices to construct initial data so that the energy of the remainder is sufficiently small initially.  This is done in the following

\begin{proposition}\label{InitialDataExists}
Let $\epsilon_0 > 0$ be chosen sufficiently small.  Then there exists initial data $(\zeta_0, v_0, w_0, \mathfrak{a}_0)$ satisfying the compatibility conditions for the water wave problem.  Moreover, with this initial data the initial energy $\mathfrak{E}(0)$ constructed through $\rho(0)$ and $\sigma(0)$ has the property that $\mathfrak{E}(0)^\frac12 \leq C\epsilon^3$.
\end{proposition}

\begin{proof}
The reader is reminded that we have chosen the initial parametrization so that $\kappa(\alpha, \beta, 0) = P$; therefore all of the formulas derived in the new variables continue to hold in the original Lagrangian variables when $t = 0$.  

In order to construct the initial parametrization, we use Banach's Fixed Point Theorem.  Define the functional $$\mathbf{F}(f) = (I + \nht_{f + P} - \mathcal{K}_{f + P})\tilde{\znew}(0)\kvec$$  By the coarse estimate $\|\mathbf{F}(f)\|_{H^{s + 1}} \leq C(1 + \|f\|_{H^{s + 1}})$ we see that $\mathbf{F} : H^{s + 1} \to H^{s + 1}$.  Moreover, by Proposition \ref{DifferenceHilbertTransformHs}, we have for $f, g \in H^s$ the estimate
\begin{align*}
\|\mathbf{F}(f) - \mathbf{F}(g)\|_{H^{s + 1}} & \leq \|(\nht_{f + P} - \nht_{g + P})\tilde{\znew}(0)\kvec + (\mathcal{K}_{f + P} - \mathcal{K}_{g + P})\tilde{\znew}(0)\kvec\|_{H^{s + 1}} \\
& \leq C\|f - g\|_{H^{s + 1}}\|\tilde{\znew}(0)\|_{W^{\infty, s + 1}} \\
& \leq C\epsilon\|f - g\|_{H^{s + 1}}
\end{align*} and so for $\epsilon_0 > 0$ chosen sufficiently small $\mathbf{F}$ is a contraction mapping.  Hence there exists a unique $\lambda_0 \in H^{s + 1}$ such that $$\lambda_0 = (I + \nht_{\zeta_0} - \mathcal{K}_{\zeta_0})\tilde{\znew}(0)\kvec,$$ where we have denoted $\zeta_0 = \lambda_0 + P$.  Taking the $\kvec$-component of this equation implies $\znew_0 = \tilde{\znew}(0)$, and so Proposition \ref{ChangeOfVariablesIJValued} implies that the scalar part of $\lambda_0$ is zero.  Denote by $\tilde{\nht}|_0$ and $\tilde{\mathcal{K}}|_0$ the operators $\tilde{\nht}$ and $\tilde{\mathcal{K}}$ evaluated at $t = 0$, respectively.  Using Proposition \ref{SingularIntegralHsEstimate} we have the estimate
\begin{align}\label{InitialLambdaFormula}
\lambda_0 - \tilde{\lambda} & =  ((\nht_{\zeta_0} - \mathcal{K}_{\zeta_0}) - (\nht_{\tilde{\zeta}(0)} - \mathcal{K}_{\tilde{\zeta}(0)}))\tilde{\znew}\kvec \notag \\
& \quad + ((\nht_{\tilde{\zeta}(0)} - \mathcal{K}_{\tilde{\zeta}(0)}) - (\tilde{\nht}|_0 - \tilde{\mathcal{K}}|_0))\tilde{\znew}\kvec \\
& \quad + (I + \tilde{\nht}|_0 - \tilde{\mathcal{K}}|_0)\tilde{\znew}(0)\kvec - \tilde{\lambda}(0) \notag
\end{align} Now we can choose $\epsilon_0 < 0$ sufficiently small so that 
\begin{align}\label{HalfDerivLambdaEstimate}
\| \, \lambda_0 - \tilde{\lambda}(0) \|_{H^{s + 1}} & \leq C\epsilon^4 + C\epsilon^3 \leq C\epsilon^3
\end{align}
Using \eqref{InitialLambdaFormula} we similarly have the estimate
\begin{align}\label{FullDerivLambdaEstimate}
\|\,|\mathcal{D}|(\lambda_0 - \tilde{\lambda}(0))\|_{H^s} \leq C\epsilon^3
\end{align}
Let $\mathcal{N}_0 = \partial_\alpha \zeta_0 \times \partial_\beta \zeta_0$ and $\mathbf{n}_0 = \mathcal{N}_0/|\mathcal{N}_0|$ be the outward unit normal to $\Sigma(t)$; choose $$v_0 = (I + \nht_{\zeta_0})\left(\mathbf{n}_0(I + \{\mathbf{n}_0 - \kvec\}_3 + \{\nht_{\zeta_0}\mathbf{n}_0\}_3)^{-1}\tilde{\znew}_t(0)\right)$$  We clearly have $(I - \nht_{\zeta_0})v_0 = 0$, we have $\{v_0\}_3 = \tilde{\znew}_t(0)$ by construction, and $\Re(v_0) = 0$ by the definition of the Hilbert transform.  (The operator $(I + \{\mathbf{n}_0 - \kvec\}_3 + \{\nht_{\zeta_0}\mathbf{n}_0\}_3)^{-1}$ is a small perturbation from the identity and can be constructed as usual using a Neumann series provided $\epsilon_0 > 0$ is chosen to be sufficiently small.  We leave the routine details to the reader.)  Finally, to satisfy the compatibility conditions, we set
$$w_0 = \mathfrak{a}_0(\partial_\alpha\zeta_0 \times \partial_\beta \zeta_0) - \kvec$$
which clearly satisfies $\Re(w_0) = 0$, and where we are again forced by compatibility to define $\mathfrak{a}_0$ so that it satisfies the relation
\begin{align}\label{AInitialFormula}
\mathfrak{a}_0 & = (I - \mathcal{K}_{\zeta_0})^{-1}\Bigl\{\kvec + [\partial_t, \nht_{\zeta_0}]v_0 + [\mathfrak{a}_0(\partial_\beta\zeta_0\partial_\alpha - \partial_\alpha \zeta_0 \partial_\beta), \nht_{\zeta_0}] (I + \nht_{\zeta_0})\znew_0\kvec  \\
& \qquad + (I - \nht)\left(-\mathfrak{a}_0\partial_\beta \zeta_0 \times (\partial_\alpha\mathcal{K}_{\zeta_0} \znew_0 \kvec) + \mathfrak{a}_0\partial_\alpha \zeta_0 \times (\partial_\beta\mathcal{K}_{\zeta_0}\znew_0\kvec) + \mathfrak{a}_0(\partial_\alpha \lambda_0 \times \partial_\beta \lambda_0)\right)\Bigr\}_3 \notag
\end{align}
That such an $\mathfrak{a}_0$ exists follows by a fixed point argument in $H^s$ applied to the map
\begin{align*}
\mathbf{G}(f) & = (I - \mathcal{K}_{\zeta_0})^{-1}\Bigl\{[\partial_t, \nht_{\zeta_0}]v_0 + [(1 + f)(\partial_\beta\zeta_0\partial_\alpha - \partial_\alpha \zeta_0 \partial_\beta), \nht_{\zeta_0}] (I + \nht_{\zeta_0})\znew_0\kvec  \\
& \qquad\qquad + (I - \nht)\Bigl(-(1 + f)\partial_\beta \zeta_0 \times (\partial_\alpha\mathcal{K}_{\zeta_0} \znew_0 \kvec) + (1 + f)\partial_\alpha \zeta_0 \times (\partial_\beta\mathcal{K}_{\zeta_0}\znew_0\kvec) \\
& \qquad\qquad + (1 + f)(\partial_\alpha \lambda_0 \times \partial_\beta \lambda_0)\Bigr)\Bigr\}_3
\end{align*}
We now give another relation between the quantities $v_0$ and $\tilde{\lambda}_t(0)$.  Since $\|\nabla \lambda\|_{L^\infty} \leq C\epsilon$ and the data is constructed to lie in $H^{s + 1} \times H^{s + \frac12} \times H^s$, we can apply Theorem \ref{LocalWellPosednessWW} to construct a solution $\Xi$ having all of the properties listed in the theorem on some time interval $[0, T_0]$.  In particular we have $\zeta_t(0) = v_0$ and $\zeta_{tt}(0) = w_0$.  Taking a derivative of \eqref{LambdaIdentity} with respect to $t$ and using gives
\begin{align*}\label{DiffLambdaIdentity}
\lambda_t - \tilde{\lambda}_t & = [\partial_t, \nht - \mathcal{K}](\znew - \tilde{\znew})\kvec + (I + \nht - \mathcal{K})(\znew_t - \tilde{\znew}_t)\kvec \\
& \quad + \left[\partial_t, \left((\nht - \mathcal{K}) - (\tilde{\nht} - \tilde{\mathcal{K}})\right)\right]\tilde{\znew}\kvec + \left((\nht - \mathcal{K}) - (\tilde{\nht} - \tilde{\mathcal{K}})\right)\tilde{\znew}_t\kvec \\
& \quad + \partial_t\left((I + \tilde{\nht} - \tilde{\mathcal{K}})\tilde{\znew}\kvec - \tilde{\lambda}\right)
\end{align*}
By our construction of $v_0$, the term $(I + \nht - \mathcal{K})(\znew_t - \tilde{\znew}_t)\kvec$ vanishes for $t = 0$.  The other terms evaluated at $t = 0$ give us the estimate
\begin{align*}
\|v_0 - \tilde{\lambda}_t(0)\|_{H^{s + \frac12}} & \leq C\epsilon\|v_0 - \tilde{\lambda}_t(0)\|_{H^{s + \frac12}} + C\epsilon^4 + C\epsilon^3 \\
& \leq C\epsilon\|v_0 - \tilde{\lambda}_t(0)\|_{H^{s + \frac12}} + C\epsilon^3
\end{align*}
and by choosing $\epsilon_0 > 0$ we obtain the estimate
\begin{equation}\label{DtDiffLambdaEstimate}
\|v_0 - \tilde{\lambda}_t(0)\|_{H^{s + \frac12}} \leq C\epsilon^3
\end{equation}

We first use this to give a bound on the remainder of $\mathfrak{a}_0$.  Note that no second derivatives in time appear in the formula \eqref{AInitialFormula}.  Therefore, since $\tilde{\mathcal{A}}(0)$ satisfies the approximate version \eqref{AFormulaApprox} of \eqref{AInitialFormula}, we can decompose $\mathfrak{a}_0 - \tilde{\mathcal{A}}(0)$ as in \eqref{DiffAFormula} and use \eqref{FullDerivLambdaEstimate} and \eqref{DtDiffLambdaEstimate} to arrive at the estimate 
\begin{equation}\label{DiffAInitialEstimate}
\|\mathfrak{a}_0 - \tilde{\mathcal{A}}(0)\|_{H^s} \leq C\epsilon^3
\end{equation}

Next, we use these estimates to show that $\mathfrak{E}(0) \leq C\epsilon^3$. The decompositions here are similar to those in the proof of Proposition \ref{CarefulControlRhoSigma}.  First consider $$\rho(0) = \frac12(I - \nht_{\zeta_0})\left((I - \nht_{\zeta_0})\znew_0\kvec - (I - \nht_{\tilde{\zeta}(0)})\tilde{\znew}(0)\kvec\right)$$  Decomposing $\zeta_0 = \tilde{\zeta}(0) + (\zeta_0 - \tilde{\zeta}(0))$ and using \eqref{HalfDerivLambdaEstimate} and \eqref{FullDerivLambdaEstimate} yields the estimates
$$\|\,(|\mathcal{D}|^\frac12\rho)(0)\|_{H^{s + \frac12}} \leq C\epsilon^3 \qquad \text{and}\qquad \|\,(|\mathcal{D}|\rho)(0)\|_{H^s} \leq C\epsilon^3$$

By Proposition \ref{CarefulControlRhoSigma} and Corollary \ref{NewBoundsOld}, we know that $\|\sigma(0) - (D_t\rho)(0)\|_{H^{s + \frac12}} \leq C\epsilon(\mathfrak{E}(0)^\frac12 + \epsilon^2)$ and $\|(D_t \sigma)(0) - (D_t^2\rho)(0)\|_{H^s} \leq C\epsilon(\mathfrak{E}(0)^\frac12 + \epsilon^2)$.  Hence it suffices to show $\|\sigma(0)\|_{H^{s + \frac12}} \leq C\epsilon^3$ and $\|(D_t^2\rho)(0)\|_{H^{s + \frac12}} \leq C\epsilon^3$.  Write $\sigma$ as
$$\sigma = \frac12(I - \nht)\left([D_t, \nht]\znew\kvec - [\tilde{D}_t, \nht_{\tilde{\zeta}}]\tilde{\znew}\kvec\right) + \frac12(I - \nht)\left((I - \nht)(D_t\znew\kvec - \tilde{D}_t\tilde{\znew}\kvec) + (\nht - \nht_{\tilde{\zeta}})\tilde{D}_t\tilde{\znew}\kvec\right)$$ and decompose into approximate plus remainder quantities.  Then evaluating at $t = 0$ and using \eqref{FullDerivLambdaEstimate} and \eqref{DtDiffLambdaEstimate} yields the estimate $\|\sigma(0)\|_{H^{s + \frac12}} \leq C\epsilon^3$.  Since we do not have adequate control over $\lambda_{tt}(0) - \tilde{\lambda}_{tt}(0)$, we cannot estimate $D_t^2\rho(0)$ as we did for $\sigma(0)$.  Instead, we use \eqref{RhoEquation} to write $$D_t^2\rho = \mathcal{A}(\zeta_\beta\rho_\alpha - \zeta_\alpha \rho_\beta) + G_\rho$$  We note that no second derivatives in time appear in either the null-form term or the cubic term $G_\rho$; similarly the formulas for $\mathfrak{a}_0$ also do not depend on second time derivatives of $\Xi$.  We can then obtain the estimate by decomposing into approximate plus remainder quantities, evaluating at $t = 0$, and applying \eqref{FullDerivLambdaEstimate}, \eqref{DtDiffLambdaEstimate} and \eqref{DiffAInitialEstimate}; we omit the details.
\end{proof}

	\subsection{Long-Time Existence of Wave Packet-Like Solutions}

It is now only a matter of combining Theorem \ref{LocalWellPosednessWW}, Proposition \ref{InitialDataExists} and Proposition \ref{APrioriEnergyBound} in a bootstrapping argument to give the

\begin{proof}[Proof of Theorem \ref{MainTheorem}]

We begin by choosing initial data $A_0 \in H^{s + 13} \cap H^3(\delta)$ of the HNLS equation \eqref{HNLSQuaternion} for some $\delta > 0$, and moreover choose $B(0) = 0$.  By Proposition \ref{HNLSLocallyWellPosed}, there is a time $\mathscr{T} > 0$, a solution $A$ to \eqref{HNLSQuaternion} in $C([0, \mathscr{T}], H^{s + 13} \cap H^3(\delta))$, and a solution $B$ to the equation \eqref{AuxiliaryEquationVector} in $C([0, \mathscr{T}], H^{s + 10} \cap L^2(\delta))$ with $B(0) = 0$.  This guarantees the existence of the formal approximation $\tilde{\zeta}$ constructed in \S 3.

This approximation is used to construct the initial data given in Proposition \ref{InitialDataExists}.  We have by the estimates of that proposition that $\|\nabla(\zeta_0 - P)\|_{L^\infty} \leq C\epsilon$, and so for sufficiently small $\epsilon_0$ the chord arc condition and $|\partial_\alpha \zeta_0 \times \partial_\beta \zeta_0|^{-1} \leq N_0$ hold for some $\nu_0, N_0$.  Thus  we are guaranteed the existence of a solution $\Xi$ to the water wave problem having all of the properties listed in Theorem \ref{LocalWellPosednessWW} on some interval of time $T_0 > 0$.  

We first turn to constructing the change of variables $\kappa$.  We can construct $\kappa$ through \eqref{KappaDefinition}; we know that this $\kappa$ can be written as
$$\kappa(\alpha, \beta, t) - P = \int_0^t b(\kappa(\alpha, \beta, \tau), \tau) \, d\tau$$ where by changing variables in \eqref{BFormulaMultiscale} we know that
$$(I - \oht)(b \circ \kappa) = -[\partial_t, \oht](I + \oht)z\kvec - (I - \oht)[\partial_t, \mathfrak{K}]z\kvec - (I - \oht)\mathfrak{K}z_t\kvec$$  Hence, we can bound $\|\nabla(\kappa - P)\|_{L^\infty}$ by $C\|z_t\|_{W^{1, \infty}}$, and so we can choose $T_0$ sufficiently small so that $\kappa$ is a diffeomorphism.  Then, we can therefore construct $\zeta = \Xi \circ \kappa^{-1}$ on $[0, T_0]$, and all of the formulas and equations in \S2 are now valid there.  

We begin a continuity argument.  Let $T_0^*$ be the largest time for which $\Xi$ exists as in Theorem \ref{LocalWellPosednessWW}, and for which $\kappa$ is a diffeomorphism.  As in in Proposition \ref{APrioriEnergyBound}, let $\mathscr{T}^{\prime\prime}$ denote either $\mathscr{T}^\prime$ in the case where $\iota = 0$ or $\mathscr{T}$ in the case where $\iota > 0$.  If $T_0^* \geq \mathscr{T}^{\prime\prime} \epsilon^{-2}$ then we are done.  If not, then by the estimates on the initial data given in Proposition \ref{InitialDataExists}, the a priori bound of Proposition \ref{APrioriEnergyBound} holds, and so the estimates of \S 4 are valid.  In particular, Proposition \ref{KappaControlledByZeta} now guarantees the control $\|\nabla(\kappa - P)\|_{L^\infty} \leq C\epsilon$, \textit{where $C$ is independent of $T_0^*$}.  Therefore $\Xi = \zeta \circ \kappa$, $\Xi_t = (D_t\zeta) \circ \kappa$ and $\Xi_{tt} = (D_t^2\zeta) \circ \kappa$ must agree with the original Lagrangian quantities by uniqueness, and moreover must satisfy the estimates
\begin{align*}
\|\Xi_t\|_{W^{\lfloor s/2 \rfloor + 3, \infty}} + \|\Xi_{tt}\|_{W^{\lfloor s/2 \rfloor + 3, \infty}} & \leq C(\|D_t\zeta\|_{W^{\lfloor s/2 \rfloor + 3, \infty}} + \|D_t^2\zeta\|_{W^{\lfloor s/2 \rfloor + 3, \infty}}) \\
& \leq C(\|\tilde{D}_t\tilde{\zeta}\|_{W^{\lfloor s/2 \rfloor + 3, \infty}} + \|\tilde{D}_t^2\tilde{\zeta}\|_{W^{\lfloor s/2 \rfloor + 3, \infty}}) + C(E^\frac12 + \epsilon^2) \\
& \leq C\epsilon + C\epsilon^2 \\
& \leq C\epsilon
\end{align*}
as well as
\begin{align*}
\frac{1}{|\Xi_\alpha \times \Xi_\beta|} - 1& \leq \frac{1}{1 - C\|\nabla (\Xi - P)\|_{L^\infty}} - 1 \leq \frac{1}{1 - C\|\nabla (\zeta - \kappa^{-1})\|_{L^\infty}} - 1 \leq C\epsilon
\end{align*}
and
\begin{align*}
\sup_{(\alpha, \beta) \neq (\alpha^\prime, \beta^\prime)} \frac{|\Xi(\alpha, \beta) - \Xi(\alpha^\prime, \beta^\prime)|}{|(\alpha, \beta) - (\alpha^\prime, \beta^\prime)|} - 1 & \leq C\|\nabla (\Xi - P)\|_{L^\infty} \leq C\|\nabla (\zeta - \kappa^{-1})\|_{L^\infty} \leq C\epsilon
\end{align*}
with $C$ independent of $T_0^*$ throughout.  But then by the blow-up alternative of Theorem \ref{LocalWellPosednessWW}, we can continue the solution $\Xi$ to a larger interval $[0, T_1^*]$ with $T_0^* < T_1^*$.  By choosing $T_1^*$ smaller but still strictly greater than $T_0^*$ we can also guarantee that $\kappa$ is a diffeomorphism on $[0, T_1^*]$.  But this contradicts the maximality of $T_0^*$.  Therefore $\Xi$ exists on the time interval $[0, \mathscr{T}^{\prime\prime}\epsilon^{-2}]$ and $\zeta$ also exists on $[0, \mathscr{T}^{\prime\prime}\epsilon^{-2}]$ and satisfies the bounds of Proposition \ref{APrioriEnergyBound} there.
\end{proof}

Finally we give the

\begin{proof}[Proof of Theorem \ref{EulerianTheorem}.]
The conclusion follows immediately from the conclusion of Theorem \ref{MainTheorem} and Corollary \ref{MainTheoremLInftyBounds} along with the definitions of $h, \tau, \tilde{h}, \tilde{\tau}$.  Equally straightforward is showing that the initial data constructed in Proposition \ref{InitialDataExists} also generates appropriate initial data for Theorem \ref{EulerianTheorem}.  We need only show that the initial data $\eta_0, \mathfrak{v}_0$ given by the hypothesis can be used to generate initial data in the sense of Proposition \ref{InitialDataExists}.  

Following the proof of that proposition, we use a contraction mapping argument to construct $\lambda_0 \in H^{s + 1}$ satisfying
$$\lambda_0 = (I + \nht_{\zeta_0} - \mathcal{K}_{\zeta_0})(h_0 \circ \tilde{\tau})\kvec$$
Taking the $\kvec$-component of this relation implies that $\znew_0 = h_0 \circ \tilde{\tau}$, and so by Proposition \ref{ChangeOfVariablesIJValued} we have $\Re(\lambda_0) = 0$.  Since $\tilde{\znew} = \tilde{h} \circ \tilde{\tau}$, we have that
\begin{align*}
\lambda_0 - \tilde{\lambda}(0) & = \left((\nht_{\zeta_0} - \mathcal{K}_{\zeta_0}) - (\tilde{\nht} - \tilde{\mathcal{K}})\right)(h_0 \circ \tilde{\tau}) + (I + \tilde{\nht} - \tilde{\mathcal{K}})U_{\tilde{\tau}}(h_0 - \tilde{h})
\end{align*}
which by hypothesis yields $\|\,|\mathcal{D}|^\frac12(\lambda_0 - \tilde{\lambda}(0))\|_{H^{s + \frac12}} \leq C\epsilon^{2 + \eta}$.  

Similarly, let $\mathcal{N}_0 = \partial_\alpha \zeta_0 \times \partial_\beta \zeta_0$ and $\mathbf{n}_0 = \mathcal{N}_0/|\mathcal{N}_0|$ be the outward unit normal to $\Sigma(t)$; choose $$v_0 = (I + \nht_{\zeta_0})\left(\mathbf{n}_0(I + \{\mathbf{n}_0 - \kvec\}_3 + \{\nht_{\zeta_0}\mathbf{n}_0\}_3)^{-1}(\mathfrak{v}_0 \circ \tilde{\tau})\right)$$  We again have $(I - \nht_{\zeta_0})v_0 = 0$, $\{v_0\}_3 = \mathfrak{v}_0 \circ \tilde{\tau}$, and $\Re(v_0) = 0$ by the definition of the Hilbert transform.  One defines the initial data $w_0, \mathfrak{a}_0$ by compatibility through $\lambda_0$ and $v_0$ as in Proposition \ref{InitialDataExists}.  This allows us to assume local well-posedness of $\Xi$ and hence derive the bound $\|v_0 - \tilde{\lambda}_t\|_{H^{s + \frac12}}$ as in Proposition \ref{InitialDataExists}; the term that was designed to vanish in that estimate is instead bounded by $C\epsilon^{2 + \eta}$ by hypothesis.  One follows the proof of this proposition to similarly bound the initial energy by $\mathfrak{E}(0)^\frac12 \leq C\epsilon^{2 + \eta}$.
\end{proof}

\appendix

\section{Formal Expansion of the Hilbert Transform}

We give a detailed derivation of expressions for the expansion of the Hilbert transform $\nht$ in terms of only the flat Hilbert transform and its commutators with approximate quantities as defined in \S 3.  Our starting point is the power series expansions \eqref{CauchyKernelExpansion}-\eqref{NormalVectorExpansion} allowing us to expand $\nht$ into a series of operators homogeneous in $\lambda$: $$\nht = \sum_{n = 0}^\infty \nht_n$$

\subsection{Contributions from $\nht_1$}

We read off that
\begin{align*}
\nht_1 f & = \frac{1}{2\pi^2} \iint \frac{\lambda - \lambda^\prime}{|P - P^\prime|^3} \kvec f^\prime dP^\prime \\
& \quad \; \frac{1}{2\pi^2} \iint \frac{P - P^\prime}{|P - P^\prime|^3} (\lambda_\alpha^\prime \times \jvec + \ivec \times \lambda_\beta^\prime) f^\prime dP^\prime \\
& \quad \; \frac{1}{2\pi^2} \iint -3\frac{P - P^\prime}{|P - P^\prime|^3}\frac{(P - P^\prime) \cdot (\lambda - \lambda^\prime)}{|P - P^\prime|^2} \kvec f^\prime dP^\prime \\
& := I_1 + I_2 + I_3
\end{align*}
At this point we can write both $I_1$ and $I_2$ in terms of Riesz potentials and Riesz transforms:
\begin{equation*}
I_1 = [\lambda \kvec, -|\mathcal{D}|]f
\end{equation*}
\begin{equation*}
I_2 = \nht_0(-\kvec(\lambda_\alpha \times \jvec + \ivec \times \lambda_\beta)f) = \nht_0\left\{(\xnew_\alpha + \ynew_\beta)f + (\kvec\mathcal{D}\znew)f\right\}
\end{equation*}  We write $I_3$ in the same form by writing it a commutator of $\xnew$ with $\kvec f$ and of $\ynew$ with $\kvec f$ by the convolution operators
\begin{align*}
I_3 & = \frac{1}{2\pi^2} \iint \left(-3\frac{P - P^\prime}{|P - P^\prime|^5}\right)\bigl((\alpha - \alpha^\prime)(\xnew - \xnew^\prime) + (\beta - \beta^\prime)(\ynew - \ynew^\prime)\bigr)\kvec f^\prime dP^\prime \\
& = [\xnew, T_3^\alpha]\kvec f + [\ynew, T_3^\beta]\kvec f,
\end{align*} where we can rewrite these operators $T_3^\alpha$ and $T_3^\beta$ by a computation on the Fourier side, giving
\begin{align*}
T_3^\alpha & = \pv (-3)\frac{1}{2\pi^2}\frac{P\alpha}{|P|^5} \star = \ivec|\mathcal{D}| - \nht_0\kvec\partial_\alpha \\
T_3^\beta & = \pv (-3)\frac{1}{2\pi^2} \frac{P\beta}{|P|^5} \star = \jvec|\mathcal{D}| - \nht_0\kvec\partial_\beta
\end{align*} and so using the formula $|\mathcal{D}| = \nht_0\kvec\mathcal{D}$ we simplify to find that
\begin{align*}
\nht_1 f & = -[\lambda, |\mathcal{D}|]\kvec f \\
& \quad + [\xnew, \ivec|\mathcal{D}| - \nht_0\kvec\partial_\alpha]\kvec f \\
& \quad + [\ynew, \jvec|\mathcal{D}| - \nht_0\kvec\partial_\beta]\kvec f \\
& \quad + \nht_0\left\{(\xnew_\alpha + \ynew_\beta)f + (\kvec\mathcal{D}\znew)f\right\} \\
& = -[\lambda, |\mathcal{D}|]\kvec f + [\xnew \ivec, |\mathcal{D}|]\kvec f] + [\ynew\jvec, |\mathcal{D}|]\kvec f \\
& \quad + \nht_0\left\{(\kvec\mathcal{D}\znew)f\right\} \\
& \quad + [\xnew, \nht_0\partial_\alpha]f + [\ynew, \nht_0\partial_\beta]f + \nht_0\left\{(\xnew_\alpha + \ynew_\beta)f\right\} \\
& = [\znew, |\mathcal{D}|]f + \nht_0\left\{(\kvec\mathcal{D}\znew)f\right\} \\
& \quad + [\xnew, \nht_0]\partial_\alpha f + [\ynew, \nht_0]\partial_\beta f
\end{align*} and so
\begin{align*}
\nht_1 & = [\xnew, \nht_0]\partial_\alpha + [\ynew, \nht_0]\partial_\beta + [\znew, \nht_0]\kvec\mathcal{D} \\
& = [(\xnew + \jvec\znew), \nht_0]\partial_\alpha + 2\znew\kvec\mathcal{R}_2\partial_\alpha \\
& + [(\ynew - \ivec\znew), \nht_0]\partial_\beta - 2\znew\kvec\mathcal{R}_1\partial_\beta \\
& = [(\xnew + \jvec\znew), \nht_0]\partial_\alpha + [(\ynew - \ivec\znew), \nht_0]\partial_\beta
\end{align*}  Introducing the quantities $p_1 = \xnew + \jvec\znew$ and $p_2 = \ynew - \ivec\znew$ along with $\partial_\alpha = \partial_1$ and $\partial_\beta = \partial_2$, we can express this formula compactly as follows:
\begin{equation}
\nht_1 f = \sum_{i = 1}^2 [p_i, \nht_0]\partial_i f
\end{equation}

\subsection{Contributions from $\nht_2$}

We again read off the expansion in terms of Riesz potentials:

\begin{align*}
\nht_2 & = \frac{1}{2\pi^2}\iint\frac{\lambda - \lambda^\prime}{|P - P^\prime|^3}(-3)\frac{(P - P^\prime) \cdot (\lambda - \lambda^\prime)}{|P - P^\prime|^2} \kvec f^\prime dP^\prime \\
& + \frac{1}{2\pi^2}\iint\frac{P - P^\prime}{|P - P^\prime|^3}\left(-\frac{3}{2}\right)\frac{|\lambda - \lambda^\prime|^2}{|P - P^\prime|^2} \kvec f^\prime dP^\prime \\
& + \frac{1}{2\pi^2}\iint\frac{P - P^\prime}{|P - P^\prime|^3}\left(\frac{15}{2}\right)\frac{((P - P^\prime) \cdot (\lambda - \lambda^\prime))^2}{|P - P^\prime|^4} \kvec f^\prime dP^\prime \\
& + \frac{1}{2\pi^2}\iint\frac{\lambda - \lambda^\prime}{|P - P^\prime|^3} (\lambda_\alpha^\prime \times \jvec + \ivec \times \lambda_\beta^\prime) f^\prime dP^\prime \\
& + \frac{1}{2\pi^2}\iint\frac{P - P^\prime}{|P - P^\prime|^3}(-3)\frac{(P - P^\prime) \cdot (\lambda - \lambda^\prime)}{|P - P^\prime|^2} (\lambda_\alpha^\prime \times \jvec + \ivec \times \lambda_\beta^\prime) f^\prime dP^\prime \\
& + \frac{1}{2\pi^2} \iint \frac{P - P^\prime}{|P - P^\prime|^3}(\lambda_\alpha^\prime \times \lambda_\beta^\prime) f^\prime dP^\prime \\
& = I_1 + I_2 + I_3 + I_4 + I_5 - \nht_0(\kvec(\lambda_\alpha \times \lambda_\beta))
\end{align*}Many of the same operators appearing in our calculation of $\nht_1$ appear here.  We have immediately that
\begin{align*}
I_4 & = -[\lambda, |\mathcal{D}|]\left\{(\xnew_\alpha + \ynew_\beta)\kvec f - (\mathcal{D}\znew)f\right\}
\end{align*}  As in the calculation of $\nht_1$, we have
\begin{align*}
I_5 & = \Bigl([\xnew, \ivec|\mathcal{D}| - \nht_0\kvec\partial_\alpha] + [\ynew, \jvec|\mathcal{D}| - \nht_0\kvec\partial_\beta]\Bigr)\left\{(\xnew_\alpha + \ynew_\beta)\kvec f - (\mathcal{D}\znew)f\right\}
\end{align*}  We cast $I_1$ as the sum of two double commutators by the convolution operators
\begin{align*}
T_1^\alpha & = \pv (-3)\frac{1}{2\pi^2}\frac{\alpha}{|P|^5} \star = -\partial_\alpha|\mathcal{D}|\\
T_1^\beta & = \pv (-3)\frac{1}{2\pi^2}\frac{\beta}{|P|^5} \star = -\partial_\beta|\mathcal{D}|
\end{align*} and so
\begin{align*}
I_1 & = -\Bigl[\lambda, [\xnew, \partial_\alpha|\mathcal{D}|] + [\ynew, \partial_\beta|\mathcal{D}|]\Bigr]\kvec f
\end{align*}  Next, writing $|\lambda - \lambda^\prime|^2 = (\xnew - \xnew^\prime)^2 + (\ynew - \ynew^\prime)^2 + (\znew -\znew^\prime)^2$ leads to the following expression of $I_2$ as a sum of double commutators:
\begin{align*}
I_2 & = -\frac{1}{2}\Bigl[\xnew, [\xnew, |\mathcal{D}|\mathcal{D}]\Bigr]\kvec f \\
& - \frac{1}{2}\Bigl[\ynew, [\ynew, |\mathcal{D}|\mathcal{D}]\Bigr]\kvec f \\
& - \frac{1}{2}\Bigl[\znew, [\znew, |\mathcal{D}|\mathcal{D}]\Bigr]\kvec f
\end{align*}  Finally, using the following expression for the convolution operator in the term $I_3$:
\begin{align*}
T_3^{\alpha, \alpha} & = \ivec|\mathcal{D}|\partial_\alpha  + \frac{1}{2}\mathcal{D}|\mathcal{D}| - \frac{1}{2}\nht_0\kvec\partial_\alpha^2 \\
T_3^{\alpha, \beta} & = (\ivec\partial_\beta + \jvec\partial_\alpha)|\mathcal{D}| - \nht_0\kvec\partial_\alpha\partial_\beta\\
T_3^{\beta, \beta} & = \jvec|\mathcal{D}|\partial_\beta  + \frac{1}{2}\mathcal{D}|\mathcal{D}|- \frac{1}{2}\nht_0\kvec\partial_\beta^2
\end{align*} in which we have taken pains to write these operators as differential operators with a Hilbert transform, we have the following expression for $I_3$:
\begin{align*}
I_3 & = \Bigl[\xnew, [\xnew, \ivec|\mathcal{D}|\partial_\alpha  + \frac{1}{2}\mathcal{D}|\mathcal{D}| - \frac{1}{2}\nht_0\kvec\partial_\alpha^2] \Bigr]\kvec f \\
& + \Bigl[\ynew, [\xnew, (\ivec\partial_\beta + \jvec\partial_\alpha)|\mathcal{D}| - \nht_0\kvec\partial_\alpha\partial_\beta] \Bigr]\kvec f \\
& + \Bigl[\ynew, [\ynew, \jvec|\mathcal{D}|\partial_\beta  + \frac{1}{2}\mathcal{D}|\mathcal{D}|- \frac{1}{2}\nht_0\kvec\partial_\beta^2] \Bigr]\kvec f \\
\end{align*}
We now collect like terms from $I_1 + \cdots + I_5$ as follows:
\begin{align*}
& \quad \Bigl[-\lambda + \ivec\xnew + \jvec \ynew, \Bigl([\xnew, |\mathcal{D}|\partial_\alpha] + [\ynew, |\mathcal{D}|\partial_\beta ]\Bigr)\Bigr]\kvec f \\
& + [-\lambda + \xnew\ivec + \ynew\jvec, |\mathcal{D}|]\left\{(x_\alpha + y_\beta)\kvec f - (\mathcal{D}\znew )f\right\} \\
& -\frac{1}{2}[\znew, [\znew, |\mathcal{D}|\mathcal{D}]]\kvec f \\
& - \frac{1}{2}[\xnew, [\xnew, \nht_0 \kvec \partial_\alpha^2]]\kvec f - \frac{1}{2}[\ynew, [\ynew, \nht_0 \kvec \partial_\beta^2]]\kvec f \\
& - \Bigl([\xnew, \nht_0\kvec\partial_\alpha] + [\ynew, \nht_0\kvec\partial_\beta]\Bigr)\left\{(\xnew_\alpha + \ynew_\beta)\kvec f - (\mathcal{D}\znew)f\right\} \\
& - [\ynew, [\xnew, \nht_0\kvec\partial_\alpha \partial_\beta]]\kvec f \\
& \\
& = \Bigl[-\znew \kvec, \Bigl([\xnew, |\mathcal{D}|\partial_\alpha] + [\ynew, |\mathcal{D}|\partial_\beta ]\Bigr)\Bigr]\kvec f \\
& + [-\znew\kvec, |\mathcal{D}|]\left\{(x_\alpha + y_\beta)\kvec f - (\mathcal{D}\znew )f\right\} - \frac{1}{2}[\znew, [\znew, |\mathcal{D}|\mathcal{D}]]\kvec f \\
& - \frac{1}{2}[\xnew, [\xnew, \nht_0 \kvec \partial_\alpha]] \kvec f_\alpha - \frac{1}{2}[\ynew, [\ynew, \nht_0 \kvec \partial_\beta]]\kvec f_\beta \\
& - [\xnew, \nht_0\kvec\partial_\alpha]\left\{\ynew_\beta \kvec f - (\mathcal{D}\znew)f\right\} \\
& - [\ynew, \nht_0\kvec\partial_\beta]\left\{\xnew_\alpha\kvec f - (\mathcal{D}\znew)f\right\} \\
& - [\ynew, [\xnew, \nht_0\kvec\partial_\alpha \partial_\beta]]\kvec f
\end{align*}
In simplifying so as to reduce the degree of the operators in the above commutators, the components of the quantity $\lambda_\alpha \times \lambda_\beta$ occur naturally.  Denote the $i$th component of this quantity by $(\lambda_\alpha \times \lambda_\beta)_i$.  Then we continue to simplify:
\begin{align*}
& = [\znew, [\xnew, |\mathcal{D}|]]f_\alpha - [\xnew, |\mathcal{D}|](\znew_\alpha f) + [\xnew, \nht_0\kvec\partial_\alpha]((\mathcal{D}\znew)f) \\
& + [\znew, [\ynew, |\mathcal{D}|]]f_\beta - [\ynew, |\mathcal{D}|](\znew_\beta f) + [\ynew, \nht_0\kvec\partial_\beta]((\mathcal{D}\znew)f) \\
& + [\xnew, [\ynew, \nht_0]]f_{\alpha \beta} - [\ynew, \nht_0](\xnew_\beta f_\alpha) - [\xnew, \nht_0](\ynew_\alpha f_\beta) - \nht_0\{(\lambda_\alpha \times \lambda_\beta)_3 f\} \\
& - \frac{1}{2}[\xnew, [\xnew, \nht_0 \kvec \partial_\alpha]] \kvec f_\alpha - \frac{1}{2}[\ynew, [\ynew, \nht_0 \kvec \partial_\beta]]\kvec f_\beta  - \frac{1}{2}[\znew, [\znew, |\mathcal{D}|]]\mathcal{D}\kvec f \\
& \\
& = [\znew, [\xnew, \nht_0\kvec]]\mathcal{D}f_\alpha - [\znew, \nht_0\kvec](\mathcal{D}\xnew)f_\alpha - [\xnew, \nht_0\kvec](\znew_\alpha(\mathcal{D}f)) - \nht_0\{\ivec (\lambda_\alpha \times \lambda_\beta)_2 f\} \\
& + [\znew, [\ynew, \nht_0\kvec]]\mathcal{D}f_\beta - [\znew, \nht_0\kvec](\mathcal{D}\ynew)f_\beta - [\ynew, \nht_0\kvec](\znew_\beta(\mathcal{D}f)) + \nht_0\{\jvec (\lambda_\alpha \times \lambda_\beta)_1 f\} \\
& + [\xnew, [\ynew, \nht_0]]f_{\alpha \beta} - [\ynew, \nht_0](\xnew_\beta f_\alpha) - [\xnew, \nht_0](\ynew_\alpha f_\beta) - \nht_0\{(\lambda_\alpha \times \lambda_\beta)_3 f\} \\
& + \frac{1}{2}[\xnew, [\xnew, \nht_0 \partial_\alpha]] f_\alpha + \frac{1}{2}[\ynew, [\ynew, \nht_0 \partial_\beta]] f_\beta  - \frac{1}{2}[\znew, [\znew, |\mathcal{D}|]]\mathcal{D}\kvec f
\end{align*}
By further rewriting the above so that all commutators contain only the Hilbert transform $\nht_0$, we can collect the above terms into a compact formula after introducing some notation.  Denoting $\lambda = \lambda_1\ivec + \lambda_2\jvec + \lambda_3\kvec$ as well as $\partial_1 = \partial_\alpha$, $\partial_2 = \partial_\beta$, $\partial_3 = \kvec\mathcal{D}$, we have the formula
\begin{equation}
\nht_2 f = -\sum_{i, j = 1}^3 [\lambda_i, \nht_0]\Bigl((\partial_i \lambda_j)(\partial_j f)\Bigr) + \frac{1}{2}\sum_{i, j = 1}^3 [\lambda_i, [\lambda_j, \nht_0]]\partial_i\partial_j f
\end{equation}  As we did with the operator $\nht_1$, we rewrite this large sum into a smaller sum involving the quantities $p_1 = \xnew + \jvec\znew$ and $p_2 = \ynew - \ivec\znew$ as follows:
\begin{align*}
\nht_2 f & = \biggl(\sum_{i, j = 1}^2 -[p_i, \nht_0](\partial_i p_j)(\partial_j f) + \sum_{i, j = 1}^2 (-1)^i 2\znew\kvec\mathcal{R}_{3 - i}((\partial_i p_j)(\partial_j f))\biggr) \\
& \quad + \biggl( \sum_{i, j = 1}^2 \frac{1}{2}[p_i, [p_j, \nht_0]]\partial_i \partial_j f + \sum_{i, j = 1}^2 (-1)^{j + 1} 2\znew\lambda_j\kvec\mathcal{R}_{3 - i}\partial_i\partial_j f + \sum_{i, j = 1}^2 (-1)^i 2\znew\kvec\mathcal{R}_{3 - i}(p_j \partial_i \partial_j f) \biggr) \\
& = \sum_{i, j = 1}^2 \biggl( -[p_i, \nht_0](\partial_i p_j)(\partial_j f) + \frac{1}{2}[p_i, [p_j, \nht_0]]\partial_i \partial_j f \biggr) \\
& \quad + \sum_{i, j = 1}^2 \biggl( (-1)^{j + 1} 2\znew\lambda_j\kvec\mathcal{R}_{3 - i}\partial_i\partial_j f + (-1)^i 2\znew\kvec\mathcal{R}_{3 - i}\partial_i(p_j \partial_j f) \biggr)
\end{align*}  The second sum above vanishes because of the identity $\sum_{i = 1}^2 (-1)^i \mathcal{R}_{3 - i}\partial_i = \mathcal{R}_1\partial_\beta - \mathcal{R}_2\partial_\alpha = 0$.  Therefore we have the formula
\begin{equation}
\nht_2 f = \sum_{i, j = 1}^2 \biggl( -[p_i, \nht_0](\partial_i p_j)(\partial_j f) + \frac{1}{2}[p_i, [p_j, \nht_0]]\partial_i \partial_j f \biggr)
\end{equation}  

\subsection{Contributions From $\nht_3$.}

We record the kernel of $\mathcal{H}_3$ here.  For brevity, we have further abbreviated $\Delta f := f - f^\prime$:
\begin{align*}
& \left(-\frac{35}{2} \frac{\Delta P(\Delta P \cdot \Delta \lambda)^3}{|\Delta P|^9} + \frac{15}{2}\frac{\Delta P (\Delta P \cdot \Delta\lambda)|\Delta \lambda|^2}{|\Delta P|^7} + \frac{15}{2}\frac{\Delta \lambda(\Delta P \cdot \Delta \lambda)^2}{|\Delta P|^7} - \frac{3}{2}\frac{\Delta \lambda |\Delta \lambda|^2}{|\Delta P|^5}\right)\kvec \\
& + \left(-3\frac{\Delta \lambda(\Delta P \cdot \Delta \lambda)}{|\Delta P|^5} - \frac{3}{2}\frac{\Delta P |\Delta \lambda|^2}{|\Delta P|^5} + \frac{15}{2}\frac{\Delta P(\Delta P \cdot \Delta \lambda)^2}{|\Delta P|^7} \right)(\lambda_\alpha^\prime \times \jvec + \ivec \times \lambda_\beta^\prime) \\
& + \left(\frac{\Delta \lambda}{|\Delta P|^3} - 3\frac{\Delta P}{|\Delta P|^3}\frac{\Delta P \cdot \Delta \lambda}{|\Delta P|^2}\right)(\lambda_\alpha \times \lambda_\beta)
\end{align*}

It suffices for our purposes to recognize that, as was the case for $\nht_1$ and $\nht_2$ above, $\nht_3$ can be written as above involving commutators of the surface coordinates $\lambda_j$, spatial derivatives $\partial_i$ and the flat Riesz transforms $\mathcal{R}_k$.%

\bibliography{Mybib}{}
\bibliographystyle{plain}

\end{document}